\documentclass[reqno]{amsart}  

\usepackage[english]{babel}
\usepackage[utf8]{inputenc}

\usepackage{amsthm,amsfonts,amssymb,amsmath}
\numberwithin{equation}{section}
\usepackage{dsfont}
\usepackage{multirow}
\usepackage{makecell}
\usepackage{url} 
\usepackage{xcolor}
\usepackage{graphicx}
\usepackage{float}
\usepackage{booktabs}
\usepackage{epstopdf}
\usepackage{psfrag}
\usepackage{mathrsfs}
\usepackage{mathtools}
\usepackage{subcaption}
\usepackage{nicefrac} 
\usepackage{paralist}
\usepackage[inline]{enumitem}

\newtheorem{lemma}{Lemma}[section]
\newtheorem{theorem}[lemma]{Theorem}
\newtheorem{proposition}[lemma]{Proposition}
\newtheorem{corollary}[lemma]{Corollary}

\theoremstyle{remark}
\newtheorem{remark}[lemma]{Remark}

\theoremstyle{definition}
\newtheorem{definition}[lemma]{Definition}
\newtheorem{assumption}[lemma]{Assumption}

\def\bbC{\mathbb{C}}
\def\bbE{\mathbb{E}}
\def\bbN{\mathbb{N}}
\def\bbP{\mathbb{P}}

\def\bbR{\mathbb{R}}
\def\bbZ{\mathbb{Z}}

\newcommand{\cB}{\mathcal{B}}

\newcommand{\cF}{\mathcal{F}}

\newcommand{\cK}{\mathcal{K}}

\newcommand{\cO}{\mathcal{O}}

\newcommand{\cS}{\mathcal{S}}

\newcommand{\cV}{\mathcal{V}}

\newcommand{\eps}{\varepsilon}

\newcommand{\esssup}{\mathop{\operatorname{ess\,sup}}}

\renewcommand{\Re}{\operatorname{Re}}
\renewcommand{\Im}{\operatorname{Im}}

\newcommand{\diff}[1]{\mathop{}\!\mathrm{d}#1} 

\newcommand{\bdg}[1]{B_{#1}}
\newcommand{\gronw}{C_{\kappa,T}}

\newcommand{\ind}[1]{\mathbf{1}_{#1}}
\newcommand{\timemesh}{\mathbb{T}}

\usepackage[pdftex,
	pdftitle={H\"older convergence for Markovian approximations}, 
	pdfauthor={N.~Corneille, 
		K.~Kirchner, 
		and P.~Pezzoli Frigerio},
	bookmarksopen,
	colorlinks,
	linkcolor=blue,
	urlcolor= blue,
	citecolor=blue
]{hyperref}

\begin{document}

\title[H\"older convergence for Markovian approximations]{Convergence 
	in H\"older norms for Markovian approximations of stochastic Volterra equations}


\author[N.~Corneille]{No\'e Corneille}

\address[N.~Corneille]{Delft Institute of Applied Mathematics\\
Delft University of Technology\\
P.O.~Box 5031 \\
2600~GA Delft \\
The Netherlands}

\email{n.corneille@tudelft.nl}


\author[K.~Kirchner]{Kristin Kirchner}

\address[Kristin Kirchner]{Department of Mathematics\\
	KTH Royal Institute of Technology\\
	Lindstedtsv\"agen~25\\
	114~28 Stockholm\\
	Sweden, and\vspace*{-7pt}}

\address[]{Delft Institute of Applied Mathematics\\
	Delft University of Technology\\
	P.O.~Box 5031 \\
	2600~GA Delft \\
	The Netherlands}

\email{krikir@kth.se}


\author[P.~Pezzoli Frigerio]{Pietro Pezzoli Frigerio}

\address[P.~Pezzoli Frigerio]{Delft Institute of Applied Mathematics\\
	Delft University of Technology\\
	P.O.~Box 5031 \\
	2600~GA Delft \\
	The Netherlands}
	
\email{pie.pezzoli@gmail.com} 


\begin{abstract}
	We bound the difference between
	two stochastic Volterra processes 
	with identical Lipschitz coefficients
	but different kernels.
	For non-convolution kernels,
	we establish estimates
	in $C^0([0,T];L^p(\Omega))$, $p\geq 2$,
	and for convolution kernels
	in $L^p(\Omega;L^q(0,T))$, $q \in [1,p]$,
	and
	$C^\beta([0,T];L^p(\Omega))$,
	$L^p(\Omega;C^\beta([0,T]))$, 
	where the range of
	the H\"older exponent $\beta \in (0,1]$
	is the maximal permitted 
	by the regularity of the processes. 
	For the fractional kernel,
	we then construct Markovian approximations
	whose error we show to decay
	as $e^{-a\sqrt{N}}$
	in the aforementioned norms,
	using an $N$-node quadrature
	based on sinc methods.
	Numerical experiments
	for the fractional Brownian motion
	verify our findings.
\end{abstract}




\date{\today}

\maketitle



\section{Introduction}\label{sec:intro} 


\subsection{Background and motivation} 

Stochastic Volterra processes 
arise in various applications 
as a mathematical tool 
to model random phenomena 
with memory.
A general stochastic Volterra equation 
takes, for every $t\geq 0$, the form 
\begin{equation}\label{eq:intr_SVE}
	X_t = X_0 + \int_0^t K_1(t,s) b(X_s) \diff{s} + \int_0^t K_2(t,s) \sigma(X_s)\diff{W}_s, 
	\qquad 
	\bbP\text{-a.s.}, 
\end{equation}
where $K_1$ and $K_2$ are 
deterministic \emph{Volterra kernels},
and $W$ is a Wiener process on a suitable space. 
Such processes play an important role for applications, 
especially for rough stochastic volatility models in finance, 
including the rough 
Bergomi~\cite{bayerFrizGatheral2016, 
	bergomiSmileDynamics2, 
	bergomiSmileDynamics3} 
and rough 
Heston~\cite{elEuchRosenbaum2019, 
	heston1993} 
models,
where the underlying volatility process is modeled 
by means of a stochastic Volterra process    
to capture the roughness of empirical 
market data~\cite{gatheraletal2018}. 
In the aforementioned applications, the models feature kernels 
of \emph{convolution type},
i.e., Volterra kernels of the form 
$K(t,s) = k(t-s) \ind{\{t > s\}}(t,s)$.
A well-known example is
the Riemann--Liouville fractional Brownian 
motion~\cite{levy1953,
	mandelbrotVanNess1968},   
\begin{equation}\label{eq:intr_RLfBM}
	B^H_t 
	= 
	\int_0^t k^{H\!}(t - s) \diff{W}_s,
	\qquad
	k^{H\!}(u) 
	:= 
	\frac{1}{\Gamma( H+\nicefrac{1}{2} )} \, u^{H-\frac{1}{2}}, 
	\quad u \in (0,\infty), 
\end{equation} 
where $H\in\bigl( -\tfrac{1}{2}, \tfrac{1}{2}\bigr)$ is the Hurst parameter 
which controls the smoothness of~$B^{H\!}$.  

In the general formulation \eqref{eq:intr_SVE},
the kernels $K_1$ and $K_2$ endow the process $X$
with a \emph{non-local} structure.
As opposed to ordinary stochastic differential equations (SDEs),
where the Markov property of solutions allows
the future evolution to be determined
from the current state alone,
the increments of~$X$ depend on
its entire historical trajectory. 
It is this property 
that renders classical approaches 
for approximating solutions 
to stochastic Volterra equations 
computationally expensive. 


\subsection{Previous work} 

Existence, uniqueness, pathwise 
continuity, and pathwise 
H\"older regularity of (strong) solutions to \eqref{eq:intr_SVE}  
in Banach spaces 
were studied by Zhang~\cite{zhang_SPDE_banach2010} 
for Lipschitz continuous coefficients $b$ and $\sigma$. 
For a more general class of coefficients, 
which are not necessarily Lipschitz continuous, 
weak solutions have been introduced and 
investigated in~\cite{abijaber2025weaksolutions, 
	abijabercuchierolarssonpulido2021, 
	abijaberlarsson2019}. 
Further analyses of stochastic Volterra processes 
	with singular kernels 
	were conducted in~\cite{CochranEtAl1995, 
		coutinDecreusefond2001, 
		decreusefond2002, 
		Wang2008}. 

Regarding numerical methods to approximate stochastic Volterra processes, 
Euler and Milstein type discretizations for \eqref{eq:intr_SVE} 
have been proposed, 
and strong error bounds in the 
norms of $C^0([0,T];L^p(\Omega))$ 
and $L^p(\Omega;C^0([0,T]))$ 
have been derived for the case that  
the coefficients $b,\sigma$ are globally Lipschitz continuous, 
see~\cite{richard2021timestep, zhangEuler2008}. 
The non-locality of \eqref{eq:intr_SVE} 
causes the computational cost of the 
time-stepping schemes in these works  
to scale quadratically with respect to 
the number of time grid points. 

More recently, 
the idea to approximate 
the non-Markovian solution of \eqref{eq:intr_SVE} 
by a solution of a higher-dimensional 
system of ordinary SDEs  
has evolved, 
see~\cite{abijabereleuch2018, 
	alfonsikebaier2022, 
	bayerbreneis2023, 
	harms2021}. 
This approach, 
known as \emph{multifactor} 
or \emph{Markovian approximations}, 
has the advantage that fast convergence 
of such approximations to the  
stochastic solution process of \eqref{eq:intr_SVE} 
can be combined with efficient computational  
methods for ordinary SDEs 
to reduce the computational cost considerably. 
More specifically, 
every convolution kernel of \emph{completely monotone type}, 
such as the fractional kernel in \eqref{eq:intr_RLfBM}, 
admits a representation as  
a Laplace transform of some measure~$\mu$ on $(0,\infty)$, 
i.e., 
\begin{equation}\label{eq:compl_mon_type_ker}
	\forall t > s:
	\qquad
	K(t,s) 
	= 
	k(t-s) 
	=
	\int_0^\infty e^{-\rho(t-s)} \diff{\mu(\rho)}, 
\end{equation}
where 
$\diff{\mu(\rho)} = 
(\Gamma( H+\nicefrac{1}{2} ) 
\Gamma( \nicefrac{1}{2}-H ) )^{-1}  
\rho^{-\frac{1}{2} - H} \diff{\rho}$ 
for the fractional kernel in \eqref{eq:intr_RLfBM}.  
Affine Volterra processes with such kernels
can thus be represented by applying 
an integral w.r.t.\ $\mu$
to an infinite-dimensional Markov process, 
see~\cite{carmona1998, 
	HarmsStefanovits2019}.
Replacing the measure~$\mu$ 
by a linear combination of $N$ Dirac measures
then corresponds to approximating  
the kernel $k$ in \eqref{eq:compl_mon_type_ker} 
by another convolution kernel $\widehat{k}$  
and the original stochastic Volterra equation  
by a finite-dimensional system
of ordinary SDEs~\cite{alfonsikebaier2022}.  
For the so-constructed Markovian approximations 
of Volterra processes, 
Alfonsi and Kebaier~\cite{alfonsikebaier2022}
bounded the pointwise strong error  
in the norm of $C^0([0,T];L^2(\Omega))$  
by the $L^2(0,T)$-difference between 
the kernels $k$ and $\widehat{k}$. 
For specific choices of the Dirac points 
and corresponding weights, 
they proved convergence rates 
of polynomial order $N^{-r}\!$, 
where the rate scales as $r\sim H$ 
for the fractional kernel in \eqref{eq:intr_RLfBM}. 

Considering the fractional kernel $k^H$ 
in  \eqref{eq:intr_RLfBM}, 
by means of geometric spacing and 
Gaussian quadrature 
of level $m$, Harms~\cite{harms2021} 
achieved the rate $r\sim Hm$ 
for approximating the Riemann--Liouville 
fractional Brownian motion 
in the pathwise topology 
of $L^p(\Omega;C^0([0,T]))$.  
Bayer and Breneis \cite{bayerbreneis2023} 
combined the ideas 
of~\cite{alfonsikebaier2022} 
and~\cite{harms2021}, and 
proposed Gaussian quadrature rules 
that achieve superpolynomial 
convergence  
of the order $e^{-a_H \sqrt{N}}$ 
to approximate~$k^H$ 
in the norms of $L^1(0,T)$ and $L^2(0,T)$, 
with $a_H > 0$ depending on 
the Hurst parameter~$H$. 
Together with \cite{alfonsikebaier2022}, 
the latter error bound implies 
strong convergence 
for approximating stochastic Volterra processes 
with the fractional kernel
in the norm of $C^0([0,T];L^2(\Omega))$, 
whereas the former 
may\linebreak   
be used to bound 
weak errors 
for the rough Heston model~\cite{bayer2023weakma_arxiv}. 
Further results on 
weak error estimates for Markovian approximations 
can be found 
in \cite{abijabereleuch2018, 
	bossy2025weakroughkernelcomparison, 
	elEuchRosenbaum2019}. 

To the best of the authors' knowledge, 
up until now there is no analysis 
available that investigates the pathwise error 
of Markovian approximations 
in H\"older norms on $C^\nu([0,T])$ 
for $\nu\in(0,1]$ --- 
an aspect that will be addressed 
in the present work.  
Furthermore, the Gaussian quadrature methods 
of \cite{bayerbreneis2023, harms2021} 
require to recompute the  
nodes and weights 
when increasing $N$.  
This causes additional computational cost  
for combining the approaches, e.g., 
with multilevel Monte Carlo techniques. 


\subsection{Contributions}

The aim of this paper is two-fold: 
Firstly, we bound the difference 
of two stochastic Volterra processes~$X,\widehat{X}$ 
with the same Lipschitz continuous coefficients $b,\sigma$, 
but with different kernels $(K_1,K_2)$ and 
$(\widehat{K}_1, \widehat{K}_2)$,  
see \eqref{eq:intr_SVE}, 
in various norms  
by deterministic quantities depending 
on the differences of the kernels.  
More specifically, 
Theorem~{\upshape\ref{thm:C0Lp-err_non-conv}}  
addresses the pointwise 
strong error in the norm 
of $C^0([0,T];L^p(\Omega))$ 
and constitutes an extension of 
\cite[Theorem~3.1]{alfonsikebaier2022} 
to $p\in[2,\infty)$ and to 
\emph{non-convolution kernels} $(K_1, K_2)$
satisfying mild integrability conditions.  
Additionally, for the case 
of convolution kernels, 
we provide results about 
the differences of their paths 
by bounding the 
error in the norms of 
$L^p(\Omega;L^q(0,T))$, 
for $p\in[2,\infty)$, $q\in[1,p]$, 
and 
of $L^p(\Omega;C^\nu([0,T]))$, 
for $p\in[2,\infty)$ and $\nu\in[0,1]$ 
up to the maximal H\"older regularity 
permitted by~$X$ and~$\widehat{X}$, 
in Corollary~{\upshape\ref{cor:path_dependent_Lp}} 
and Theorem~{\upshape\ref{thm:SVE_holder_Lp_bound}}, 
respectively. 
To the best of the authors' 
knowledge, the latter is the first 
result in the literature 
quantifying the difference of 
stochastic Volterra processes 
with different kernels 
in H\"older norms. 
To obtain this result, 
we first derive an 
$L^p(\Omega)$-estimate for increments 
of the difference process $X - \widehat{X}$ 
in Proposition~{\upshape\ref{prop:Lp_error_of_increment}}. 

Secondly, we introduce and analyze 
a specific class of Markovian approximations 
for stochastic Volterra equations \eqref{eq:intr_SVE} 
when $(K_1, K_2)$ are fractional kernels \eqref{eq:intr_RLfBM} 
with Hurst parameters 
$H_1\in\bigl(-\tfrac{1}{2},\tfrac{1}{2}\bigr)$, 
$H_2\in\bigl(0,\tfrac{1}{2}\bigr)$. 
This class is based on 
a quadrature  
for \eqref{eq:compl_mon_type_ker} using 
\emph{sinc approximations}~\cite{lund1992sinc, 
	stenger1993sinc}, 
whose nodes and weights 
are easy to compute and 
facilitate a straightforward combination with 
multilevel Monte Carlo methods. 
After bounding the sinc quadrature error  
for the fractional kernel  
in Theorem~{\upshape\ref{thm:error_bound_Lr_general_gamma}}, 
we prove that the so-constructed 
Markovian approximations 
converge  
in the norm of $C^\beta([0,T];L^p(\Omega))$, 
for all $p\in[2,\infty)$ and 
every $\beta  \in [ 0, (H_1 + \nicefrac{1}{2} ) \wedge  H_2 )$, 
at a rate that is exponential 
in the square root of the 
number of quadrature nodes 
in Theorems~{\upshape\ref{thm:C0_error_fracker}} 
and~{\upshape\ref{thm:holder_bound_fracker}}.
Using these results we derive pathwise 
exponential convergence 
in $L^p(\Omega;L^q(0,T))$ 
and $L^p(\Omega;C^\beta([0,T]))$,  
see also 
Theorems~{\upshape\ref{thm:C0_error_fracker}} 
and~{\upshape\ref{thm:holder_bound_fracker}}. 
We note that the range of 
the H\"older exponent $\beta$, essentially 
up to $(H_1 + \nicefrac{1}{2} ) \wedge  H_2$, 
is the maximal permitted by the regularity 
of the exact stochastic Volterra process.  


\subsection{Outline}

In Section~{\upshape\ref{sec:prelim}}, 
after introducing the necessary 
notation and function spaces, 
we summarize 
the preliminaries 
including the 
standing assumptions on 
the coefficients 
and the Volterra kernels. 
We then consider the general stochastic 
Volterra equation \eqref{eq:intr_SVE} 
in Section~{\upshape\ref{sec:general_SVE}} 
and address  
pathwise continuity and 
H\"older regularity of its solution 
in Subsection~{\upshape\ref{subsec:general_SVE_regularity}}, 
as well as 
bounding the difference to another 
Volterra process 
with different kernels, 
but with the same coefficients $b,\sigma$, 
in the norms of 
$C^0([0,T];L^p(\Omega))$ 
and $C^\beta([0,T];L^p(\Omega))$ 
in Subsections~{\upshape\ref{subsec:general_SVE_diff}} 
and~{\upshape\ref{subsec:general_SVE_holder_diff}}, respectively. 
Section~{\upshape\ref{sec:sinc-fractional}} is devoted 
to constructing quadrature methods  
for the Laplace transform \eqref{eq:compl_mon_type_ker} 
of the fractional kernel $k^{H\!}$, 
and estimating the corresponding error 
pointwise, in $L^r(0,T)$, and in Sobolev 
sense. 
In Section~{\upshape\ref{sec:markov_approx}} 
we combine~the~results of 
Sections~{\upshape\ref{sec:general_SVE}} and~{\upshape\ref{sec:sinc-fractional}}
to construct Markovian approximations
for stochastic Volterra equations \eqref{eq:intr_SVE}
with two possibly distinct fractional kernels, 
and we prove strong convergence 
in various norms, including 
pathwise H\"older convergence. 
Finally, 
Section~{\upshape\ref{sec:numerical_experiments}} aims 
to empirically verify 
the convergence of the sinc quadrature approximations  
for the fractional kernel established in Section~{\upshape\ref{sec:sinc-fractional}} 
as well as the strong convergence of 
the corresponding Markovian approximations  
derived in Section~{\upshape\ref{sec:markov_approx}}.


\section{Preliminaries}\label{sec:prelim} 


\subsection{Notation}\label{subsec:notation} 

The sets 
$\bbN := \{1,2,3,\ldots\}$ and 
$\bbN_0 := \bbN \cup \{0\}$ 
denote the positive
and non-negative integers, respectively. 
We write $s\wedge t$ (or $s\vee t$) 
for the minimum
(or maximum) of two real numbers $s,t\in\bbR$. 
The real and imaginary parts of a
complex number $z \in \bbC$ 
are denoted by $\Re z$ and $\Im z$, 
respectively. 

For a measure space $(S,\mathcal{S},\mu)$, 
a Banach space $(E,\| \,\cdot\, \|_E)$, 
and ${p \in [1,\infty)}$, 
we define 
the Lebesgue--Bochner space 
$L^p(S;E) = L^p(S,\cS,\mu;E)$  
as the Banach space of (all equivalence classes
of) strongly $\mu$-measurable functions $f\colon S \to E$, 
for which 
\begin{equation*}
	\| f \|_{L^p(S;E)}
	:= 
	\biggl( \int_S \| f(s) \|_E^p \diff{\mu}(s) \biggr)^{\nicefrac{1}{p}}
	< \infty.
\end{equation*}  
The indicator function of a measurable subset $A\in\cS$ 
of $S$ is denoted by $\mathbf 1_A$. 

Whenever $S$ is an interval 
$S = J \subseteq \bbR$, 
we consider the Borel $\sigma$-algebra and the
Lebesgue measure on $J$, 
and we define  
$\triangle J := \{(t,s) \in J\times J : s \leq t\}$. 
We denote 
the space of continuous functions from $J$ to $E$ 
by $C^0(J;E)$. 
If $J$ is compact, 
this space is endowed with the supremum norm
denoted by $\| f \|_{C^0(J;E)} := \sup_{t \in J} \| f(t) \|_E$.
For $\beta \in (0,1]$ and a compact interval~$J$,
we define the space $C^\beta(J;E)$
as the Banach space of all 
$\beta$-H\"older continuous functions $f\colon J \to E$, 
with the norm 
\begin{equation*}
	\| f \|_{C^\beta(J;E)} := \| f \|_{C^0(J;E)} + | f |_{C^\beta(J;E)},
	\quad
	| f |_{C^\beta(J;E)} := \sup_{t,s \in J, \, s \neq t} \frac{\| f(t)-f(s) \|_E}{| t-s |^\beta}.
\end{equation*}

If the codomain 
is a Euclidean space,  
$E=\bbR^m$ or $E=\bbR^{m\times d\!}$, 
equipped with the Euclidean norm, 
respectively, the Frobenius norm, 
and the dimensions $m, d\in \bbN$ are evident from context, 
we abbreviate  
$L^p(S) := L^p(S;E)$ 
and 
$C^\beta(J) := C^\beta(J;E)$. 

Throughout this paper,
we work on a complete probability space
$(\Omega, \mathcal{F}, \bbP)$
with a right-continuous filtration 
$(\mathcal{F}_t)_{t \geq 0}$, 
where $\cF_0$ contains all $\bbP$-null sets.  
The expecta-\linebreak 
tion operator 
on $(\Omega, \mathcal{F}, \bbP)$ is 
denoted by 
$\bbE [Z] := \int_\Omega Z(\omega) \diff{\bbP}(\omega)$.
We mark equations which hold almost everywhere or 
$\bbP$-almost surely with a.e.\ 
and $\bbP$-a.s., respectively.


\subsection{Definitions and setting}\label{subsec:setting}

\begin{definition}\label{def:SVE_definition} 
	Assume that $K_1, K_2\colon[0,\infty)\times [0,\infty) \to \bbR$
	are Volterra kernels (see Definition~{\upshape\ref{def:volterra_kernel}}).
	Let $m,d \in\bbN$, and suppose that   
	$b\colon\bbR^m \to \bbR^m\!$, 
	$\sigma\colon\bbR^m \to \bbR^{m\times d}$
	are Borel measurable mappings,
	$X_0\colon \Omega \to \bbR^m$ is
	an $\cF_0$-measurable random variable, 
	and $W\colon [0,\infty) \times \Omega \to \bbR^d$ is
	a $d$-dimensional standard Brownian motion.  
	
	A stochastic process $X\colon [0,\infty) \times \Omega \to \bbR^m$
	is called a 
	\emph{continuous 
		stochastic Volterra process
		with initial value $X_0$,
		coefficients $b, \sigma$,
		and kernels $K_1, K_2$},
	if $X$ is $(\cF_t)_{t\geq 0}$-adapted, 
	$X$ has continuous sample paths and, 
	for all $t \in [0,\infty)$, 
	\begin{equation}\label{eq:SVE_def_eq}
		X_t = X_0 + \int_0^t K_1(t,s) b(X_s) \diff{s} + \int_0^t K_2(t,s) \sigma(X_s)\diff{W_s},
		\qquad
		\bbP\text{-a.s.}
	\end{equation} 
\end{definition}

\begin{remark}\label{rmk:conv_kernels}
	We say that the stochastic Volterra process~$X$ 
	in \eqref{eq:SVE_def_eq} has \emph{convolution kernels} 
	if there exist Borel measurable functions $k_1,k_2\colon (0,\infty) \to \bbR$ such that
	\begin{equation}\label{eq:induced_conv_volterra_kernel} 
		\forall j \in \{1,2\}: 
		\quad\; 
		K_j(t,s) =
		\begin{cases}
			k_j(t-s) & \text{if }t > s,   \\
			0        & \text{if }t \leq s,
		\end{cases}
		\quad\;     
		(t,s)\in 
		[0,\infty) \times [0,\infty).  
	\end{equation}
	In this case, 
	the stochastic Volterra equation 
	\eqref{eq:SVE_def_eq} simplifies, 
	for all $t\in[0,\infty)$, to 
	\begin{equation}\label{eq:SVE_def_eq_conv}
		X_t = X_0 + \int_0^t k_1(t-s) b(X_s) \diff{s} + \int_0^t k_2(t-s) \sigma(X_s) \diff{W_s},
		\qquad \bbP\text{-a.s.}
	\end{equation} 
\end{remark} 

In what follows, we summarize our assumptions 
on the coefficients $b,\sigma$ 
as well as on the (non-convolution) 
kernels $K_1, K_2$ in \eqref{eq:SVE_def_eq} 
and the convolution kernels $k_1,k_2$ 
in \eqref{eq:SVE_def_eq_conv}, 
respectively.  

\begin{assumption}\label{ass:lipschitz_coef}
	The functions $b\colon \bbR^m \to \bbR^m$
	and $\sigma \colon \bbR^m \to \bbR^{m\times d}$ are Borel measurable
	and there exists a constant $L \in (0,\infty)$ such that, for all $x,y \in \bbR^m\!$,
	\begin{equation*}
		\| b(x) - b(y) \|_{\bbR^m} \leq L \| x-y \|_{\bbR^m}, 
		\qquad 
		\| \sigma(x) - \sigma(y) \|_{\bbR^{m\times d}} \leq L \| x-y \|_{\bbR^m}.
	\end{equation*}
\end{assumption}

\begin{assumption}[Non-convolution kernels]
	\label{ass:assump_nonconv_kernels}
	Let $K_1, K_2 \colon [0,\infty) \times [0,\infty) \to \bbR$ 
	be Volterra kernels on $[0,\infty)$, 
	see Definition~{\upshape\ref{def:volterra_kernel}}.
	For every $T \in (0,\infty)$, there exist constants 
	$\eta = \eta_T \in (0,\infty)$ and $\alpha = \alpha_T \in (0,1]$ such that 
	the following hold: 
	\begin{enumerate}[leftmargin=8mm, label={\upshape(\roman*)}]
		\item\label{item:assu_noco_i} 
		For all $(t,s)\in\triangle[0,T]$,  
		\begin{equation*}
			\| K_1(t,\,\cdot\,) \|_{L^1(s,t)} \leq \eta (t-s)^{\alpha},
			\qquad
			\| K_2(t,\,\cdot\,) \|_{L^2(s,t)} \leq \eta (t-s)^{\alpha} . 
		\end{equation*}
		\item\label{item:assu_noco_ii} 
		For all $(t,s)\in\triangle[0,T]$,  
		\begin{equation*}
			\| K_1(t,\,\cdot\,) - K_1(s,\,\cdot\,) \|_{L^1(0,s)} \leq \eta (t-s)^\alpha ,
			\quad 
			\| K_2(t,\,\cdot\,) - K_2(s,\,\cdot\,) \|_{L^2(0,s)} \leq \eta (t-s)^\alpha . 
		\end{equation*}
	\end{enumerate}
\end{assumption}

\begin{remark} 
	In Assumption~{\upshape\ref{ass:assump_nonconv_kernels}}  
	and below, 
	we work with the continuous 
	representatives  
	of the Volterra kernels $K_1, K_2$ 
	in $C^0([0,T];L^1(0,T))$ 
	and $C^0([0,T];L^2(0,T))$, 
	respectively, 
	which exist 
	by Lemma~{\upshape\ref{lma:continuous_representative}}. 
	Moreover,  
	by Lemma~{\upshape\ref{lma:u_continuity_lemma}} 
	the left-hand sides 
	of the conditions~{\upshape\ref{item:assu_noco_i}} 
	and~{\upshape\ref{item:assu_noco_ii}} 
	are continuous on $\triangle [0,T]$ 
	and, thus, well-defined.   
\end{remark}  
		 
For $j \in \{1,2\}$ 
and $(t,s)\in\triangle[0,\infty)$, 
we introduce the following abbreviation, 
\begin{equation}\label{eq:K_jts_definition}
	K_{j;t,s}\colon [0,s] \to \bbR, 
	\qquad
	K_{j;t,s} := K_j(t,\,\cdot\,) - K_j(s,\,\cdot\,).
\end{equation}

\begin{assumption}[Convolution kernels]
	\label{ass:assump_conv_kernels}
	Let $k_1, k_2 \colon (0,\infty) \to \bbR$ be Borel measurable functions 
	which are non-negative and non-increasing 
	a.e.\ on $(0,\infty)$, i.e., 
	there exists a Lebesgue 
	null set $A\subset (0,\infty)$ 
	such that  
	$k_j(t) \geq 0$ and 
	$k_j(t) \leq k_j(s)$ 
	holds for $j\in\{1,2\}$ and 
	all $t,s \in (0,\infty)\setminus A$ 
	with $s<t$.   
	Furthermore, for every  
	$T \in (0,\infty)$, there exist 
	constants 
	$\eta = \eta_T \in (0,\infty)$ and 
	$\alpha = \alpha_T \in (0,1]$ such that 
	\begin{equation*}
		\forall t\in [0,T]: 
		\quad 
		\| k_1 \|_{L^1(0,t)} \leq \eta \, t^{\alpha}, 
		\quad\;\; 
		\| k_2 \|_{L^2(0,t)} \leq \eta \, t^{\alpha}. 
	\end{equation*}
\end{assumption} 

\begin{remark}\label{rmk:assumptions_implication_structure} 
	If $k_1,k_2$ satisfy 
	Assumption~{\upshape\ref{ass:assump_conv_kernels}} 
	with the constants 
	$(\eta_T)_{T>0}\subseteq(0,\infty)$ 
	and 
	$(\alpha_T)_{T>0}\subseteq (0,1]$, 
	then a straightforward calculation 
	shows that 
	the Volterra kernels $K_1$ and $K_2$
	defined according to  \eqref{eq:induced_conv_volterra_kernel} 
	fulfill  Assumption~{\upshape\ref{ass:assump_nonconv_kernels}} 
	with the same families of 
	constants $(\eta_T)_{T>0}\subseteq(0,\infty)$ 
	and 
	$(\alpha_T)_{T>0}\subseteq (0,1]$,  
	cf.~Lemma~{\upshape\ref{lma:kernel_assumption_implications}}. 
\end{remark}


\section{Difference in H\"older norms of stochastic Volterra processes}
\label{sec:general_SVE}


\subsection{Existence, uniqueness, and regularity}
\label{subsec:general_SVE_regularity}

Existence and uniqueness of 
continuous stochastic Volterra processes
satisfying \eqref{eq:SVE_def_eq}
is a well-known result, see e.g.\ 
\cite[Theorems~3.1 \&~3.3]{zhang_SPDE_banach2010}.
This subsection serves as a means to connect our framework 
to that of \cite{zhang_SPDE_banach2010} 
and to establish a relation between  
H\"older regularity 
of Volterra processes on $[0,T]$ 
and the constant $\alpha_T$ 
in Assumptions~{\upshape\ref{ass:assump_nonconv_kernels}} 
and~{\upshape\ref{ass:assump_conv_kernels}}, 
respectively. 

\begin{proposition}\label{prop:existence_uniqueness_SVE}
	Let $p \in [2,\infty)$ and  $X_0 \in L^p(\Omega,\cF_0,\bbP)$.
	Suppose that ${b\colon \bbR^m\! \to \bbR^m}$
	and $\sigma\colon \bbR^m \to \bbR^{m\times d}$
	fulfill Assumption~{\upshape\ref{ass:lipschitz_coef}} 
	with the constant $L \in (0,\infty)$. 
	Let $K_1, K_2$ be Volterra kernels
	satisfying Assumption~{\upshape\ref{ass:assump_nonconv_kernels}\ref{item:assu_noco_i}}, 
	with the constants $\eta_T \in (0,\infty)$ 
	and $\alpha_T \in (0,1]$, 
	for every $T \in (0,\infty)$. 
	Then, there exists a unique measurable 
	$(\cF_t)_{t\geq 0}$-adapted stochastic process
	$Y\colon [0,\infty) \times \Omega \to \bbR^m$ 
	which satisfies \eqref{eq:SVE_def_eq} 
	for almost all $t\in[0,\infty)$. 
	Furthermore, 
	for every $T \in (0,\infty)$, 
	we have that 
	\begin{equation*} 
		\esssup_{t \in [0,T]} 
		\| Y_t \|_{L^p(\Omega)}^p 
		=
		\esssup_{t \in [0,T]} 
		\bbE\bigl[ \| Y_t \|_{\bbR^m}^p \bigr]  
		< \infty . 
	\end{equation*} 
\end{proposition}

\begin{proof}
	The assertion follows from 
	\cite[Theorem~3.1]{zhang_SPDE_banach2010}, 
	provided that the conditions 
	(H1)--(H3) 
	of \cite[pp.~1375--1376]{zhang_SPDE_banach2010} 
	are satisfied, which we verify below. 
	
	By Assumption~{\upshape\ref{ass:lipschitz_coef}}, 
	the following holds: For all $(t,s) \in \triangle [0,\infty)$
	and $x \in \bbR^m$,
	\begin{align*}
		\| K_1(t,s)b(x) \|_{\bbR^m}
		&\leq 
		| K_1(t,s) | 
		\left(\| b(0) \|_{\bbR^m} + L \| x \|_{\bbR^m} \right)\!,
		\\
		\| K_2(t,s)\sigma(x) \|_{\bbR^{m\times d}}^2
		&\leq 
		2 \, | K_2(t,s) |^2 
		\bigl( \| \sigma(0) \|_{\bbR^{m\times d}}^2 + L^2 \| x \|_{\bbR^m}^2 \bigr).
	\end{align*}
	Moreover, for all $(t,s) \in \triangle [0,\infty)$ 
	and $x,y \in \bbR^m$,
	we have that
	\begin{align*}
		\| K_1(t,s)(b(x) - b(y)) \|_{\bbR^m}
		&\leq 
		L \, | K_1(t,s) | \| x-y \|_{\bbR^m},       
		\\
		\| K_2(t,s)(\sigma(x)-\sigma(y)) \|_{\bbR^{m\times d}}^2
		&\leq 
		L^2 | K_2(t,s) |^2 \| x-y \|_{\bbR^m}^2.
	\end{align*}
	Thus, 
	if we define $\kappa(t,s) := | K_1(t,s) | + | K_2(t,s) |^2$, 
	the conditions (H2) and (H3) of \cite{zhang_SPDE_banach2010}
	are satisfied provided that $\kappa \in \cK_0$,
	which is defined on \cite[pp.~1365--1366]{zhang_SPDE_banach2010}.
	Assumption~{\upshape\ref{ass:assump_nonconv_kernels}\ref{item:assu_noco_i}}  
	implies that, for every $T\in(0,\infty)$ and  
	all $(t,s) \in\triangle[0,T]$, 
	\begin{equation*} 
		\int_s^t \kappa(t,u) \diff{u} \le \eta (t-s)^\alpha + \eta^2 (t-s)^{2\alpha}, 
	\end{equation*}
	where $\eta := \eta_T$ and 
	$\alpha := \alpha_T$. 
	It follows that, for all $T \in (0,\infty)$, 
	\begin{equation*} 
		\esssup_{t \in [0,T]} \int_0^t \kappa(t,u) \diff{u} < \infty,
		\qquad 
		\limsup_{\eps \downarrow 0} \esssup_{t \in [0,T]} 
		\int_t^{t+\eps} \kappa(t+\eps,u) \diff{u} 
		= 0. 
	\end{equation*} 
	This shows that~$\kappa$ indeed 
	is an element of~$\cK_0$.
	Finally, since $X_0\in L^p(\Omega,\cF_0,\bbP)$ 
	does not depend on $t\in[0,\infty)$ 
	and $\kappa\in\cK_0$, 
	also condition~(H1) is fulfilled. 
\end{proof}

\begin{remark}\label{rmk:def_bT_sigmaT}
	Combining the stability bound 
	of Proposition~{\upshape\ref{prop:existence_uniqueness_SVE}} 
	with Assumption~{\upshape\ref{ass:lipschitz_coef}} shows that, 
	for all $T \in (0,\infty)$,  
	\begin{equation*}
		\esssup_{t \in [0,T]} \| b(Y_t) \|_{L^p(\Omega)}
		\leq 
		\overline{b}_T < \infty,
		\qquad
		\esssup_{t \in [0,T]} \| \sigma(Y_t) \|_{L^p(\Omega)}
		\leq 
		\overline{\sigma}_T  < \infty,
	\end{equation*}  
	where we define
	\begin{equation}\label{eq:bbar_sigmabar}  
		\begin{split}  
			\overline{b}_T 
			&:= \| b(0) \|_{\bbR^m} 
			+ L \esssup\nolimits_{t \in [0,T]} \| Y_t \|_{L^p(\Omega)},  
			\\ 
			\overline{\sigma}_T 
			&:= 
			\| \sigma(0) \|_{\bbR^{m\times d}} 
			+ L \esssup\nolimits_{t \in [0,T]} \| Y_t \|_{L^p(\Omega)}.  
		\end{split} 
	\end{equation}
\end{remark}

Zhang \cite[Theorem~3.3]{zhang_SPDE_banach2010} 
derives pathwise H\"older continuity  
of stochastic Volterra processes  
for \emph{some H\"older exponent $\nu \in (0,1]$}.
In~\cite[Lemma~2.4]{abijaberlarsson2019},
Abi Jaber et al.\ obtain an explicit range for $\nu$
for the case of convolution kernels 
which satisfy conditions similar to 
Assumption~{\upshape\ref{ass:assump_nonconv_kernels}}.
In the next proposition,
we prove 
existence, uniqueness 
and explicit H\"older regularity 
of continuous 
stochastic Volterra processes  
for the case that 
the kernels $K_1, K_2$ in \eqref{eq:SVE_def_eq}
are general (non-convolution) 
Volterra kernels,  
which satisfy 
the conditions of 
Assumption~{\upshape\ref{ass:assump_nonconv_kernels}}.

\begin{proposition}\label{prop:X_ct_holder_Lp}
	Let $p \in [2,\infty)$ and  $X_0 \in L^p(\Omega,\cF_0,\bbP)$.
	Suppose that ${b\colon \bbR^m\! \to \bbR^m}$
	and $\sigma\colon \bbR^m \to \bbR^{m\times d}$
	fulfill Assumption~{\upshape\ref{ass:lipschitz_coef}} 
	with the constant $L \in (0,\infty)$. 
	Let $K_1, K_2$ be Volterra kernels
	satisfying Assumption~{\upshape\ref{ass:assump_nonconv_kernels}}, 
	with constants $\eta_T \in (0,\infty)$ 
	and $\alpha_T \in (\nicefrac{1}{p},1]$, 
	for every $T \in (0,\infty)$. 
	Then, there exists
	a continuous stochastic Volterra process
	$X\colon [0,\infty) \times \Omega \to \bbR^m$
	with initial value $X_0$,
	coefficients $b,\sigma$,
	and kernels $K_1, K_2$,
	which is unique up to indistinguishability.
	
	For every $T \in (0,\infty)$, and 
	for $\alpha:=\alpha_T \in (\nicefrac{1}{p}, 1]$,  
	the continuous stochastic Volterra process $X$ is stable 
	in $ C^\alpha([0,T];L^p(\Omega))$, i.e.,   
	\begin{equation*}
		\| X \|_{C^\alpha([0,T];L^p(\Omega))}
		= 
		\sup_{t\in[0,T]} \| X_t \|_{L^p(\Omega)}
		+ 
		\sup_{0\leq s < t \leq T} 
		\frac{ \| X_t - X_s \|_{L^p(\Omega)} }{(t-s)^\alpha} 
		< 
		\infty. 
	\end{equation*}
	Furthermore, 
	for all $\nu\in(0,\alpha-\nicefrac{1}{p})$, 
	$X$ has $\nu$-H\"older continuous sample paths, 
	and    
	\begin{equation*} 
		\| X \|_{L^p(\Omega;C^\nu([0,T]))}  
		= 
		\biggl(
		\bbE\biggl[ 
		\biggl( 
		\sup_{t\in[0,T]} \| X_t \|_{\bbR^m}  
		+ 
		\sup_{0 \leq s < t \leq T}
		\frac{ \| X_t - X_s \|_{\bbR^m}}{ (t-s)^{\nu} } 
		\biggr)^p 
		\biggr] 
		\biggr)^{\nicefrac{1}{p}}  
		< \infty.
	\end{equation*}
\end{proposition}

\begin{proof}
	By Proposition~{\upshape\ref{prop:existence_uniqueness_SVE}} 
	there exists a unique measurable 
	$(\cF_t)_{t\geq 0}$-adapted stochastic process
	$Y\colon [0,\infty) \times \Omega \to \bbR^m$ 
	which satisfies \eqref{eq:SVE_def_eq} 
	for almost all $t\in[0,\infty)$.  
	Then, $Y$ admits a progressively measurable 
	modification 
	(see \cite[Proposition~1.12]{KaratzasShreve1991}), 
	which also satisfies \eqref{eq:SVE_def_eq} 
	for almost all $t\in[0,\infty)$, 
	and which we again denote by $Y$.  
	
	Let $T\in(0,\infty)$ and 
	let $A\subset [0,\infty)$ be the 
	null set where \eqref{eq:SVE_def_eq}  
	does not hold. 
	Let $s,t \in [0,T]\setminus A$ 
	with $s<t$, and 
	define $Y_{t,s} := Y_t - Y_s$. 
	Then,  
	\begin{equation}\label{eq:SVE_increment_decomposition} 
		\begin{split} 
			Y_{t,s}
			&=  
			\int_s^t K_1(t,u) b(Y_u) \diff{u} 
			+ 
			\int_0^s \bigl( K_1(t,u) - K_1(s,u) \bigr) b(Y_u) \diff{u}           
			\\
			&\quad  
			+ 
			\int_s^t K_2(t,u) \sigma(Y_u) \diff{W_u} 
			+  
			\int_0^s \bigl(K_2(t,u) - K_2(s,u) \bigr) \sigma(Y_u) \diff{W_u}, 
			\quad 
			\bbP\text{-a.s.}\hspace{-4mm} 
		\end{split} 
	\end{equation}
	By defining the increments $K_{j;t,s}$ 
	as in \eqref{eq:K_jts_definition}
	and by applying the triangle 
	and the Cauchy--Schwarz inequalities
	we thus obtain that 
	\begin{align}
		\| Y_{t,s} \|_{L^p(\Omega)}^2
		&\leq 
		4 \left\| \int_s^t K_1(t,u) b(Y_u) \diff{u} \right\|_{L^p(\Omega)}^2
		+  
		4 \left\| \int_0^s K_{1;t,s}(u) b(Y_u) \diff{u} \right\|_{L^p(\Omega)}^2
		\notag 
		\\
		&\quad 
		+ 
		4 \left\| \int_s^t K_2(t,u) \sigma(Y_u) \diff{W_u} \right\|_{L^p(\Omega)}^2       
		+ 
		4 \left\| \int_0^s K_{2;t,s}(u) \sigma(Y_u) \diff{W_u} \right\|_{L^p(\Omega)}^2 
		\notag 
		\\
		&=:   
		4 \bigl( \mathrm{(I)} + \mathrm{(II)} + \mathrm{(III)} + \mathrm{(IV)} \bigr). 
		\label{eq:decompose_Yts_in_4}  
	\end{align} 
	We now estimate the four terms 
	of \eqref{eq:decompose_Yts_in_4} separately.
	For this purpose, we recall  
	the constants 
	$\overline{b}_T, \overline{\sigma}_T \in [0,\infty)$
	from \eqref{eq:bbar_sigmabar}, 
	and let $\eta:= \eta_T \in(0,\infty)$, 
	$\alpha:= \alpha_T \in(\nicefrac{1}{p}, 1]$. 
	
	For the first term 
	in \eqref{eq:decompose_Yts_in_4}, 
	we readily obtain 
	by Assumption~{\upshape\ref{ass:assump_nonconv_kernels}\ref{item:assu_noco_i}} 
	that 
	\begin{equation*} 
		\mathrm{(I)}                                                                                              
		\leq 
		\left( \int_s^t | K_1(t,u) | \| b(Y_u) \|_{L^p(\Omega)} \diff{u} \right)^2 
		\leq 
		\overline{b}_T^2 
		\| K_1(t,\,\cdot\,) \|_{L^1(s,t)}^2 
		\leq 
		\overline{b}_T^2 \eta^2 (t-s)^{2\alpha}, 
	\end{equation*} 
	whereas 
	for the second term 
	we use 
	Assumption~{\upshape\ref{ass:assump_nonconv_kernels}\ref{item:assu_noco_ii}} 
	and find similarly that 
	\begin{equation*}
		\mathrm{(II)}                  
		\leq 
		\left(\int_0^s | K_{1;t,s}(u) | \| b(Y_u) \|_{L^p(\Omega)} \diff{u}\right)^2 
		\leq 
		\overline{b}_T^2 
		\| K_{1;t,s} \|_{L^1(0,s)}^2 
		\leq 
		\overline{b}_T^2 \eta^2 (t-s)^{2\alpha}.
	\end{equation*}
	
	For the third and fourth terms  
	in \eqref{eq:decompose_Yts_in_4},  
	we use the Burkholder--Davis--Gundy inequality 
	(with constant $\bdg{p} \in(0,\infty)$, 
	see Lemma~{\upshape\ref{lma:BDG_SVE}}), 
	and conclude that  
	\begin{align*}
		\mathrm{(III)}                                                                                              
		&\leq 
		\bdg{p} 
		\overline{\sigma}_T^2 
		\| K_2(t,\,\cdot\,) \|_{L^2(s,t)}^2 
		\leq 
		\bdg{p} 
		\overline{\sigma}_T^2 \eta^2 (t-s)^{2\alpha}, 
		\\ 
		\mathrm{(IV)}                                                                                              
		&\leq  
		\bdg{p} 
		\overline{\sigma}_T^2 
		\| K_{2;t,s}  \|_{L^2(0,s)}^2 
		\leq 
		\bdg{p} 
		\overline{\sigma}_T^2 \eta^2 (t-s)^{2\alpha}.  
	\end{align*} 
	
	Collecting all four terms of \eqref{eq:decompose_Yts_in_4} 
	and taking the square root on both sides 
	yields that, for all $s,t \in [0,T]\setminus A$ with $s<t$,  
	\begin{equation*}
		\| Y_{t,s} \|_{L^p(\Omega)} 
		\leq 
		2 \eta \bigl( 2\overline{b}_T^2 + 2\bdg{p} \overline{\sigma}_T^2  \bigr)^{\nicefrac{1}{2}}  
		\, 
		(t-s)^\alpha .
	\end{equation*}
	By completeness of $L^p(\Omega)$, 
	the restriction of the process $Y$ to $[0,T]\times\Omega$
	has a unique   
	representative in $C^\alpha([0,T];L^p(\Omega))$, 
	which also satisfies \eqref{eq:SVE_def_eq} 
	for almost all $t\in[0,T]$, 
	again denoted by~$Y$. 
	By the Kolmogorov--Chentsov continuity theorem 
	\cite[Theorem~3.9]{coxCompleteness2024},  
	there exists a modification $X$ of $Y$ 
	which, for every ${\nu \in (0, \alpha - \nicefrac{1}{p})}$, 
	has $\nu$-H\"older continuous sample paths 
	and 
	${\| X \|_{C^\alpha([0,T];L^p(\Omega))} 
	= \| Y \|_{C^\alpha([0,T];L^p(\Omega))} < \infty}$,  
	$\| X \|_{L^p(\Omega;C^\nu([0,T]))} < \infty$. 
	Moreover, $X$ then satisfies \eqref{eq:SVE_def_eq} 
	for all $t\in[0,T]$, and 
	since $T\in(0,\infty)$ was arbitrary, 
	we conclude that there exists 
	a continuous stochastic Volterra process
	$X\colon [0,\infty) \times \Omega \to \bbR^m$
	with initial value~$X_0$,
	coefficients $b,\sigma$,
	and kernels $K_1, K_2$. 
	Moreover, 
	if $\widetilde{X}\colon [0,\infty) \times \Omega \to \bbR^m$ 
	is another such 
	continuous stochastic Volterra process, 
	then by \cite[Theorem~3.1]{zhang_SPDE_banach2010}, 
	the processes 
	$X$ and $\widetilde{X}$ are equal a.e.\ in $[0,\infty)\times\Omega$ 
	and, by their sample path continuity, 
	indistinguishable. 
\end{proof}


\subsection{Pointwise differences of Volterra processes}
\label{subsec:general_SVE_diff}

In what follows, 
we consider a stochastic Volterra process 
$X$ with kernels $K_1, K_2$, 
and a family of Volterra kernels
$( \widehat{K}^\xi_1, \widehat{K}^\xi_2 )_{\xi \in \Xi}$ 
which approximate $K_1,K_2$.
This kernel family induces 
a corresponding family
of Volterra processes 
$(\widehat{X}^\xi)_{\xi \in \Xi}$ 
approximating $X$. 
By Proposition~{\upshape\ref{prop:X_ct_holder_Lp}} 
the following assumption 
guarantees existence and uniqueness 
of $\widehat{X}^\xi$, for every $\xi\in\Xi$.  

\begin{assumption}[Non-convolution kernel approximations]
\label{ass:assu_noco_unif_xi}
	Let $\Xi$ be a non-empty set,
	and let
	$( \widehat{K}^\xi_1, \widehat{K}^\xi_2 )_{\xi \in \Xi}$
	be a family of Volterra kernels on $[0,\infty)$, 
	see Definition~{\upshape\ref{def:volterra_kernel}}. 
	We suppose that there exist 
	Volterra kernels 
	$\overline{K}_1, \overline{K}_2 \colon [0,\infty)\times[0,\infty)\to[0,\infty)$
	and, for all $T \in (0,\infty)$, 
	there are constants
	$\eta = \eta_T \in (0,\infty)$, $\alpha = \alpha_T \in (0,1]$ 
	such that: 
	\begin{enumerate}[leftmargin=8mm, label={\upshape(\roman*)}] 
		\item \label{item:assu_noco_i_unif}  
		For every $\xi\in\Xi$,  
		\begin{equation*} 
			| \widehat{K}^\xi_1(t,s) | \leq \overline{K}_1(t,s),
			\quad\;\; 
			| \widehat{K}^\xi_2(t,s) | \leq \overline{K}_2(t,s),
			\quad\;\;  
			\text{a.e.\ in }\triangle[0,T], 
		\end{equation*}
		$\overline{K}_j 
		\in C^0([0,T]; L^j(0,T))$, 
		for $j\in\{1,2\}$,  
		and, for all $(t,s) \in \triangle[0,T]$, 
		\begin{equation*}
			\left\| \overline{K}_1(t,\,\cdot\,) \right\|_{L^1(s,t)} \leq \eta (t-s)^{\alpha},
			\qquad
			\left\| \overline{K}_2(t,\,\cdot\,) \right\|_{L^2(s,t)} \leq \eta (t-s)^{\alpha}.
		\end{equation*}
		\item \label{item:assu_noco_ii_unif} 
		For every $\xi\in\Xi$ and all $(t,s) \in \triangle[0,T]$, 
		\begin{equation*} 
			\| \widehat{K}^\xi_1(t,\,\cdot\,) - \widehat{K}^\xi_1(s,\,\cdot\,) \|_{L^1(0,s)} 
			\leq \eta (t-s)^\alpha,
			\;\;\;  
			\| \widehat{K}^\xi_2(t,\,\cdot\,) - \widehat{K}^\xi_2(s,\,\cdot\,) \|_{L^2(0,s)} 
			\leq \eta (t-s)^\alpha.
		\end{equation*} 
	\end{enumerate}
\end{assumption}

In the case of convolution kernels
we will work in the following setting.

\begin{assumption}[Convolution kernel approximations I]
\label{ass:assu_co_unif_xi}
	Let $\Xi$ be a non-empty~set, 
	and let
	$( \widehat{k}^\xi_1, \widehat{k}^\xi_2 )_{\xi \in \Xi}$
	be a family of Borel measurable real-valued functions on $(0,\infty)$ 
	which are 
	non-negative and non-increasing 
	a.e.\ on $(0,\infty)$. 
	We suppose that there~are 
	non-increasing   
	Borel measurable  
	functions 
	$\overline{k}_1, \overline{k}_2 \colon (0,\infty) \to[0,\infty)$ 
	such that, for every ${\xi\in\Xi}$, 
	\begin{equation*}  
		\widehat{k}^\xi_1 (t) \leq \overline{k}_1(t),
		\quad\;\; 
		\widehat{k}^\xi_2 (t) \leq \overline{k}_2(t), 
		\quad\;\;  
		\text{a.e.\ in } (0,\infty), 
	\end{equation*} 
	and, for all $T \in (0,\infty)$, 
	there are 
	$\eta = \eta_ T \in (0,\infty)$ 
	and $\alpha = \alpha_T \in (0,1]$
	such that 
	\begin{equation*} 
		\forall t\in [0,T]: 
		\quad\; 
		\| \overline{k}_1 \|_{L^1(0,t)} \leq \eta \, t^{\alpha}, 
		\quad\;\; 
		\| \overline{k}_2 \|_{L^2(0,t)} \leq \eta \, t^{\alpha}. 
	\end{equation*} 
\end{assumption}
 
Lemma~{\upshape\ref{lma:kernel_assumption_implications}} 
shows that Assumption~{\upshape\ref{ass:assu_co_unif_xi}} 
implies Assumption~{\upshape\ref{ass:assu_noco_unif_xi}} 
for the associated family of Volterra kernels 
defined as convolution kernels, 
cf.~\eqref{eq:induced_conv_volterra_kernel}. 
Importantly, 
the preceding assumptions guarantee 
that a family of continuous 
stochastic Volterra processes 
$(\widehat{X}^\xi)_{\xi \in \Xi}$
with kernels $( \widehat{K}^\xi_1, \widehat{K}^\xi_2 )_{\xi \in \Xi}$  
is stable in $C^0([0,T];L^p(\Omega))$, 
uniformly in $\xi\in\Xi$. 
This $\xi$-uniform stability is subject  
of the following proposition.  

\begin{proposition}\label{prop:xi_bdd_stability}
	Let $p \in [2,\infty)$ and  
	$X_0 \in L^p(\Omega,\cF_0,\bbP)$.  
	Suppose that 
	$( \widehat{K}^\xi_1, \widehat{K}^\xi_2 )_{\xi \in \Xi}$ 
	is a family of Volterra kernels 
	satisfying Assumption~{\upshape\ref{ass:assu_noco_unif_xi}} with 
	constants 
	$\eta_T \in (0,\infty)$ 
	and $\alpha_T \in (\nicefrac{1}{p},1]$, 
	for all $T\in(0,\infty)$. 
	For every $\xi\in\Xi$, 
	let $\widehat{X}^\xi\colon [0,\infty)\times \Omega \to \bbR^m$ be a  
	continuous stochastic Volterra process with
	initial value $X_0$,
	coefficients $b,\sigma$ fulfilling   
	Assumption~{\upshape\ref{ass:lipschitz_coef}} 
	with $L\in(0,\infty)$, 
	and kernels $\widehat{K}^\xi_1, \widehat{K}^\xi_2$. 
	Then, for every~$T\in(0,\infty)$, 
	\begin{equation*} 
		\sup\limits_{\xi\in\Xi} \| \widehat{X}^\xi \|_{C^0([0,T];L^p(\Omega))} 
		= 
		\sup\limits_{\xi\in\Xi} 
		\sup\limits_{t\in[0,T]} \| \widehat{X}^\xi_t \|_{L^p(\Omega)}
		< \infty. 
	\end{equation*} 
\end{proposition}

\begin{proof}
	Let $T\in(0,\infty)$, and set 
	$\eta:= \eta_T\in(0,\infty)$, 
	$\alpha:=\alpha_T\in(\nicefrac{1}{p}, 1]$. 
	By the triangle, Cauchy--Schwarz, 
	and Burkholder--Davis--Gundy inequalities 
	(with the constant $\bdg{p} \in(0,\infty)$, 
	see Lemma~{\upshape\ref{lma:BDG_SVE}})
	we obtain that, for every $\xi\in\Xi$ 
	and all $t\in[0,T]$, 
	\begin{align*}
		\| \widehat{X}^\xi_t \|_{L^p(\Omega)}^2
		&\leq 
		3 \| X_0 \|_{L^p(\Omega)}^2
		+ 
		3 \left( 
		\int_0^t | \widehat{K}^\xi_1(t,u)| \|b(\widehat{X}^\xi_u)\|_{L^p(\Omega)} \diff{u} 
		\right)^2
		\\
		&\;\; 
		+ 3 \bdg{p} 
		\int_0^t | \widehat{K}^\xi_2(t,u) |^2 \| \sigma(\widehat{X}^\xi_u) \|_{L^p(\Omega)}^2 \diff{u} 
		=: 
		3 \bigl( \| X_0 \|_{L^p(\Omega)}^2 + \mathrm{(I)} + \bdg{p}  \mathrm{(II)} \bigr) .  
	\end{align*}
	Using Assumptions~{\upshape\ref{ass:lipschitz_coef}} 
	and~{\upshape\ref{ass:assu_noco_unif_xi}\ref{item:assu_noco_i_unif}} 
	we can bound the terms $\mathrm{(I)}$ 
	and $\mathrm{(II)}$ as follows, 
	\begin{align*} 
		\mathrm{(I)} 
		&\leq 
		\left( 
		\| b(0) \|_{\bbR^m} 
		\int_0^t | \overline{K}_1(t,u)| \diff{u}  
		+ 
		L 
		\int_0^t | \overline{K}_1(t,u)| \| \widehat{X}^\xi_u \|_{L^p(\Omega)} \diff{u} 
		\right)^2 
		\\
		&\leq 
		2 \| b(0) \|_{\bbR^m}^2 \| \overline{K}_1(t, \,\cdot\,) \|_{L^1(0,t)}^2 
		+ 
		2 L^2 \| \overline{K}_1(t, \,\cdot\,) \|_{L^1(0,t)} 
		\int_0^t | \overline{K}_1(t,u)| \| \widehat{X}^\xi_u \|_{L^p(\Omega)}^2 \diff{u} 
		\\
		&\leq 
		2\| b(0) \|_{\bbR^m}^2 \eta^2 t^{2\alpha} 
		+ 
		2 L^2 \eta t^{\alpha} \int_0^t | \overline{K}_1(t,u) | \| \widehat{X}^\xi_u \|_{L^p(\Omega)}^2 \diff{u}, 
		\\
		\mathrm{(II)} 
		&\leq                                                                                    
		\int_0^t | \overline{K}_2(t,u) |^2 \left( \| \sigma(0) \|_{ \bbR^{m\times d}} 
		+ 
		L \| \widehat{X}^\xi_u \|_{L^p(\Omega)} \right)^2 \diff{u} 
		\\ 
		&\leq 
		2 \| \sigma(0) \|_{\bbR^{m\times d}}^2 
		\eta^2 t^{2\alpha} 
		+ 
		2 L^2 \int_0^t | \overline{K}_2(t,u) |^2 \| \widehat{X}^\xi_u \|_{L^p(\Omega)}^2  \diff{u}.
	\end{align*}
	Next, define the functions 
	$f\colon[0,T]\to [0,\infty)$ and 
	$\kappa \colon \triangle [0,T] \to [0,\infty)$ by
	\begin{align*}
		f(t)      
		& := 
		3 \| X_0 \|_{L^p(\Omega)}^2 
		+ 
		6 \left(\| b(0) \|_{\bbR^m}^2 
		+ 
		\bdg{p} \| \sigma(0) \|_{ \bbR^{m\times d}}^2\right) \eta^2 t^{2\alpha}, 
		\\
		\kappa(t,s) 
		& := 
		6 L^2 \eta T^\alpha | \overline{K}_1(t,s) | + 6\bdg{p} L^2 | \overline{K}_2(t,s) |^2 .
	\end{align*}
	Then, we obtain, for every $\xi\in\Xi$ and all 
	$t\in[0,T]$, 
	the estimate  
	\begin{equation*} 
		\| \widehat{X}^\xi_t \|_{L^p(\Omega)}^2 
		\leq 
		f(t) + \int_0^t \kappa(t,u) \| \widehat{X}^\xi_u \|_{L^p(\Omega)}^2 \diff{u}.
	\end{equation*} 
	It is clear that $f$
	is continuous on~$[0,T]$. 
	Additionally, by 
	Assumption~{\upshape\ref{ass:assu_noco_unif_xi}\ref{item:assu_noco_i_unif}}, 
	\begin{equation*} 
		\forall (t,s)\in\triangle[0,T]: 
		\quad 
		\int_s^t \kappa(t,u) \diff{u} 
		\leq 
		6 L^2 (1 + \bdg{p}) \eta^2 T^\alpha (t-s)^\alpha.
	\end{equation*} 
	Proposition~{\upshape\ref{prop:X_ct_holder_Lp}} 
	(applied for $\widehat{X}^\xi$) 
	shows that the mapping 
	$[0,T]\ni t \mapsto \widehat{X}^\xi_t \in L^p(\Omega)$ 
	is continuous, 
	for every $\xi\in\Xi$, and, thus, 
	also 
	\begin{equation*} 
		x^\xi\colon[0,T]\to[0,\infty), 
		\qquad 
		x^\xi(t) := \| \widehat{X}^\xi_t \|_{L^p(\Omega)}^2, 
	\end{equation*} 
	is continuous on $[0,T]$. 
	Therefore, for every $\xi\in\Xi$, 
	all of the assumptions 
	of the Gr\"onwall inequality 
	(see   
	Lemma~{\upshape\ref{lem:gronwall_ineq_1var}}) 
	are satisfied, 
	whose application yields 
	the existence of a constant $C_{\kappa,T}\in(0,\infty)$, 
	depending only on $\kappa$ and $T$, 
	such that 
	\begin{equation*} 
		\sup_{t\in[0,T]} 
		\| \widehat{X}^\xi_t \|_{L^p(\Omega)}^2 
		= 
		\sup_{t\in[0,T]} 
		x^\xi (t) 
		\leq 
		C_{\kappa,T} \sup_{v \in [0,T]} f(v) . 
	\end{equation*}  
	Since $\kappa$ depends only on the kernels 
	$\overline{K}_1, \overline{K}_2$, and not on 
	$\widehat{K}_1^\xi, \widehat{K}_2^\xi$, the constant  
	$C_{\kappa,T} \in (0,\infty)$ is independent of $\xi\in\Xi$,  
	and we conclude that 
	\begin{equation*} 
		\sup_{\xi\in\Xi}
		\sup_{t\in[0,T]} 
		\| \widehat{X}^\xi_t \|_{L^p(\Omega)}^2 
		\leq 
		3 C_{\kappa,T} 
		\bigl( 
		\| X_0 \|_{L^p(\Omega)}^2 
		+ 
		2 \left(\| b(0) \|_{\bbR^m}^2 
		+ 
		\bdg{p} \| \sigma(0) \|_{ \bbR^{m\times d}}^2\right) \eta^2 T^{2\alpha} 
		\bigr), 
	\end{equation*} 
	which completes the proof.  
\end{proof} 

\begin{remark}\label{rmk:def_bT-hat_sigmaT-hat} 
	Analogously to Remark~{\upshape\ref{rmk:def_bT_sigmaT}}, 
	we have under the assumptions 
	of Proposition~{\upshape\ref{prop:xi_bdd_stability}} 
	that  
	\begin{equation*}
		\sup_{\xi \in \Xi} \sup_{t \in [0,T]} 
		\| b(\widehat{X}^{\xi}_t) \|_{L^p(\Omega)}
		\leq 
		\widehat{b}_T < \infty,
		\qquad
		\sup_{\xi \in \Xi} \sup_{t \in [0,T]} 
		\| \sigma(\widehat{X}^{\xi}_t) \|_{L^p(\Omega)}
		\leq 
		\widehat{\sigma}_T < \infty,
	\end{equation*}  
	where we define 
	\begin{equation}\label{eq:bhat_sigmahat}    
		\begin{split}
			\widehat{b}_T 
			&:=
			\| b(0) \|_{\bbR^m} 
			+ 
			L \sup\limits_{\xi \in \Xi} \sup\limits_{t \in [0,T]} \| \widehat{X}^{\xi}_t \|_{L^p(\Omega)},
			\\ 
			\widehat{\sigma}_T 
			&:= 
			\| \sigma(0) \|_{\bbR^{m\times d}} 
			+ 
			L \sup\limits_{\xi \in \Xi} \sup\limits_{t \in [0,T]} \| \widehat{X}^{\xi}_t \|_{L^p(\Omega)}.
		\end{split} 
	\end{equation}
\end{remark} 

For abbreviation,
we denote the differences between 
the true and approximating kernels
by $\Delta^\xi_1$ and $\Delta^\xi_2$,
i.e., for $j \in \{1,2\}$, $\xi \in \Xi$,
$(t,s) \in\triangle [0,\infty)$, and $u \in [0,s]$,
\begin{equation}\label{eq:delta_definition}
	\Delta^\xi_j(t,s) := K_j(t,s) - \widehat{K}^\xi_j(t,s),
	\qquad
	\Delta^\xi_{j;t,s}(u) := \Delta^\xi_j(t,u) - \Delta^\xi_j(s,u).
\end{equation}
With slight abuse of notation,
we denote the differences 
of convolution kernels also 
by $\Delta^\xi_1$ and $\Delta^\xi_2$, 
with one argument, i.e.,
for $t \in [0,\infty)$, $s \in [0,t]$,
and $u \in [0,s]$, 
\begin{equation}\label{eq:delta_definition_conv}
	\Delta^\xi_j(t) := k_j(t) - \widehat{k}^\xi_j(t),
	\qquad
	\Delta^\xi_{j;t,s}(u) := \Delta^\xi_j(t-u) - \Delta^\xi_j(s-u).
\end{equation}

In the following theorem, 
we bound the difference 
between a stochastic Volterra 
process~$X$ with (possibly non-convolution) kernels 
$K_1,K_2$ and an approximation~$\widehat{X}^\xi$ 
with approximate kernels $\widehat{K}_1^\xi, \widehat{K}_2^\xi$ 
in the norm of $C^0([0,T];L^p(\Omega))$ 
for any ${p\in[2,\infty)}$. 
This result constitutes an extension 
of \cite[Theorem~3.1]{alfonsikebaier2022}, 
where the case of convolution kernels 
of completely monotone type 
and $p=2$ is considered.

\begin{theorem}\label{thm:C0Lp-err_non-conv}
	Let $p \in [2,\infty)$, and  
	suppose that ${X \colon [0,\infty) \times \Omega \to \bbR^m}$   
	is a continuous stochastic Volterra process with
	initial value $X_0 \in L^p(\Omega,\cF_0,\bbP)$,
	coefficients~$b, \sigma$ fulfilling 
	Assumption~{\upshape\ref{ass:lipschitz_coef}} 
	with the constant $L \in (0,\infty)$,
	and kernels $K_1, K_2$ satisfying
	Assumption~{\upshape\ref{ass:assump_nonconv_kernels}} 
	with   
	$\eta_T \in (0,\infty)$ and 
	$\alpha_T \in (\nicefrac{1}{p},1]$, 
	for every $T\in(0,\infty)$. 
	
	Assume that 
	$(\widehat{X}^\xi \colon [0,\infty) \times \Omega \to \bbR^m)_{\xi\in\Xi}$ 
	is a family 
	of continuous stochastic Volterra processes 
	with the same initial value and coefficients as~$X$, 
	whose kernels  
	$(\widehat{K}^\xi_1, \widehat{K}^\xi_2)_{\xi \in \Xi}$
	satisfy Assumption~{\upshape\ref{ass:assu_noco_unif_xi}} 
	with the same constants 
	$(\eta_T)_{T>0}$ and~$(\alpha_T)_{T>0}$. 
	
	For every $T\in(0,\infty)$, 
	there exists a constant $C \in (0,\infty)$  
	such that, for all $\xi \in \Xi$,  
	\begin{align*} 
			\forall t \in [0,T]: 
			\quad 
			\| X_t - \widehat{X}^\xi_t \|_{L^p(\Omega)}
			&\leq 
			C \Bigl( \, \sup_{v \in [0,t]}\| K_1(v,\,\cdot\,) - \widehat{K}^\xi_1(v,\,\cdot\,) \|_{L^1(0,v)}
			\\
			&\qquad\;\; + 
			\sup_{v \in [0,t]}\| K_2(v,\,\cdot\,) - \widehat{K}^\xi_2(v,\,\cdot\,) \|_{L^2(0,v)} \Bigr). 
	\end{align*}
	In particular, the constant $C \in (0,\infty)$ is independent 
	of $\xi\in\Xi$. 
\end{theorem}

\begin{proof}
	Let $T\in(0,\infty)$ and $\xi \in \Xi$ be 
	arbitrary, and define 
	$\Delta^\xi_1$, $\Delta^\xi_2$ as in \eqref{eq:delta_definition}. 
	Note that, for all $t \in [0,T]$ and $u \in [0,t]$, 
	the following identities hold: 
	\begin{align} 
		\hspace*{-2mm}
		K_1(t,u) b(X_u) - \widehat{K}^{\xi}_1(t,u) b(\widehat{X}^{\xi}_u)
		&=  
		\Delta^\xi_1(t,u) b(\widehat{X}^{\xi}_u) 
		+ K_1(t,u) \bigl( b(X_u) - b(\widehat{X}^{\xi}_u) \bigr), 
		\hspace*{-1.8mm} 
		\label{eq:decomposition_for_X_minus_Xhat}
		\\ 
		\hspace*{-2mm}
		K_2(t,u) \sigma(X_u) - \widehat{K}^{\xi}_2(t,u) \sigma(\widehat{X}^{\xi}_u) 
		&= 
		\Delta^\xi_2(t,u) \sigma(\widehat{X}^{\xi}_u) 
		+ 
		K_2(t,u) \bigl( \sigma(X_u) - \sigma(\widehat{X}^{\xi}_u) \bigr). 
		\hspace*{-1.8mm} 
		\notag 
	\end{align} 
	Therefore, we obtain that, for all $t \in [0,T]$, $\bbP$-a.s., 
	\begin{align*}
		X_t - \widehat{X}^{\xi}_t
		&= 
		\int_0^t \Delta^\xi_1(t,u) b(\widehat{X}^{\xi}_u) \diff{u} 
		+ 
		\int_0^t K_1(t,u) \bigl(b(X_u) - b(\widehat{X}^{\xi}_u)  \bigr) \diff{u}                           
		\\
		&\quad 
		+ 
		\int_0^t \Delta^\xi_2(t,u) \sigma(\widehat{X}^{\xi}_u) \diff{W_u} 
		+ 
		\int_0^t K_2(t,u) \bigl(\sigma(X_u) - \sigma(\widehat{X}^{\xi}_u) \bigr) \diff{W_u}. 
	\end{align*}
	The triangle and Cauchy--Schwarz inequalities yield, 
	for all $t\in[0,T]$, the~bound 
	\begin{align*} 
		&\| X_t - \widehat{X}^{\xi}_t \|_{L^p(\Omega)}^2 
		\leq 
		4\,
		\biggl\| \int_0^t \Delta^\xi_1(t,u) b(\widehat{X}^{\xi}_u) \diff{u} \biggr\|_{L^p(\Omega)}^2 
		\\
		&\quad 
		+ 4\, 
		\biggl\| \int_0^t K_1(t,u) \bigl( b(X_u) - b(\widehat{X}^{\xi}_u) \bigr) \diff{u} \biggr\|_{L^p(\Omega)}^2  
		+
		4\, 
		\biggl\| \int_0^t \Delta^\xi_2(t,u) \sigma(\widehat{X}^{\xi}_u) \diff{W_u} \biggr\|_{L^p(\Omega)}^2 
		\\
		&\quad 
		+ 4\, 
		\biggl\| \int_0^t K_2(t,u) \bigl( \sigma(X_u) - \sigma(\widehat{X}^{\xi}_u) \bigr) \diff{W_u} \biggr\|_{L^p(\Omega)}^2  
		= 
		4 \bigl( \mathrm{(I)} + \mathrm{(II)} + \mathrm{(III)} + \mathrm{(IV)} \bigr). 
	\end{align*}  
	
	Letting the constants $\widehat{b}_T, \widehat{\sigma}_T \in [0, \infty)$ 
	be defined 
	as in \eqref{eq:bhat_sigmahat} 
	of Remark~{\upshape\ref{rmk:def_bT-hat_sigmaT-hat}}, 
	and setting $\eta:=\eta_T \in(0,\infty)$, 
	$\alpha := \alpha_T \in (\nicefrac{1}{p}, 1]$, 
	we obtain that 
	\begin{equation*} 
		\mathrm{(I)} 
		\leq 
		\biggl(\int_0^t | \Delta^\xi_1(t,u) | \| b(\widehat{X}^{\xi}_u) \|_{L^p(\Omega)} \diff{u}  \biggr)^2  
		\leq 
		\widehat{b}_T^2 \| \Delta^\xi_1(t,\,\cdot\,) \|_{L^1(0,t)}^2 . 
	\end{equation*} 
	The second term can be bounded 
	by H\"older's inequality 
	and 
	Assumptions~{\upshape\ref{ass:lipschitz_coef}},~{\upshape\ref{ass:assump_nonconv_kernels}\ref{item:assu_noco_i}}, 
	\begin{align*} 
		\mathrm{(II)} 
		&\leq  
		L^2  \| K_1(t,\,\cdot\,) \|_{L^1(0,t)} \int_0^t | K_1(t,u) |  \| X_u - \widehat{X}^{\xi}_u \|_{L^p(\Omega)}^2 \diff{u}
		\\ 
		&\leq 
		L^2 \eta T^{\alpha}  \int_0^t | K_1(t,u) |  \| X_u - \widehat{X}^{\xi}_u \|_{L^p(\Omega)}^2 \diff{u} . 
	\end{align*}

	For the terms $\mathrm{(III)}$ and 
	$\mathrm{(IV)}$, 
	we additionally use 
	the Burkholder--Davis--Gundy inequality 
	(with constant $\bdg{p} \in (0,\infty)$, 
	see Lemma~{\upshape\ref{lma:BDG_SVE}}), 
	and obtain that 
	\begin{align*} 
		\mathrm{(III)} 
		&\leq 
		\bdg{p} \int_0^t | \Delta^\xi_2(t,u) |^2  \| \sigma(\widehat{X}^{\xi}_u) \|_{L^p(\Omega)}^2 \diff{u} 
		\leq 
		\bdg{p}  \widehat{\sigma}_T^2 \| \Delta^\xi_2(t,\,\cdot\,) \|_{L^2(0,t)}^2,
		\\ 
		\mathrm{(IV)} 
		&\leq 
		\bdg{p} L^2  \int_0^t | K_2(t,u) |^2 \| X_u - \widehat{X}^{\xi}_u \|_{L^p(\Omega)}^2 \diff{u}.
	\end{align*} 
	Define the functions 
	$f^\xi\colon [0,T] \to [0,\infty)$ and 
	$\kappa\colon \triangle [0,T] \to [0,\infty)$ by
	\begin{align*}
		f^\xi(t) 
		&:= 
		4 \, \widehat{b}_T^2 \, \| \Delta^\xi_1(t,\,\cdot\,) \|_{L^1(0,t)}^2 
		+ 4 \bdg{p} \, \widehat{\sigma}_T^2 \, \| \Delta^\xi_2(t,\,\cdot\,) \|_{L^2(0,t)}^2,
		\\
		\kappa(t,s) 
		&:= 
		4L^2 \bigl( \eta T^\alpha | K_1(t,s) | + \bdg{p} | K_2(t,s) |^2 \bigr).
	\end{align*}
	Under Assumptions~{\upshape\ref{ass:assump_nonconv_kernels}} 
	and~{\upshape\ref{ass:assu_noco_unif_xi}}, 
	the maps    
	$t \mapsto K_j(t, \,\cdot\,)$ 
	and 
	${t \mapsto \widehat{K}^\xi_j(t, \,\cdot\,)}$, 
	${j\in\{1,2\}}$, 
	are elements of $C^\alpha([0,T]; L^j(0,T))$, 
	see 
	Lemma~{\upshape\ref{lma:continuous_representative}}. 
	Thus, the map 
	${t \mapsto \Delta^\xi_j(t, \,\cdot\,)}$
	is also an element 
	of $C^\alpha([0,T]; L^j(0,T)) 
	\subset C^0([0,T]; L^j(0,T))$, 
	and 
	Lemma~{\upshape\ref{lma:u_continuity_lemma}} 
	implies that 
	$f^\xi\colon[0,T]\to [0,\infty)$ 
	is continuous.  
	In addition, by 
	Assumption~{\upshape\ref{ass:assump_nonconv_kernels}\ref{item:assu_noco_i}},  
	\begin{equation*} 
		\forall (t,s)\in\triangle[0,T]: 
		\quad 
		\int_s^t \kappa(t,u) \diff{u} 
		\leq 
		4L^2 (1 + \bdg{p} ) \eta^2 T^\alpha (t-s)^\alpha.
	\end{equation*}
	Define the function 
	$x^\xi\colon[0,T]\to [0,\infty)$
	by 
	$x^\xi(t) := \| X_t - \widehat{X}^{\xi}_t \|_{L^p(\Omega)}^2$.
	Proposition~{\upshape\ref{prop:X_ct_holder_Lp}} 
	(applied for $X$ and $\widehat{X}^\xi$) 
	shows that $x^\xi$ is continuous, 
	and by the Gr\"onwall inequality 
	of Lemma~{\upshape\ref{lem:gronwall_ineq_1var}},
	there exists a constant $\gronw \in (0,\infty)$ 
	such that, for all $t \in [0,T]$,
	\begin{equation*}
		\| X_t - \widehat{X}^{\xi}_t \|_{L^p(\Omega)}^2
		\leq 
		\gronw \,
		\sup\nolimits_{v \in [0,t]} 
		\Bigl( 
		4\widehat{b}_T^2 \, \| \Delta^\xi_1(v,\,\cdot\,) \|_{L^1(0,v)}^2
		+ 4\bdg{p} \widehat{\sigma}_T^2 \, \| \Delta^\xi_2(v,\,\cdot\,) \|_{L^2(0,v)}^2 
		\Bigr).
	\end{equation*}
	Since $\kappa$ depends only 
	on $K_1, K_2$ and $T$, but not on $\widehat{K}_1^\xi, \widehat{K}_2^\xi$, 
	the constant $\gronw$ is independent of $\xi\in\Xi$, 
	and taking the square root on both sides 
	completes the proof of the claimed estimate 
	for $C := 2 \bigl( \gronw
	\bigl( \widehat{b}_T^2 \vee \bdg{p} \widehat{\sigma}_T^2 \bigr) 
	\bigr)^{\nicefrac{1}{2}}$. 
\end{proof}

\begin{corollary}\label{cor:strong_error_Lp_conv} 
	Let $p \in [2,\infty)$, and  
	suppose that $X \colon [0,\infty) \times \Omega \to \bbR^m$   
	is a continuous stochastic Volterra process with
	initial value $X_0 \in L^p(\Omega,\cF_0,\bbP)$,
	coefficients $b, \sigma$ fulfilling 
	Assumption~{\upshape\ref{ass:lipschitz_coef}}, 
	and convolution kernels $k_1, k_2$ satisfying
	Assumption~{\upshape\ref{ass:assump_conv_kernels}} with  
	the constants 
	$\eta_T \in (0,\infty)$ 
	and $\alpha_T \in (\nicefrac{1}{p},1]$, 
	for every $T\in(0,\infty)$. 
	
	Let 
	$(\widehat{X}^\xi \colon [0,\infty) \times \Omega \to \bbR^m)_{\xi\in\Xi}$ 
	be a family 
	of continuous stochastic Volterra processes 
	with the same initial value and coefficients as~$X$, 
	whose convolution kernels 
	$(\widehat{k}^\xi_1, \widehat{k}^\xi_2)_{\xi \in \Xi}$
	satisfy Assumption~{\upshape\ref{ass:assu_co_unif_xi}} 
	with the same constants~$(\eta_T)_{T>0}$ 
	and~$(\alpha_T)_{T>0}$. 
	
	For every $T\in(0,\infty)$, 
	there exists a constant $C \in (0,\infty)$ 
	such that, for all $\xi \in \Xi$, 
	\begin{equation*} 
		\forall t\in[0,T]: 
		\quad 
		\| X_t - \widehat{X}^\xi_t \|_{L^p(\Omega)}
		\leq C 
		\bigl( \| k_1 - \widehat{k}^\xi_1 \|_{L^1(0,t)}
		+ \| k_2 - \widehat{k}^\xi_2 \|_{L^2(0,t)} \bigr).
	\end{equation*} 
	In particular, the constant $C \in (0,\infty)$ is independent of $\xi\in\Xi$. 
\end{corollary}

\begin{proof}
	The claim follows by defining 
	the kernels 
	$K_1, K_2$ 
	and $(\widehat{K}^\xi_1, \widehat{K}^\xi_2)_{\xi\in\Xi}$ 
	as in \eqref{eq:induced_conv_volterra_kernel}, 
	see Remark~{\upshape\ref{rmk:conv_kernels}}, 
	and applying 
	Lemma~{\upshape\ref{lma:kernel_assumption_implications}} 
	and Theorem~{\upshape\ref{thm:C0Lp-err_non-conv}} above. 
	Moreover,  
	\begin{equation*} 
		\sup\nolimits_{v \in [0,t]} \| K_j(v,\,\cdot\,) - \widehat{K}^\xi_j(v,\,\cdot\,) \|_{L^j(0,v)} 
		= 
		\sup\nolimits_{v \in [0,t]} \| k_j - \widehat{k}^\xi_j \|_{L^j(0,v)}
		= 
		\| k_j - \widehat{k}^\xi_j \|_{L^j(0,t)} 
	\end{equation*} 
	holds, for $j \in \{1,2\}$, 
	every $\xi\in\Xi$ 
	and all $t\in[0,T]$, 
	which completes the proof.
\end{proof}

Combined with the integral version 
of the Minkowski inequality, 
Corollary~{\upshape\ref{cor:strong_error_Lp_conv}} 
allows to estimate 
differences of paths in the norm 
of $L^p(\Omega;L^q(0,T))$, for~$q\in[1,p]$. 

\begin{corollary}\label{cor:path_dependent_Lp}
	Let $p \in [2,\infty)$, and  
	suppose that $X \colon [0,\infty) \times \Omega \to \bbR^m$   
	is a continuous stochastic Volterra process with
	initial value $X_0 \in L^p(\Omega,\cF_0,\bbP)$,
	coefficients $b, \sigma$ fulfilling 
	Assumption~{\upshape\ref{ass:lipschitz_coef}}, 
	and convolution kernels $k_1, k_2$ satisfying
	Assumption~{\upshape\ref{ass:assump_conv_kernels}} with  
	the constants 
	$\eta_T \in (0,\infty)$ 
	and $\alpha_T \in (\nicefrac{1}{p},1]$, 
	for every $T\in(0,\infty)$. 
	
	Let 
	$(\widehat{X}^\xi \colon [0,\infty) \times \Omega \to \bbR^m)_{\xi\in\Xi}$ 
	be a family 
	of continuous stochastic Volterra processes 
	with the same initial value and coefficients as~$X$, 
	whose convolution kernels 
	$(\widehat{k}^\xi_1, \widehat{k}^\xi_2)_{\xi \in \Xi}$
	satisfy Assumption~{\upshape\ref{ass:assu_co_unif_xi}} 
	with the same constants~$(\eta_T)_{T>0}$ 
	and~$(\alpha_T)_{T>0}$.   
	
	Then, for every $T\in(0,\infty)$, 
	there exists a constant $C \in (0,\infty)$ 
	such that, for all $q\in[1,p]$ 
	and every $\xi \in \Xi$,  
	\begin{align*}
		\| X-\widehat{X}^\xi &\|_{L^p(\Omega;L^q(0,T))}
		= 
		\Biggl( 
		\bbE\Biggl[\biggl| \int_0^T \| X_u - \widehat{X}^\xi_u \|_{\bbR^m}^q \diff{u} \biggr|^{\nicefrac{p}{q}}\Biggr] 
		\Biggr)^{\nicefrac{1}{p}} \\
		&\leq 
		C 
		\Biggl[ 
		\biggl( \int_0^T \| k_1-\widehat{k}^\xi_1 \|_{L^1(0,u)}^q \diff{u} \biggr)^{\nicefrac{1}{q}}
		+ 
		\biggl( \int_0^T \| k_2-\widehat{k}^\xi_2 \|_{L^2(0,u)}^q \diff{u} \biggr)^{\nicefrac{1}{q}} 
		\Biggr].
	\end{align*}
	In particular, the constant $C\in (0,\infty)$ is independent of 
	$q\in[1,p]$ and $\xi\in\Xi$. 
\end{corollary}

\begin{proof} 
	Let $T\in(0,\infty)$, 
	$q\in[1,p]$, and $\xi\in\Xi$.  
	Then, by Corollary~{\upshape\ref{cor:strong_error_Lp_conv}}, 
	there exists a constant $C \in (0,\infty)$, 
	independent of $\xi\in\Xi$ and $q\in[1,p]$,  such that 
	\begin{equation*}
		\forall u\in[0,T]: 
		\quad 
		\| X_u - \widehat{X}^\xi_u \|_{L^p(\Omega)} 
		\leq 
		C  
		\bigl( \| k_1-\widehat{k}^\xi_1 \|_{L^1(0,u)} + \| k_2-\widehat{k}^\xi_2 \|_{L^2(0,u)} \bigr).
	\end{equation*}
	Combining this bound with 
	Minkowski's integral inequality \cite[Proposition~1.2.22]{analysis_banach_I}, 
	\begin{equation*}
		\| X - \widehat{X}^\xi \|_{L^p(\Omega;L^q(0,T))}
		\leq 
		\| X - \widehat{X}^\xi \|_{L^q(0,T;L^p(\Omega))}
		= 
		\biggl( 
		\int_0^T \| X_u - \widehat{X}^\xi_u \|_{L^p(\Omega)}^{q} \diff{u} 
		\biggr)^{\nicefrac{1}{q}}, 
	\end{equation*} 
	and with the triangle inequality 
	on $L^q(0,T)$ 
	shows the claim. 
\end{proof}


\subsection{Differences of Volterra processes in H\"older norms}
\label{subsec:general_SVE_holder_diff}

In this subsection, 
we bound the difference process 
$X - \widehat{X}^\xi$  
in H\"older sense, i.e., in the norms 
of $C^\nu([0,T];L^p(\Omega))$ 
and 
$L^p(\Omega;C^\nu([0,T]))$,  
for $p\in[2,\infty)$, and $\nu\geq 0$ 
up to the maximal\pagebreak  

\noindent 
H\"older regularity  
permitted 
by $X$ and $(\widehat{X}^\xi)_{\xi\in\Xi}$. 
For this purpose, 
the next~proposition,  
that bounds the increments 
of the difference process  
in the norm of $L^p(\Omega)$ 
by quantities 
depending only on the kernel differences
$\Delta^\xi_1, \Delta^\xi_2$ 
in \eqref{eq:delta_definition}, 
will be key.  

\begin{proposition}\label{prop:Lp_error_of_increment}
	Let $p \in [2,\infty)$, and  
	suppose that ${X \colon [0,\infty) \times \Omega \to \bbR^m}$   
	is a continuous stochastic Volterra process with
	initial value $X_0 \in L^p(\Omega,\cF_0,\bbP)$,
	coefficients~$b, \sigma$ fulfilling 
	Assumption~{\upshape\ref{ass:lipschitz_coef}} 
	with $L \in (0,\infty)$,
	and kernels $K_1, K_2$ satisfying
	Assumption~{\upshape\ref{ass:assump_nonconv_kernels}} 
	with the constants 
	$\eta_T \in (0,\infty)$,  
	$\alpha_T \in (\nicefrac{1}{p},1]$, 
	for every $T\in(0,\infty)$. 
	
	Assume that 
	$(\widehat{X}^\xi \colon [0,\infty) \times \Omega \to \bbR^m)_{\xi\in\Xi}$ 
	is a family 
	of continuous stochastic Volterra processes 
	with the same initial value and coefficients as~$X$, 
	whose kernels  
	$(\widehat{K}^\xi_1, \widehat{K}^\xi_2)_{\xi \in \Xi}$
	satisfy Assumption~{\upshape\ref{ass:assu_noco_unif_xi}} 
	with the same constants 
	$(\eta_T)_{T>0}$ 
	and $(\alpha_T)_{T>0}$. 
	For each $\xi\in\Xi$,  
	let $\Delta^\xi_1 \colon \triangle[0,\infty)\to\bbR$ 
	and $\Delta^\xi_2 \colon \triangle[0,\infty)\to\bbR$ 
	be defined  
	as in \eqref{eq:delta_definition}.
	
	For every $T\in(0,\infty)$, 
	there exists a constant $C \in (0,\infty)$ 
	such that, for all $\xi \in \Xi$ 
	and all ${(t,s)\in\triangle [0,T]}$,  
	\begin{align*}   
			\| (X_t - \widehat{X}^\xi_t) - (&X_s - \widehat{X}^\xi_s) \|_{L^p(\Omega)}
			\\
			&\;\;\leq 
			C \biggl[
			\sup_{v \in [0,t]} 
			\Bigl( \| \Delta^\xi_1(v,\,\cdot\,) \|_{L^1(0,v)}
			+ \| \Delta^\xi_2(v,\,\cdot\,) \|_{L^2(0,v)} \Bigr) (t-s)^{\alpha}
			\\
			&\qquad\qquad 
			+ \| \Delta^\xi_1(t,\,\cdot\,) \|_{L^1(s,t)}
			+ \| \Delta^\xi_1(t,\,\cdot\,) - \Delta^\xi_1(s,\,\cdot\,) \|_{L^1(0,s)} 
			\\ 
			&\qquad\qquad 
			+ \| \Delta^\xi_2(t,\,\cdot\,) \|_{L^2(s,t)} 
			+ \| \Delta^\xi_2(t,\,\cdot\,) - \Delta^\xi_2(s,\,\cdot\,) \|_{L^2(0,s)}
			\biggr]. 
	\end{align*}
\end{proposition}

\begin{proof}
	Let $T\in(0,\infty)$, 
	$\xi\in\Xi$, and $(t,s) \in \triangle[0,T]$, 
	and set $\eta:=\eta_T \in (0,\infty)$ and 
	${\alpha:=\alpha_T \in (\nicefrac{1}{p},1]}$. 
	We consider the increments 
	$X_{t,s} := X_t - X_s$ 
	and 
	$\widehat{X}^{\xi}_{t,s} := \widehat{X}^{\xi}_t - \widehat{X}^{\xi}_s$. 
	For $j \in \{1,2\}$,
	define the functions 
	$K_{j;t,s}, \widehat{K}^\xi_{j;t,s}$
	as in \eqref{eq:K_jts_definition} 
	and $\Delta^\xi_j, \Delta^\xi_{j;t,s}$
	as in~\eqref{eq:delta_definition}.
	By expressing the increments of the Volterra processes~$X$ and~$\widehat{X}^\xi$ 
	as in \eqref{eq:SVE_increment_decomposition},  
	the difference of the increments  
	can be split into four integrals, 
	and we obtain that 
	\begin{align}
		\| X_{t,s} - \widehat{X}^{\xi}_{t,s} \|_{L^p(\Omega)}^2 
		&\leq 
		4\, 
		\biggl\| \int_s^t \left[K_1(t,u) b(X_u) - \widehat{K}^{\xi}_1(t,u) b(\widehat{X}^{\xi}_u)\right] \diff{u} \biggr\|_{L^p(\Omega)}^2 
		\notag
		\\
		&\qquad+ 
		4\, 
		\biggl\| \int_s^t \left[K_2(t,u) \sigma(X_u) - \widehat{K}^{\xi}_2(t,u) \sigma(\widehat{X}^{\xi}_u)\right] \diff{W}_u 
		\biggr\|_{L^p(\Omega)}^2 
		\notag
		\\
		&\qquad+ 
		4\, 
		\biggl\| \int_0^s \left[K_{1;t,s}(u) b(X_u) - \widehat{K}^{\xi}_{1;t,s}(u) b(\widehat{X}^{\xi}_u) \right] \diff{u} 
		\biggr\|_{L^p(\Omega)}^2 
		\notag
		\\
		&\qquad+ 
		4\, 
		\biggl\| \int_0^s \left[K_{2;t,s}(u) \sigma(X_u) - \widehat{K}^{\xi}_{2;t,s}(u) \sigma(\widehat{X}^{\xi}_u) \right] \diff{W}_u 
		\biggr\|_{L^p(\Omega)}^2 
		\notag 
		\\
		&=: 4 \bigl( \mathrm{(I)} + \mathrm{(II)} + \mathrm{(III)} + \mathrm{(IV)} \bigr) , 
		\label{eq:Holder_diff_whole} 
	\end{align}
	where we also used the triangle and Cauchy--Schwarz inequalities. 
	We will estimate each of the 
	four terms in \eqref{eq:Holder_diff_whole} separately. 
	To this end, 
	we define the constants $\widehat{b}_T, \widehat{\sigma}_T \in [0,\infty)$
	as in \eqref{eq:bhat_sigmahat}, 
	see Remark~{\upshape\ref{rmk:def_bT-hat_sigmaT-hat}}, 
	and we recall that by Theorem~{\upshape\ref{thm:C0Lp-err_non-conv}}, 
	\begin{equation}\label{eq:def_Theta_t_xi} 
		\begin{split} 
			\forall t \in[0,T]: 
			\quad 
			&\sup\limits_{v \in [0,t]} \| X_v - \widehat{X}^\xi_v \|_{L^p(\Omega)}
			\leq 
			\Theta_t(\xi), 
			\\
			\Theta_t(\xi) 
			:= 
			\widehat{C}
			\bigl(\, \sup\nolimits_{v \in [0,t]} 
			&\| \Delta^\xi_1(v,\,\cdot\,) \|_{L^1(0,v)}
			+ \sup\nolimits_{v \in [0,t]} \| \Delta^\xi_2(v,\,\cdot\,) \|_{L^2(0,v)} \bigr), 
		\end{split} 
	\end{equation}
	where $\widehat{C} \in (0,\infty)$ is a constant, 
	independent of $\xi\in\Xi$.   
	
	\textbf{Term~$\boldsymbol{\mathrm{(I)}}$:} 
	Using the identity observed in \eqref{eq:decomposition_for_X_minus_Xhat}, 
	we find that 
	\begin{align*}
		\mathrm{(I)} 
		&\leq  
		2\, 
		\biggl\|  \int_s^t \Delta^{\xi}_1(t,u) b(\widehat{X}^\xi_u) \diff{u}  \biggr\|_{L^p(\Omega)}^2 
		+
		2\, 
		\biggl\|  \int_s^t K_1(t,u) \bigl( b(X_{u}) - b(\widehat{X}^{\xi}_{u}) \bigr)\diff{u}  \biggr\|_{L^p(\Omega)}^2 
		\\
		&=: 
		2 \bigl( \mathrm{(Ia)} + \mathrm{(Ib)} \bigr) , 
	\end{align*}
	where we readily can bound the term $\mathrm{(Ia)}$ 
	using \eqref{eq:bhat_sigmahat} 
	as follows, 
	\begin{equation*}
		\mathrm{(Ia)} 
		\leq 
		\biggl(\int_s^t | \Delta^{\xi}_1(t,u) | \| b(\widehat{X}^\xi_u) \|_{L^p(\Omega)} \diff{u} \biggr)^2  
		\leq 
		\widehat{b}_T^2 \| \Delta^{\xi}_1(t,\,\cdot\,) \|_{L^1(s,t)}^2.
	\end{equation*}
	For term $\mathrm{(Ib)}$, we 
	find by Lipschitz continuity of~$b$, 
	postulated in Assumption~{\upshape\ref{ass:lipschitz_coef}}, 
	that   
	\begin{align*}
		\mathrm{(Ib)} 
		&\leq  
		\biggl(
		\int_s^t |K_1(t,u)| \, L \, 
		\| X_u-\widehat{X}^{\xi}_u \|_{L^p(\Omega)}  
		\diff{u} \biggr)^2 
		\\
		&\leq 
		L^2 \sup_{v\in[s,t]}  \| X_v - \widehat{X}^{\xi}_v \|_{L^p(\Omega)}^2
		\| K_1(t, \,\cdot\,)\|_{L^1(s,t)}^2 
		\leq  
 		L^2 \eta^2 \Theta_t^2(\xi) (t-s)^{2\alpha} , 
	\end{align*}
	where the last step follows from \eqref{eq:def_Theta_t_xi} and 
	Assumption~{\upshape\ref{ass:assump_nonconv_kernels}\ref{item:assu_noco_i}}.
	
	\textbf{Term~$\boldsymbol{\mathrm{(II)}}$:} 
	We proceed similarly as for term $\mathrm{(I)}$: 
	We first split it as follows, 
	\begin{align*}
		\mathrm{(II)} 
		&\leq 
		2 \biggl\| \int_s^t \Delta^{\xi}_2(t,u) \sigma(\widehat{X}^\xi_u) \diff{W}_u  \biggr\|_{L^p(\Omega)}^2 \!\! 
		+ 
		2 \biggl\| \int_s^t K_2(t,u) \bigl( \sigma(X_u) - \sigma(\widehat{X}^{\xi}_u) \bigr) \diff{W}_u \biggr\|_{L^p(\Omega)}^2  
		\\
		&=: 
		2 \bigl( \mathrm{(IIa)} + \mathrm{(IIb)} \bigr), 
	\end{align*}
	and then apply the 
	Burkholder--Davis--Gundy inequality 
	(with constant $\bdg{p} \in(0,\infty)$, 
	see Lemma~{\upshape\ref{lma:BDG_SVE}})  
	and \eqref{eq:bhat_sigmahat} 
	to estimate the 
	$L^p(\Omega)$-norm of 
	the stochastic integrals, 
	\begin{align*}
		\mathrm{(IIa)} 
		&\leq 
		\bdg{p} \int_s^t | \Delta^{\xi}_2(t,u) |^2 \| \sigma(\widehat{X}^\xi_u) \|_{L^p(\Omega)}^2 \diff{u} 
		\leq \bdg{p} \widehat{\sigma}_T^2 \| \Delta^{\xi}_2(t,\,\cdot\,) \|_{L^2(s,t)}^2 , 
		\\
		\mathrm{(IIb)}  
		&\leq  
		\bdg{p} L^2 
		\int_s^t |K_2(t,u)|^2  
		\| X_u - \widehat{X}^{\xi}_u \|_{L^p(\Omega)}^2 \diff{u} 
		\\
		&\leq 
		\bdg{p} L^2 
		\sup_{v\in[s,t]}  \| X_v - \widehat{X}^{\xi}_v \|_{L^p(\Omega)}^2 
		\| K_2(t,\,\cdot\,) \|_{L^2(s,t)}^2
		\leq 
		\bdg{p} 
		L^2 \eta^2 \Theta_t^2(\xi) (t-s)^{2\alpha} ,  
	\end{align*}
	where we also used 
	Assumptions~{\upshape\ref{ass:lipschitz_coef}},~{\upshape\ref{ass:assump_nonconv_kernels}\ref{item:assu_noco_i}} 
	and \eqref{eq:def_Theta_t_xi} to bound $\mathrm{(IIb)}$. 
	
	\textbf{Term~$\boldsymbol{\mathrm{(III)}}$:} 
	We first note the identity 
	\begin{equation}\label{eq:decomposition_Kb-minus_Khbh}
		\begin{split} 
			K_{1;t,s}(u) b(X_u) - \widehat{K}^{\xi}_{1;t,s}&(u) b(\widehat{X}^{\xi}_u)
			\\
			&= 
			\Delta^{\xi}_{1;t,s}(u) b(\widehat{X}^{\xi}_u)  
			+ 
			K_{1;t,s}(u) \bigl(b(X_u) - b(\widehat{X}^{\xi}_u) \bigr) , 
		\end{split} 
	\end{equation}
	and use it to bound the third term as follows, 
	\begin{align*}
		\mathrm{(III)} 
		&=  
		\biggl\| 
		\int_0^s \left[ \Delta^{\xi}_{1;t,s}(u) b(\widehat{X}^{\xi}_u) + 
		K_{1;t,s}(u) \bigl(b(X_u) - b(\widehat{X}^{\xi}_u) \bigr) 
		\right] \diff{u} 
		\biggr\|_{L^p(\Omega)}^2
		\\ 
		&\leq 
		2 \, 
		\biggl\| 
		\int_0^s \Delta^{\xi}_{1;t,s}(u) b(\widehat{X}^{\xi}_u) \diff{u} 
		\biggr\|_{L^p(\Omega)}^2
		+ 
		2 \, 
		\biggl\| 
		\int_0^s  K_{1;t,s}(u) \bigl(b(X_u) - b(\widehat{X}^{\xi}_u) \bigr) \diff{u} 
		\biggr\|_{L^p(\Omega)}^2 
		\\
		&=: 
		2 \bigl( \mathrm{(IIIa)} + \mathrm{(IIIb)} \bigr)  . 
	\end{align*} 
	Then, by  
	\eqref{eq:bhat_sigmahat} 
	and 
	\eqref{eq:def_Theta_t_xi} 
	as well as 
	Assumptions~{\upshape\ref{ass:lipschitz_coef}} 
	and~{\upshape\ref{ass:assump_nonconv_kernels}\ref{item:assu_noco_ii}}, 
	we conclude that 
	\begin{equation*} 
		\mathrm{(IIIa)} 
		\leq 
		\biggl( \int_0^s | \Delta^{\xi}_{1;t,s}(u) | \| b(\widehat{X}^{\xi}_u) \|_{L^p(\Omega)} \diff{u} \biggr)^2  
		\leq 
		\widehat{b}_T^2 \| \Delta^{\xi}_{1;t,s} \|_{L^1(0,s)}^2,
	\end{equation*}
	\begin{align*} 
		\mathrm{(IIIb)} 
		&\leq 
		\biggl( \int_0^s| K_{1;t,s}(u) | \| b(X_u) - b(\widehat{X}^{\xi}_u) \|_{L^p(\Omega)} \diff{u} \biggr)^2  
		\\
		&\leq 
		L^2 \sup_{v \in [0,s]} \| X_v - \widehat{X}^{\xi}_v \|_{L^p(\Omega)}^2 \| K_{1;t,s} \|_{L^1(0,s)}^2 
		\leq 
		L^2 \eta^2 \Theta_s^2(\xi) (t-s)^{2\alpha}. 
	\end{align*} 
	
	\textbf{Term~$\boldsymbol{\mathrm{(IV)}}$:} 
	In analogy to \eqref{eq:decomposition_Kb-minus_Khbh}, 
	we split the fourth term, 
	\begin{align*}
		\mathrm{(IV)} 
		&\leq  
		2\, 
		\biggl\| 
		\int_0^s 
		\Delta^{\xi}_{2;t,s}(u)  
		\sigma(\widehat{X}^{\xi}_u) \diff{W}_u 
		\biggr\|_{L^p(\Omega)}^2 
		\\
		&\qquad + 
		2\, 
		\biggl\|  
		\int_0^s K_{2;t,s}(u) \bigl( \sigma(X_u) - \sigma(\widehat{X}^{\xi}_u) \bigr) 
		\diff{W}_u 
		\biggr\|_{L^p(\Omega)}^2 
		=: 
		2 \bigl( \mathrm{(IVa)} + \mathrm{(IVb)} \bigr) , 
	\end{align*} 
	and bound the remaining terms 
	by means of the Burkholder--Davis--Gundy 
	inequality, 
	\eqref{eq:bhat_sigmahat} 
	and 
	\eqref{eq:def_Theta_t_xi} 
	as well as 
	Assumptions~{\upshape\ref{ass:lipschitz_coef}} 
	and~{\upshape\ref{ass:assump_nonconv_kernels}\ref{item:assu_noco_ii}}, 
	\begin{align*} 
		\mathrm{(IVa)}
		&\leq 
		\bdg{p} \int_0^s | \Delta^{\xi}_{2;t,s}(u) |^2 \| \sigma(\widehat{X}^{\xi}_u) \|_{L^p(\Omega)}^2 \diff{u} 
		\leq 
		\bdg{p} \widehat{\sigma}_T^2 \| \Delta^{\xi}_{2;t,s} \|_{L^2(0,s)}^2, 
		\\
		\mathrm{(IVb)}
		&\leq 
		\bdg{p} \int_0^s | K_{2;t,s}(u) |^2 \| \sigma(X_u) - \sigma(\widehat{X}^{\xi}_u) \|_{L^p(\Omega)}^2 \diff{u} 
		\\
		&\leq 
		\bdg{p} L^2 \sup_{v \in [0,s]} \| X_v - \widehat{X}^{\xi}_v \|_{L^p(\Omega)}^2 \| K_{2;t,s} \|_{L^2(0,s)}^2 
		\leq 
		\bdg{p} L^2 \eta^2 \Theta_s^2(\xi)(t-s)^{2\alpha} . 
	\end{align*}  
	By recalling $\Theta_t(\xi)$
	from \eqref{eq:def_Theta_t_xi} 
	and noting that $\Theta_s(\xi) \leq \Theta_t(\xi)$,   
	we can summarize the bounds 
	for the terms (I)--(IV) in \eqref{eq:Holder_diff_whole}, 
	\begin{align*}
		\| X_{t,s} - \widehat{X}^{\xi}_{t,s} \|_{L^p(\Omega)}^2
		&\leq 
		8 \, \Bigl[
		2 L^2 \eta^2 (1+\bdg{p}) \Theta_t^2(\xi)   
		(t-s)^{2\alpha} 
		\\
		&\qquad\, 
		+ \widehat{b}_T^2 
		\, \Bigl( \| \Delta^\xi_1(t,\,\cdot\,) \|_{L^1(s,t)}^2
		+  \| \Delta^\xi_1(t,\,\cdot\,) - \Delta^\xi_1(s,\,\cdot\,) \|_{L^1(0,s)}^2 \Bigr)
		\\
		&\qquad\, 
		+ \bdg{p} \widehat{\sigma}_T^2 \, 
		\Bigl( \| \Delta^\xi_2(t,\,\cdot\,) \|_{L^2(s,t)}^2
		+  \| \Delta^\xi_2(t,\,\cdot\,) - \Delta^\xi_2(s,\,\cdot\,) \|_{L^2(0,s)}^2 \Bigr)
		\Bigr], 
	\end{align*} 
	and the claim follows by taking the square root 
	on both sides, as well as
	subadditivity of the square root function.  
\end{proof}

\begin{corollary}\label{cor:Lp_error_of_increment_conv}  
	Let $p \in [2,\infty)$, and  
	suppose that $X \colon [0,\infty) \times \Omega \to \bbR^m$   
	is a continuous stochastic Volterra process with
	initial value $X_0 \in L^p(\Omega,\cF_0,\bbP)$,
	coefficients $b, \sigma$ fulfilling 
	Assumption~{\upshape\ref{ass:lipschitz_coef}}, 
	and convolution kernels $k_1, k_2$ satisfying
	Assumption~{\upshape\ref{ass:assump_conv_kernels}} with  
	the constants 
	$\eta_T \in (0,\infty)$ 
	and $\alpha_T \in (\nicefrac{1}{p},1]$, 
	for every $T\in(0,\infty)$. 
	
	Let 
	$(\widehat{X}^\xi \colon [0,\infty) \times \Omega \to \bbR^m)_{\xi\in\Xi}$ 
	be a family 
	of continuous stochastic Volterra processes 
	with the same initial value and coefficients as~$X$, 
	whose convolution kernels 
	$(\widehat{k}^\xi_1, \widehat{k}^\xi_2)_{\xi \in \Xi}$
	satisfy Assumption~{\upshape\ref{ass:assu_co_unif_xi}} 
	with the same constants~$(\eta_T)_{T>0}$ 
	and~$(\alpha_T)_{T>0}$.    
	For every $\xi\in\Xi$,  
	let $\Delta^\xi_1 \colon [0,T]\to\bbR$ 
	and $\Delta^\xi_2 \colon [0,T]\to\bbR$ 
	be defined  
	as in \eqref{eq:delta_definition_conv}. 
	
	For every $T\in(0,\infty)$, 
	there exists a constant $C \in (0,\infty)$ 
	such that, for all $\xi \in \Xi$ 
	and all ${(t,s)\in\triangle [0,T]}$,  
	\begin{align*}
		\| (X_t &- \widehat{X}^\xi_t) - (X_s - \widehat{X}^\xi_s) \|_{L^p(\Omega)}
		\\
		&\leq 
		C \Bigl[
		\Bigl( 
		\| \Delta^\xi_1 \|_{L^1(0,t)}
		+ \| \Delta^\xi_2 \|_{L^2(0,t)} \Bigr) (t-s)^{\alpha}
		+ \| \Delta^\xi_1 \|_{L^1(0,t-s)}
		+ \| \Delta^\xi_2 \|_{L^2(0,t-s)}
		\\ 
		&\qquad 
		+ 
		\| \Delta^\xi_1(t-\,\cdot\,) - \Delta^\xi_1(s-\,\cdot\,) \|_{L^1(0,s)}
		+  
		\| \Delta^\xi_2(t-\,\cdot\,) - \Delta^\xi_2(s-\,\cdot\,) \|_{L^2(0,s)}
		\Bigr].
	\end{align*} 
\end{corollary}

\begin{proof}
	The claim follows from 
	Proposition~{\upshape\ref{prop:Lp_error_of_increment}} 
	since, for $j \in \{1,2\}$ and for every $(v,u)\in\triangle[0,\infty)$, 
	we have that 
	$\| \Delta^\xi_j(v-\,\cdot\,) \|_{L^j(u,v)}
	= 
	\| \Delta^\xi_j \|_{L^j(0,v-u)}$.  
\end{proof} 

The last two terms in 
the bound of Corollary~{\upshape\ref{cor:Lp_error_of_increment_conv}} 
cannot readily be simplified. 
While $k_1, k_2$ and $\widehat{k}_1^\xi, \widehat{k}_2^\xi$ 
are each individually monotone 
by Assumptions~{\upshape\ref{ass:assump_conv_kernels}} 
and~{\upshape\ref{ass:assu_co_unif_xi}}, 
their differences 
$\Delta^\xi_1$ and $\Delta^\xi_2$ may not. 
These terms 
will be addressed in 
Lemma~{\upshape\ref{lma:kernel_increments_holder_and_rate}} 
in the setting specified in the following assumption.  

\begin{assumption}[Convolution kernel 
	approximations II]\label{ass:convker_holder_assu} 
	Suppose that the functions 
	$k_1, k_2 \colon (0,\infty) \to \bbR$  
	satisfy Assumption~{\upshape\ref{ass:assump_conv_kernels}} 
	with the constants $\eta_T \in (0,\infty)$ 
	and $\alpha_T \in (0,1]$, 
	for all $T\in(0,\infty)$.
	Let $(\widehat{k}^\xi_1,\widehat{k}^\xi_2)_{\xi \in \Xi}$ 
	be a family of functions
	satisfying Assumption~{\upshape\ref{ass:assu_co_unif_xi}} 
	with the same constants $\eta_T$ 
	and $\alpha_T$, 
	for all $T\in(0,\infty)$. 
	In addition, assume that, 
	for every $T \in(0,\infty)$, 
	there exist functions  
	${\theta_T \colon \Xi \to (0,\infty)}$, 
	${g_T\in L^1(0,T)}$,  
	and 
	constants 
	$\beta=\beta_T \in (0,1]$, 
	${\eta' = \eta'_T \in(0,\infty)}$ 
	such that:   
	\begin{enumerate}[leftmargin=8mm, label={\upshape(\roman*)}] 
		\item\label{item:diff_diff_holder_conv_i}
		For every $\xi\in\Xi$,  
		\begin{equation*}
			\qquad 
			\forall t \in [0,T]: 
			\quad 
			\| k_1 - \widehat{k}^\xi_1 \|_{L^1(0,t)} \leq \theta_T(\xi) \, t^{\beta},
			\qquad
			\| k_2 - \widehat{k}^\xi_2 \|_{L^2(0,t)} \leq \theta_T(\xi) \, t^{\beta}. 
		\end{equation*}
		\item\label{item:diff_diff_holder_conv_ii}
		For every $\xi\in\Xi$,   
		$k_1-\widehat{k}^\xi_1$ and $k_2-\widehat{k}^\xi_2$
		are weakly differentiable on $(0,T)$ with 
		\begin{equation*}
			\qquad 
			\| (k_1 - \widehat{k}^\xi_1)' \|_{L^1(u,T)} \leq \theta_T(\xi) \, g_T(u),
			\qquad
			\| (k_2 - \widehat{k}^\xi_2)' \|_{L^2(u,T)} \leq \theta_T(\xi) \, g_T(u),
		\end{equation*}
		holding for almost all $u \in (0,T)$, and 
		\begin{equation*}
			\qquad 
			\forall t \in [0, T]: \quad
			\| g_T \|_{L^1(0,t)} \leq \eta' \, t^{\beta}. 
		\end{equation*}
	\end{enumerate}  
\end{assumption}

\begin{lemma}
\label{lma:kernel_increments_holder_and_rate}
	Suppose that $k_1,k_2$ and the family 
	$(\widehat{k}^\xi_1,\widehat{k}^\xi_2)_{\xi \in \Xi}$
	are functions satisfying 
	Assumption~{\upshape\ref{ass:convker_holder_assu}\ref{item:diff_diff_holder_conv_ii}}, 
	with constants 
	$\eta_T' \in (0,\infty)$, 
	$\beta_T \in (0,1]$,
	and functions
	$\theta_T \colon \Xi \to (0,\infty)$,
	$g_T \in L^1(0,T)$, 
	for every $T\in(0,\infty)$. 
	In addition, for every $\xi\in\Xi$,  
	let $\Delta^\xi_1 \colon [0,\infty)\to\bbR$ 
	and $\Delta^\xi_2 \colon [0,\infty)\to\bbR$ 
	be defined  
	as in \eqref{eq:delta_definition_conv}.  
	
	Then, for all $T\in(0,\infty)$, 
	for $j \in \{1,2\}$, 
	and for every~$\xi \in \Xi$,   
	we have that 
	\begin{equation*} 
		\forall (t,s) \in \triangle[0,T]: 
		\quad 
		\| \Delta^\xi_j(t-\,\cdot\,) - \Delta^\xi_j(s-\,\cdot\,) \|_{L^j(0,s)}
		\leq 
		\eta'_T \, \theta_T(\xi) \, (t-s)^{\beta_T}.
	\end{equation*} 
\end{lemma}

\begin{proof}
	Let $T\in(0,\infty)$, 
	$\xi \in \Xi$, $j \in \{1,2\}$, 
	$(t,s) \in \triangle (0,T]$. 
	Then, for all $v \in (0,s)$, 
	we have that 
	$(s-v, t-v) \subseteq (0,T)$, 
	and 
	using Assumption~{\upshape\ref{ass:convker_holder_assu}\ref{item:diff_diff_holder_conv_ii}}, 
	we find that 
	\begin{align*}
		&\| \Delta^\xi_j (t-\,\cdot\,) - \Delta^\xi_j(s-\,\cdot\,) \|_{L^j(0,s)} 
		= 
		\left\| \int_{s-\,\cdot\,}^{t-\,\cdot\,} (\Delta^\xi_j)'(u) \diff{u} \right\|_{L^j(0,s)}
		\\ 
		&\qquad\leq   
		\left\| \int_0^{t-s} | (\Delta^\xi_j)'(u - \,\cdot\, + s) | \diff{u} \right\|_{L^j(0,s)} 
		\leq  
		\int_0^{t-s} \| (\Delta^\xi_j)'(u-\,\cdot\,+s) \|_{L^j(0,s)} \diff{u}, 
	\end{align*}
	where the last step holds by Minkowski's integral inequality \cite[Proposition~1.2.22]{analysis_banach_I}.
	Again, by a change of variables and 
	Assumption~{\upshape\ref{ass:convker_holder_assu}\ref{item:diff_diff_holder_conv_ii}}, 
	we obtain that 
	\begin{align*}
		&\int_0^{t-s} \| (\Delta^\xi_j)'(u-\,\cdot\,+s) \|_{L^j(0,s)} \diff{u} 
		=  
		\int_0^{t-s} 
		\| (k_j - \widehat{k}^\xi_j)' \|_{L^j(u,u+s)} \diff{u} 
		\\ 
		&\quad 
		\leq 
		\int_0^{t-s} 
		\| (k_j - \widehat{k}^\xi_j)' \|_{L^j(u,T)} \diff{u} 
		\leq 
		\theta_T(\xi) \int_0^{t-s} | g_T(u) | \diff{u} 
		\leq 
		\eta_T' \theta_T(\xi) \, (t-s)^{\beta_T}.\;  \qedhere
	\end{align*}
\end{proof} 

As stipulated in 
Assumption~{\upshape\ref{ass:convker_holder_assu}\ref{item:diff_diff_holder_conv_i}}, 
the function $\theta_T \colon\Xi \to (0,\infty)$ 
measures the accuracy 
of the kernel approximations $\widehat{k}_1^\xi, \widehat{k}_2^\xi$  
in the norms of $L^1(0,T)$ and $L^2(0,T)$, respectively. 
Using this in  
the bound of Corollary~{\upshape\ref{cor:strong_error_Lp_conv}} 
we conclude that 
\begin{equation}\label{eq:C0Lp_theta_bd}
	\| X - \widehat{X}^\xi \|_{C^0([0,T];L^p(\Omega))}
	\leq 
	2 C T^\beta  \theta_T(\xi) . 
\end{equation}
Furthermore, 
Assumption~{\upshape\ref{ass:convker_holder_assu}\ref{item:diff_diff_holder_conv_i}} 
facilitates estimating the accuracy 
of the first four terms 
in the bound of 
Corollary~{\upshape\ref{cor:Lp_error_of_increment_conv}}. 
Finally, 
as seen in 
Lemma~{\upshape\ref{lma:kernel_increments_holder_and_rate}} above, 
under 
Assumption~{\upshape\ref{ass:convker_holder_assu}\ref{item:diff_diff_holder_conv_ii}} 
also the last two terms  
of Corollary~{\upshape\ref{cor:Lp_error_of_increment_conv}} 
can be bounded. 
In the next theorem, 
we summarize these observations, 
and obtain a result on the accuracy 
and convergence with respect to $\xi$ 
in (pathwise) H\"older norms.

\begin{theorem}\label{thm:SVE_holder_Lp_bound}
	Let $p \in [2,\infty)$, 
	and suppose that $k_1,k_2$ and 
	the family 
	$(\widehat{k}^\xi_1,\widehat{k}^\xi_2)_{\xi \in \Xi}$
	are functions fulfilling 
	Assumption~{\upshape\ref{ass:convker_holder_assu}}, 
	with constants
	$\eta_T,\eta_T' \in (0,\infty)$, 
	$\beta_T \in (0,1]$,
	${\alpha_T  \in (\nicefrac{1}{p},1]}$,
	and functions
	$\theta_T \colon \Xi \to (0,\infty)$,
	$g_T\in L^1(0,T)$,  
	for every $T\in(0,\infty)$. 
	Assume 
	that $X \colon [0,\infty) \times \Omega \to \bbR^m$   
	is a continuous stochastic Volterra process with
	initial value $X_0 \in L^p(\Omega,\cF_0,\bbP)$,
	coefficients $b, \sigma$ satisfying
	Assumption~{\upshape\ref{ass:lipschitz_coef}},  
	and convolution kernels $k_1, k_2$.  
	Let 
	$(\widehat{X}^\xi \colon [0,\infty) \times \Omega \to \bbR^m)_{\xi\in\Xi}$ 
	be a family 
	of continuous stochastic Volterra processes 
	with the same initial value and coefficients as~$X$, 
	and with convolution kernels $(\widehat{k}^\xi_1, \widehat{k}^\xi_2)_{\xi \in \Xi}$.  
	
	Then, for every $T\in(0,\infty)$, 
	and for $\alpha:=\alpha_T\in(\nicefrac{1}{p},1]$, 
	$\beta:=\beta_T\in(0,1]$, 
	there exists a constant $C \in (0,\infty)$ 
	such that, for all $\xi \in \Xi$,
	\begin{equation}\label{eq:thm:SVE_holder_Lp_bound}
		\| X - \widehat{X}^\xi \|_{C^{\alpha \wedge \beta}([0,T];L^p(\Omega))}
		\leq 
		C \, \theta_T(\xi) . 
	\end{equation}
	
	Moreover, if $\beta > \nicefrac{1}{p}$, 
	then, for every 
	$T\in(0,\infty)$ 
	and all 
	$\nu \in (0,(\alpha \wedge \beta) - \nicefrac{1}{p})$,
	there exists a constant $C_\nu\in(0,\infty)$ such that, 
	for all $\xi\in\Xi$, 
	\begin{equation}\label{eq:kolm-chents_difference}
		\| X - \widehat{X}^\xi \|_{L^p(\Omega;C^\nu([0,T]))}  
		\leq 
		C_\nu \, \theta_T(\xi) . 
	\end{equation}
\end{theorem}

\begin{proof}
	Let $T\in(0,\infty)$ and $\xi \in \Xi$ be arbitrary, 
	and set $\alpha:=\alpha_T$, 
	$\beta:=\beta_T$, 
	$\eta'\! := \eta'_T$.  
	By Assumption~{\upshape\ref{ass:convker_holder_assu}\ref{item:diff_diff_holder_conv_i}} 
	we obtain that, for all $(t,s)\in\triangle[0,T]$, 
	\begin{align*}
		\bigl( 
		\| \Delta^\xi_1 \|_{L^1(0,t)}
		+ \| \Delta^\xi_2 \|_{L^2(0,t)} \bigr) 
		&(t-s)^{\alpha}
		+ \| \Delta^\xi_1 \|_{L^1(0,t-s)}
		+ \| \Delta^\xi_2 \|_{L^2(0,t-s)}
		\\
		&\leq 
		2 \theta_T(\xi)  
		\bigl( t^{\beta}  (t-s)^\alpha + (t-s)^\beta \bigr) 
		\leq 
		C'  \theta_T(\xi) (t-s)^{\alpha \wedge \beta},
	\end{align*}
	where $C' := 2 T^{\beta - (\alpha\wedge\beta)} (T^\alpha + 1)$.
	Furthermore, 
	by Lemma~{\upshape\ref{lma:kernel_increments_holder_and_rate}} 
	we have that 
	\begin{align*} 
		\| \Delta^\xi_1(t-\,\cdot\,) - \Delta^\xi_1(s-\,\cdot\,) \|_{L^1(0,s)}
		&+  
		\| \Delta^\xi_2(t-\,\cdot\,) - \Delta^\xi_2(s-\,\cdot\,) \|_{L^2(0,s)}
		\\
		&\leq 
		2 \eta' \theta_T(\xi) (t-s)^{\beta} 
		\leq 
		C'' \theta_T(\xi) (t-s)^{\alpha\wedge\beta},
	\end{align*}
	where $C'' := 2 \eta' T^{\beta - (\alpha\wedge\beta)}$. 
	By combining these estimates 
	with Corollary~{\upshape\ref{cor:Lp_error_of_increment_conv}},   
	we conclude that 
	there exists a constant $\widehat{C}\in(0,\infty)$ 
	such that, for all $\xi\in\Xi$, 
	\begin{align*}
		| X-\widehat{X}^\xi |_{C^{\alpha\wedge\beta}([0,T];L^p(\Omega))}
		&= 
		\sup_{0 \leq s < t \leq T} 
		\frac{\| (X_t - \widehat{X}^\xi_t) - (X_s - \widehat{X}^\xi_s) \|_{L^p(\Omega)}}{ 
			(t-s)^{\alpha\wedge\beta} }
		\\
		&\leq 
		\widehat{C} \left( C' + C'' \right)  \theta_T(\xi) . 
	\end{align*}
	Furthermore, 
	by Corollary~{\upshape\ref{cor:strong_error_Lp_conv}} 
	(see also \eqref{eq:C0Lp_theta_bd}),
	there exists a constant $\widetilde{C} \in (0,\infty)$
	such that  
	$\| X-\widehat{X}^\xi \|_{C^0([0,T];L^p(\Omega))} 
	\leq 
	\widetilde{C} \, \theta_T(\xi)$ 
	holds for all $\xi\in\Xi$. 
	We thus have~established 
	\eqref{eq:thm:SVE_holder_Lp_bound} 
	for the constant $C:= \widetilde{C} + \widehat{C} \left( C' + C'' \right)$ 
	which is independent of~$\xi\in\Xi$. 
	 
	By the Kolmogorov--Chentsov continuity theorem 
	(see, e.g.,~\cite[Theorem~3.9]{coxCompleteness2024})  
	there exist 
	(i) a modification $Y^\xi$ of $X - \widehat{X}^\xi$ 
	which, for all $\nu \in (0, (\alpha\wedge\beta)- \nicefrac{1}{p} )$, 
	has $\nu$-H\"older continuous sample paths, 
	and 
	(ii) a constant $\widetilde{C}_{\alpha,\beta,\nu,p,T}\in(0,\infty)$, 
	which is  independent of $\xi\in\Xi$, 
	such that  
	\begin{equation*}
		\| Y^\xi \|_{L^p(\Omega;C^\nu([0,T]))} 
		\leq  
		\widetilde{C}_{\alpha,\beta,\nu,p,T}
		\| X - \widehat{X}^\xi \|_{C^{\alpha \wedge \beta}([0,T];L^p(\Omega))} 
		\leq 
		\widetilde{C}_{\alpha,\beta,\nu,p,T} \, 
		C \, \theta_T(\xi) . 
	\end{equation*} 
	Since $X - \widehat{X}^\xi$ has continuous sample paths,
	$Y^\xi$ is indistinguishable from $X - \widehat{X}^\xi$, 
	and their sample paths coincide $\bbP$-a.s. 
	Thus, also 
	their pathwise $C^\nu([0,T])$-norms are identical, $\bbP$-a.s., 
	and the claim \eqref{eq:kolm-chents_difference} follows.  
\end{proof}

  
\section{Sinc quadrature approximations for the fractional kernel}
\label{sec:sinc-fractional}

In this section, we introduce
and analyze sinc quadrature approximations 
of the fractional kernel  
$k_\gamma \colon (0,\infty) \to (0,\infty)$, 
given by
\begin{equation}\label{eq:fracker} 
	k_\gamma(t) :=  t^{-\gamma},
	\qquad
	\gamma\in(0,\infty).
\end{equation}
To this end,
we first represent~$k_\gamma$
using the gamma function, 
and perform a
change of variables
$x = \log \rho$
to obtain an integral over the whole real line,
\begin{equation}\label{eq:fracker_as_laplace}
	k_\gamma(t)
	=
	\frac{1}{\Gamma(\gamma)}
	\int_0^\infty e^{-\rho t} \rho^{\gamma - 1} \diff{\rho}
	=
	\frac{1}{\Gamma(\gamma)}
	\int_{-\infty}^\infty e^{-t e^x + \gamma x} \diff{x},
	\qquad
	t\in(0,\infty).
\end{equation}
For constructing 
a sequence of approximating kernels,
we apply a \emph{sinc quadrature} to 
the latter integral representation 
in \eqref{eq:fracker_as_laplace}.
This quadrature technique
corresponds to the trapezoidal rule 
and, by leveraging tools from complex analysis, 
we derive convergence rates 
that are exponential in the reciprocal 
of the quadrature step size.

First, in Subsection~{\upshape\ref{subsec:sinc_quad_pointw}}, 
we bound the pointwise error
when $t^{-\gamma}$ is approximated
using an infinite series
corresponding to a trapezoidal rule
for the integral representation  
\eqref{eq:fracker_as_laplace}
on the entire real line.
In Subsection~{\upshape\ref{subsec:sinc_quad_trunc}},
we then extend this analysis
to obtain error bounds
in $L^r(0,T)$ and in Sobolev
norms for the truncated series.  


\subsection{Pointwise error of the infinite sinc series}
\label{subsec:sinc_quad_pointw}

It is well known that,
in the case that the latter integrand
in \eqref{eq:fracker_as_laplace}
can be extended
from the real line to
some strip 
$S_d := \{z \in \bbC : | \Im(z) | < d\}$ 
in the complex plane 
of half-width  
$d\in(0,\infty)$ 
such that it exhibits certain regularity
and integrability properties on $S_d$,
then the sinc quadrature method 
converges exponentially in the reciprocal of 
the (uniform) quadrature step size~$h\in(0,\infty)$. 
 
More specifically,
if we define the mapping
\begin{equation}\label{eq:def_g_fracker}
	g_\gamma\colon(0,\infty)\times \bbC \to \bbC,
	\qquad
	g_\gamma(t,z)
	:=
	\frac{1}{\Gamma(\gamma)} \, e^{-t e^z + \gamma z}, 
\end{equation}
then pointwise convergence
of the sinc quadrature method  
to approximate $t^{-\gamma}$, 
for some $t\in(0,\infty)$, 
follows and is of the order 
$e^{-a_d/h}$, 
with $a_d\in(0,\infty)$ 
depending on $d$, 
provided that $g_\gamma(t,\,\cdot\,)$ is
an element of the Hardy space $H^1(S_d)$
(see \cite[Section~3.1]{stenger1993sinc}).
In the following proposition, 
we establish and use the fact that this 
indeed is the case for every  
$d \in \bigl( 0, \tfrac{\pi}{2} \bigr)$.

\begin{proposition}\label{prop:error_term_full_sum}
	Let $\gamma \in (0,\infty)$, let $h \in (0, \infty)$,
	and let $\delta \in \bigl(0, \pi^2\bigr)$.
	Suppose that $k_\gamma$ and
	$g_\gamma$ are defined as in
	\eqref{eq:fracker}
	and \eqref{eq:def_g_fracker}, respectively.
	Then, for all $t \in (0,\infty)$,
	\begin{equation}\label{eq:prop:error_term_full_sum}
		\biggl| k_\gamma(t) -h \sum_{\ell  \in \bbZ} g_\gamma(t,\ell h) \biggr|
		\leq
		2 \,
		\bigl[ \sin\bigl(\tfrac{\delta}{2\pi} \bigr) \bigr]^{-\gamma}
		\frac{ e^{-(\pi^2-\delta)/h} }{
			1-e^{-(\pi^2-\delta)/h} } \, t^{-\gamma }.
	\end{equation}
\end{proposition}

\begin{proof}
	Let $t \in (0,\infty)$ be arbitrary,
	and define the strip 
	half-width $d := \tfrac{\pi}{2} - \tfrac{\delta}{2\pi}$.
	Since $\delta \in \bigl(0,\pi^2 \bigr)$,
	we have that $d \in \bigl(0, \tfrac{\pi}{2}\bigr)$.
	To prove the claim, we resort to
	standard sinc quadrature error estimates 
	(see \cite[Theorem~2.20]{lund1992sinc} or \cite[Theorem 3.2.1]{stenger1993sinc}).
	To apply these results,
	we need to verify that
	$g_\gamma(t, \,\cdot\,) \in H^1(S_d)$.
	It is clear that $g_\gamma(t,\,\cdot\,)$
	is holomorphic everywhere 
	on $\bbC \supset S_d$.
	Therefore, it is an 
	element of $H^1(S_d)$ if 
	\begin{gather*}
		\lim_{x \to \pm \infty}
		\int_{-d}^d |g_\gamma(t,x+iy)| \diff{y} < \infty,
		\\ 
		N_1(g_\gamma(t,\,\cdot\,), S_d)  
		:= 
		\lim_{y \uparrow d}\int_{-\infty}^\infty  \left( | g_\gamma(t,x+iy) | +  | g_\gamma(t,x-iy) |\right)\diff{x} < \infty.
	\end{gather*}  
	
	Let $z \in \bbC$,
	and set $x := \Re z$
	and $y := \Im z$.
	Using Euler's formula, we find that
	\begin{equation*} 
		g_\gamma(t,z) 
		= 
		\frac{1}{\Gamma(\gamma)} {e^{-t e^z} e^{\gamma z}}
		= 
		\frac{1}{\Gamma(\gamma)} {e^{-t e^{x} \cos(y) + \gamma  x} e^{i \left[-t e^{x} \sin(y) + \gamma y\right]}}.
	\end{equation*}
	Consequently, $| g_\gamma(t,z) | 
	= 
	\frac{1}{\Gamma(\gamma)} 
	e^{-t e^{x} \cos( y) + \gamma  x}$
	and,
	for all $x \in \bbR$,
	we derive the bound
	\begin{equation*}
		\int_{-d}^d | g_\gamma(t,x+iy) |\diff{y}
		= \frac{1}{\Gamma(\gamma)} \int_{-d}^d e^{-t e^x \cos(y) + \gamma x} \diff{y} \le \frac{2d}{\Gamma(\gamma)}  \, e^{-t e^x \cos(d)}
		e^{\gamma x}.
	\end{equation*}
	Since $d \in \bigl(0,\tfrac{\pi}{2}\bigr)$
	and $t \in (0,\infty)$,
	we have that $t \cos(d) \in (0,\infty)$
	and,
	for $x \to +\infty$,
	the latter expression tends to~$0$ 
	because of
	the double exponential term,
	whereas the term~$e^{\gamma x}$  
	renders it converging to $0$
	for $x \to -\infty$.
	Moreover,
	since $\cos(y) = \cos(-y)$ holds
	for all $y \in [0,\infty)$,
	we find that
	\begin{equation*}
		N_1(g_\gamma(t,\,\cdot\,), S_d) 
		=
		2 \int_{-\infty}^\infty | g_\gamma(t,x+id) | \diff{x}
		=
		\frac{2}{\Gamma(\gamma)}
		\int_{-\infty}^\infty e^{-t e^x \cos(d)} e^{\gamma x} \diff{x},
	\end{equation*}
	and using
	the substitution $u = t \cos(d) \, e^{x\!}$, 
	the definition of the gamma function, 
	and that $\cos(d) = \sin\bigl( \frac{\pi}{2} - d \bigr)$, 
	we can explicitly evaluate $N_1(g_\gamma(t,\,\cdot\,), S_d)$ 
	as follows, 
	\begin{equation*}
		N_1(g_\gamma(t,\,\cdot\,), S_d) 
		=
		\frac{2}{\Gamma(\gamma)}
		[\cos(d)]^{-\gamma } \, t^{-\gamma }
		\int_0^\infty e^{-u}  \, u^{\gamma-1} \diff{u}
		=
		2 \, \bigl[ \sin\bigl (\tfrac{\pi}{2}-d\bigr)\bigr]^{-\gamma } \, t^{-\gamma }.
	\end{equation*}
	Since this value is finite for any fixed $t \in (0,\infty)$,
	it follows that $g_\gamma(t,\,\cdot\,) \in H^1(S_d)$.
	
	By \cite[Theorem~3.2.1]{stenger1993sinc}
	and by recalling that
	$d = \frac{1}{2} \bigl( \pi - \frac{\delta}{\pi}\bigr)$,
	we obtain that 
	\begin{equation*}
		\biggl|
		k_\gamma(t) -h \sum_{\ell  \in \bbZ} g_\gamma(t,\ell h) \biggr|
		\leq
		\frac{N_1(g_\gamma(t,\,\cdot\,),S_d)}{
			2 \sinh( \pi d / h ) }
		\, e^{- \pi d / h }
		=
		\bigl[ \sin\bigl( \tfrac{\delta}{2\pi} \bigr)\bigr]^{-\gamma}
		\,
		\frac{ e^{- \frac{ \pi^2 - \delta }{ 2h } } }{
			\sinh\bigl( \tfrac{\pi^2 - \delta }{2h}\bigr)}
		\,
		t^{-\gamma} .
	\end{equation*}
	Finally, the claim \eqref{eq:prop:error_term_full_sum}
	follows as the identity
	$\sinh(x) 
	=
	\tfrac{1}{2} e^x (1 - e^{-2x})$
	implies that
	\begin{equation*}
		\bigl[ \sinh \bigl( \tfrac{\pi^2 - \delta}{2h} \bigr) \bigr]^{-1}
		e^{- \frac{ \pi^2 - \delta }{ 2h } }
		=
		\frac{2 \, e^{-(\pi^2 - \delta)/h } }{
			1-e^{-(\pi^2 - \delta)/h}} .
		\qedhere
	\end{equation*}
\end{proof}


\subsection{Error of the truncated series in Lebesgue spaces}
\label{subsec:sinc_quad_trunc} 

To extend the pointwise estimate
of the sinc approximation in
\eqref{eq:prop:error_term_full_sum}
to an error bound
for a truncated series
in the $L^r$-norm,
we first derive
explicit bounds for the tails of the series. 
These  estimates additionally imply 
absolute convergence
of the infinite series in~$L^r(s,t)$.    

The following lemma distinguishes 
intervals bounded away from zero 
from those starting at $s=0$. 
In the latter case, the integrability condition 
$\gamma r < 1$ is required.

\begin{lemma}\label{lma:sinc_quad_series_conv+tail}
	Let $\gamma\in(0,\infty)$,
	$h \in (0,\infty)$, $t\in(0,\infty)$,
	and define $g_\gamma$ as in \eqref{eq:def_g_fracker}.
	
	Then, for every $r \in [1,\infty)$
	and all $s \in (0,t]$,
	the sequence
	$(g_\gamma(\,\cdot\,,\ell h))_{\ell \in \bbZ}$
	is absolutely summable in $L^r(s,t)$.
	In addition, for all $M^{-\!}, M^{+\!} \in \bbN$,
	every $\widetilde{\beta} \in \bigl[0,\tfrac{1}{r} \bigr]$,
	and all
	${\eps \in \bigl(
		0 \vee \bigl(\gamma+\widetilde{\beta}-\frac{1}{r} \bigr),
		\infty\bigr)}$,
	the following tail estimates hold, 
	\begin{align}\label{eq:sinc_tail_neg} 
		h \sum_{\ell  = -\infty}^{-(M^{-\!}+1)}
		\| g_\gamma(\,\cdot\,,\ell h) \|_{L^r(s,t)}
		&\leq
		\frac{1}{\Gamma(\gamma + 1)}
		\, e^{-\gamma  M^{-\!} h} \, (t-s)^{\nicefrac{1}{r}}, 
		\\ 
		h \sum_{\ell  = M^{+\!}+1}^\infty
		\| g_\gamma(\,\cdot\,,\ell h) \|_{L^r(s,t)}
		&\leq
		\frac{ \eps^\eps e^{-\eps} }{
			\bigl(\frac{1}{r} - \gamma - \widetilde{\beta} + \eps \bigr) \, \Gamma(\gamma)}
		\,
		e^{-\left( \frac{1}{r} - \gamma -\widetilde{\beta} + \eps \right) M^{+\!} h}
		\, s^{-\eps} (t-s)^{\widetilde{\beta}} .  
		\label{eq:sinc_tail_pos_s>0}
	\end{align} 
	
	Moreover, if $\gamma\in (0,1)$
	and $r\in\bigl[ 1, \gamma^{-1} \bigr)$,
	the sequence $(g_\gamma(\,\cdot\,,\ell h))_{\ell \in \bbZ}$
	is absolutely summable in $L^r(0,t)$,
	and
	the negative tail satisfies
	the bound \eqref{eq:sinc_tail_neg}
	for all $s \in [0,t]$.   
	For
	all $r\in\bigl[ 1, \gamma^{-1} \bigr)$,
	every $\beta \in \bigl[ 0, \frac{1}{r} - \gamma \bigr)$,
	and all $s \in [0,t]$,
	the positive tail satisfies 
	\begin{equation}\label{eq:sinc_tail_pos}
		h \sum_{\ell  = M^{+\!}+1}^\infty
		\| g_\gamma(\,\cdot\,,\ell h) \|_{L^r(s,t)}
		\leq  \frac{e^{-s \, e^{M^{+\!}h}} }{
			\bigl( \frac{1}{r}-\gamma-\beta \bigr) \, \Gamma(\gamma) }
		\, e^{-\left( \frac{1}{r}-\gamma-\beta \right)M^{+\!} h}
		\, (t-s)^{\beta}.
	\end{equation}
\end{lemma}

\begin{proof}
	Let
	$\gamma,h,t\in(0,\infty)$,
	$r\in[1, \infty)$ and $s\in[0,t]$
	be arbitrary.
	Then, for all~$\ell \in \bbZ$,
	\begin{align}
		\| g_\gamma(\,\cdot\,,\ell h) \|_{L^r(s,t)}
		& =
		\frac{1}{\Gamma(\gamma)}
		\left( \int_s^t e^{-ru \, e^{\ell h} + r\gamma \ell h} \diff{u}\right)^{\nicefrac{1}{r}}
		\notag
		\\
		& =
		\frac{ e^{\left(\gamma-\frac{1}{r} \right)\ell h} }{
			r^{\nicefrac{1}{r}} \, \Gamma(\gamma)}
		\, e^{-s \, e^{\ell h}}
		\Bigl( 1 - e^{-r(t-s) e^{\ell h}} \Bigr)^{\nicefrac{1}{r}} \!.
		\label{eq:g_gamma_Lr_general}
	\end{align}
	Note that the term $1-e^{-x}$ (with $x \geq 0$)
	can be bounded by both $x$ and $1$.
	Thus, for every $\theta\in[0,1]$
	and all $x\geq 0$, we have that
	$1-e^{-x}
	= (1-e^{-x} )^\theta ( 1-e^{-x} )^{1-\theta}
	\leq 
	x^\theta\!$.
	Using this estimate
	with $\theta := r \widetilde{\beta}$,
	for $\widetilde{\beta} \in \bigl[0,\tfrac{1}{r} \bigr]$,
	we can
	bound \eqref{eq:g_gamma_Lr_general} by
	\begin{equation}\label{eq:g_gamma_Lr_beta} 
		\| g_\gamma(\,\cdot\,,\ell h) \|_{L^r(s,t)}
		\leq
		\frac{1}{\Gamma(\gamma)}
		\, e^{-s \, e^{\ell h}}
		\, e^{\left(\gamma + \widetilde{\beta} - \frac{1}{r} \right) \ell h}
		(t-s)^{\widetilde{\beta}},
	\end{equation}
	where we also used that
	$r^{\widetilde{\beta} - \frac{1}{r}} \leq 1$.
	Choosing $\widetilde{\beta} = \frac{1}{r}$
	in this upper bound, it follows
	that $(g_\gamma(\,\cdot\,,\ell h))_{\ell \in \bbZ}$
	is absolutely summable in $L^r(s,t)$ for $s > 0$.
	Deriving the case $s=0$ and 
	the tail estimates  
	\eqref{eq:sinc_tail_neg}--\eqref{eq:sinc_tail_pos} 
	will constitute 
	the remainder of this proof.
	
	\textbf{Negative tail, $\boldsymbol{s \in [0,t]}$:}
	As $\ell \to -\infty$, the double exponential
	in \eqref{eq:g_gamma_Lr_beta}
	can be bounded from above by $1$,
	for all $s\in[0,t]$,
	and the term
	$e^{\left(\gamma + \widetilde{\beta} - \frac{1}{r} \right)\ell h}$
	is exponentially decaying
	if we choose
	$\widetilde{\beta}$ as large as possible, i.e.,
	$\widetilde{\beta} = \frac{1}{r}$.
	It follows that,
	for all~$n\in\bbN$,
	\begin{equation*}
		\| g_\gamma(\,\cdot\,,-n h) \|_{L^r(s,t)}
		\leq
		\tfrac{1}{\Gamma(\gamma)} \,
		e^{-\gamma n h} \, (t-s)^{\nicefrac{1}{r}},
		\quad
		t\in(0,\infty),
		\quad
		s\in[0,t].
	\end{equation*}
	As this bound is monotonically decreasing in $n$,
	we can use a Riemann sum estimate
	which shows that,
	for every $t\in(0,\infty)$ and all $s\in[0,t]$,
	\begin{align*}
		h \sum_{\ell =-\infty}^{-(M^{-\!}+1)}
		\| g_\gamma(\,\cdot\,,\ell h) 
		& \|_{L^r(s,t)}
		\leq
		\frac{ (t-s)^{\nicefrac{1}{r}} }{\Gamma(\gamma)} \, h
		\sum_{n =M^{-\!}+1}^{\infty}
		e^{-\gamma n h}
		\\
		& \leq
		\frac{(t-s)^{\nicefrac{1}{r}}}{\Gamma(\gamma)}
		\int_{ M^{-\!} h}^\infty e^{-\gamma x} \diff{x}
		=
		\frac{1}{\Gamma(\gamma+1)} \, e^{-\gamma  M^{-\!} h} \, (t-s)^{\nicefrac{1}{r}},
	\end{align*}
	where we used that $\gamma\Gamma(\gamma)=\Gamma(\gamma+1)$,
	completing the proof of \eqref{eq:sinc_tail_neg}.
	
	\textbf{Positive tail, $\boldsymbol{s \in(0,t]}$:}
	Let $\widetilde{\beta} \in \bigl[0,\frac{1}{r} \bigr]$
	and $\eps \in \bigl( 0 \vee \bigl(\gamma+\widetilde{\beta}-\frac{1}{r}\bigr), \infty \bigr)$.
	Using the identity
	$e^{\left(\gamma + \widetilde{\beta} - \frac{1}{r} \right)\ell h}
	= e^{\left(\gamma + \widetilde{\beta} - \frac{1}{r} - \eps \right) \ell h} e^{\eps \ell h}$
	and Lemma~{\upshape\ref{lma:exponential_maximum_bound}} 
	in \eqref{eq:g_gamma_Lr_beta}, 
	we find that 
	\begin{equation*}
		\| g_\gamma(\,\cdot\,,\ell h) \|_{L^r(s,t)}
		\leq
		\frac{ \eps^\eps e^{-\eps} }{ \Gamma(\gamma) }
		\,
		e^{-\left(\frac{1}{r} - \gamma - \widetilde{\beta} +\eps \right) \ell h}
		\, s^{-\eps} (t-s)^{\widetilde{\beta}} , 
		\qquad 
		\ell\in\bbZ. 
	\end{equation*}
	The claim \eqref{eq:sinc_tail_pos_s>0}
	then follows again from
	a Riemann sum 
	estimate,
	\begin{align*}
		h \sum_{\ell = M^{+\!} + 1}^{\infty} \| g_\gamma(\,\cdot\,,\ell h) \|_{L^r(s,t)}
		&\leq
		\frac{ \eps^\eps e^{-\eps} }{ \Gamma(\gamma) }
		\, s^{-\eps}(t-s)^{\widetilde{\beta}}
		\int_{M^{+\!} h }^\infty
		e^{-\left( \frac{1}{r} - \gamma - \widetilde{\beta} + \eps \right)x } \diff{x} 
		\\  
		&\leq 
		\frac{ \eps^\eps e^{-\eps} }{
			\bigl( \frac{1}{r} - \gamma-\widetilde{\beta} + \eps\bigr) \, \Gamma(\gamma)}
		\, e^{-\left( \frac{1}{r} - \gamma-\widetilde{\beta}+\eps\right) M^{+\!} h}
		\, s^{-\eps} (t-s)^{\widetilde{\beta}}.
	\end{align*} 
	
	\textbf{Positive tail, $\boldsymbol{s \in [0,t]}$:}
	Suppose now that
	$r\in\bigl[ 1, \gamma^{-1} \bigr)$, i.e.,
	$\gamma \in \bigl( 0,\frac{1}{r} \bigr)$,
	let $\beta \in \bigl[ 0,\frac{1}{r}-\gamma\bigr)$
	and choose $\widetilde{\beta} = \beta$
	in \eqref{eq:g_gamma_Lr_beta}.
	By the same Riemann sum argument, 
	also the tail 
	estimate \eqref{eq:sinc_tail_pos} follows, 
	which completes the proof: 
	\begin{align*}
		h \sum_{\ell = M^{+\!} + 1}^{\infty} \| g_\gamma(\,\cdot\,,\ell h) \|_{L^r(s,t)}
		& \leq
		\frac{ e^{-s \, e^{M^{+\!}h}} }{ \Gamma(\gamma) }
		\, (t-s)^{\beta}
		\int_{M^{+\!} h}^{\infty}
		e^{-\left( \frac{1}{r} - \gamma - \beta \right) x} \diff{x}
		\\
		& = \frac{ e^{-s\, e^{M^{+\!}h}}
		}{ \bigl( \frac{1}{r}-\gamma-\beta\bigr) \, \Gamma(\gamma)}
		\, e^{-\left(\frac{1}{r}-\gamma-\beta \right)M^{+\!} h}
		\, (t-s)^{\beta}. \qedhere
	\end{align*}
\end{proof} 

Let $\gamma\in(0,\infty)$ and
the function~$g_\gamma$ be given
by \eqref{eq:def_g_fracker}.
For a step size ${h \in (0,\infty)}$
and truncation parameters $M^{-\!}, M^{+\!} \in \bbN$,
we define 
the (\emph{truncated}) 
\emph{sinc quadrature approximation} 
$\widehat{k}_\gamma^{(h,M^{-\!},M^{+\!})}
\colon (0,\infty) \to (0,\infty)$ 
of the fractional kernel 
$k_\gamma$ in \eqref{eq:fracker} 
by
\begin{equation}\label{eq:sinc_quad_kernel}
	\widehat{k}_\gamma^{(h,M^{-\!},M^{+\!})}(t)
	:=
	h
	\sum_{\ell =-M^{-\!}}^{M^{+\!}} g_\gamma(t,\ell h)
	=
	\frac{h}{\Gamma(\gamma)}
	\sum_{\ell=-M^{-\!}}^{M^{+\!}} e^{-t \, e^{\ell h}} e^{\gamma \ell h}.
\end{equation}

In the next theorem, 
we combine the estimate \eqref{eq:prop:error_term_full_sum} 
of Proposition~{\upshape\ref{prop:error_term_full_sum}} 
for the pointwise error of the 
infinite sinc quadrature series with 
the tail estimates 
\eqref{eq:sinc_tail_neg}--\eqref{eq:sinc_tail_pos}
of Lemma~{\upshape\ref{lma:sinc_quad_series_conv+tail}} 
to estimate the error of 
approximating $k_\gamma$ by 
$\widehat{k}_\gamma^{(h,M^{-\!},M^{+\!})}$ 
in the norm of $L^r(s,t)$ for $t\in(0,T]$ and, 
depending on $\gamma$, 
for $s\in(0,t]$ or $s\in[0,t]$.   
In addition, we 
bound the $L^r(s,t)$-difference  
of the derivatives 
for $s\in(0,t]$.  

\begin{theorem}
	\label{thm:error_bound_Lr_general_gamma}
	Let $T\in(0,\infty)$, 
	let $\gamma \in (0,\infty)$,
	and let $k_\gamma$ be given by \eqref{eq:fracker}. 
	Let $h_0 \in (0,\infty)$
	and, for $h \in (0,h_0]$
	and $M^{-\!},M^{+\!} \in \bbN$,
	define $\widehat{k}_\gamma^{(h,M^{-\!},M^{+\!})}$
	as in \eqref{eq:sinc_quad_kernel}.
	\begin{enumerate}[leftmargin=8mm, label={\upshape(\roman*)}] 
		\item\label{item:thm:error_bound_Lr_general_gamma-i} 
		For all $r \in [1,\infty)$,
		$\widetilde{\beta} \in \bigl[ 0, \tfrac{1}{r} \bigr)$,
		$\eps \in \bigl( 0 \vee \bigl( \gamma + 
		\widetilde{\beta} - \frac{1}{r} \bigr), \gamma\bigr]$, 
		and 
		$\delta \in \bigl( 0,\pi^2 \bigr)$, 
		there exists a constant $c'\!\in(0,\infty)$ 
		such that, 
		for all $t\in(0,T]$ and $s\in(0,t]$,
		\begin{align*} 
			&\bigl\|
			k_\gamma - \widehat{k}_\gamma^{(h,M^{-\!},M^{+\!})} 
			\bigr\|_{L^r(s,t)}
			\\
			&\qquad\quad 
			\leq 
			c' 
			\Bigl(e^{-(\pi^2 - \delta)/h} 
			+ e^{-\gamma M^{-\!} h} 
			+  e^{-\left(\frac{1}{r} - \gamma - \widetilde{\beta} + \eps\right) M^{+\!} h} \Bigr)
			s^{-\eps} (t-s)^{\widetilde{\beta}}. 
		\end{align*}
		The constant $c'\in (0,\infty)$
		is independent of $(h,M^{-\!},M^{+\!})$
		and satisfies
		\begin{equation*}
			c' =
			\cO\Bigl(
			(1-r\gamma+r\eps)^{-\nicefrac{1}{r}}
			\cdot \bigl[ \sin\bigl(\tfrac{\delta}{2\pi} \bigr)
			\bigr]^{-\gamma}
			+ 
			\bigl( \tfrac{1}{r} - \gamma - \widetilde{\beta} + \eps 
			\bigr)^{-1}
			\Bigr).
		\end{equation*}
		
		\item\label{item:thm:error_bound_Lr_general_gamma-ii}  
		If $\gamma  \in (0,1)$,
		then for all ${r\in\bigl[ 1, \gamma^{-1} \bigr)}$,
		$\beta \in \bigl[0, \frac{1}{r} - \gamma \bigr)$,
		and 
		$\delta \in \bigl(0,\pi^2 \bigr)$, 
		there exists a constant $c''\!\in(0,\infty)$ 
		such that, for all $t\in(0,T]$ and $s \in [0,t]$,
		\begin{equation*}  
			\bigl\|
			k_\gamma -  \widehat{k}_\gamma^{(h,M^{-\!},M^{+\!})}
			\bigr\|_{L^r(s,t)}
			\leq
			c'' 
			\Bigl(
			e^{-(\pi^2 - \delta)/h}
			+
			e^{-\gamma M^{-\!} h}
			+
			e^{-\left(\frac{1}{r} - \gamma - \beta\right) M^{+\!} h}  \Bigr) (t-s)^{\beta}. 
		\end{equation*}
		The constant $c''\in(0,\infty)$ is 
		independent of $(h,M^{-\!},M^{+\!})$
		and satisfies 
		\begin{equation*} 
			c''  
			=
			\cO\Bigl(
			( 1-r\gamma )^{-\nicefrac{1}{r}}
			\cdot
			\bigl[ \sin\bigl(\tfrac{\delta}{2\pi} \bigr)\bigr]^{-\gamma}
			+
			\bigl(\tfrac{1}{r} - \gamma - \beta \bigr)^{-1} \Bigr). 
		\end{equation*}
		
		\item\label{item:thm:error_bound_Lr_general_gamma-iii}  
		If $\gamma  \in (0,1)$,
		then for all ${r\in\bigl[ 1, \gamma^{-1} \bigr)}$,
		$\beta \in \bigl[0, \frac{1}{r} - \gamma \bigr)$,
		and 
		$\delta \in \bigl( 0,\pi^2 \bigr)$, 
		there exists a constant $c'''\!\in(0,\infty)$ 
		such that, for all $t\in(0,T]$ and $s \in (0,t]$, 
		\begin{equation*}
			\bigl\| 
			\bigl( k_\gamma -  \widehat{k}_\gamma^{(h,M^{-\!},M^{+\!})} \bigr)' \bigr\|_{L^r(s,t)}
			\leq 
			c''' 
			\Bigl( e^{-(\pi^2 - \delta)/h} + e^{-(\gamma+1) M^{-\!} h} + e^{-\left(\frac{1}{r} - \gamma - \beta\right) M^{+\!} h} \Bigr) 
			s^{\beta - 1}. 
		\end{equation*}
		The constant $c'''\in(0,\infty)$ is 
		independent of $(h,M^{-\!},M^{+\!})$
		and satisfies  
		\begin{equation*}
			c'''  
			=
			\cO\Bigl(
			(1 - r\gamma - r\beta)^{-\nicefrac{1}{r}}
			\cdot 
			\bigl[\sin\bigl(\tfrac{\delta}{2\pi}\bigr)\bigr]^{-(\gamma + 1)}
			+ 
			\bigl(\tfrac{1}{r} - \gamma - \beta\bigr)^{-1} 
			\Bigr).
		\end{equation*} 
	\end{enumerate} 
\end{theorem}

\begin{proof} 
	Recall that
	$\widehat{k}_\gamma^{(h,M^{-\!},M^{+\!})}
	=
	h \sum_{\ell =-M^{-\!}}^{M^{+\!}} g_\gamma(\,\cdot\,,\ell h)$,
	where $g_\gamma$ is defined
	in \eqref{eq:def_g_fracker}.
	By Lemma~{\upshape\ref{lma:sinc_quad_series_conv+tail}},
	under the conditions of this theorem,
	the infinite series
	converges absolutely in $L^r(s,t)$.
	Hence, by the triangle inequality,
	\begin{equation}
		\label{eq:Lp_norm_three_terms_general}
		\begin{split}
			&\bigl\|
			k_\gamma -  \widehat{k}_\gamma^{(h,M^{-\!},M^{+\!})}
			\bigr\|_{L^r(s,t)}
			\leq
			\biggl\| k_\gamma -  h \sum_{\ell  \in \bbZ} g_\gamma(\,\cdot\,,\ell h) \biggr\|_{L^r(s,t)}
			\\
			&\hspace*{20mm}
			+\biggl\| h \sum_{\ell =-\infty}^{-(M^{-\!}+1)} g_\gamma(\,\cdot\,, \ell h) \biggr\|_{L^r(s,t)}
			+\biggl\|  h \sum_{\ell =M^{+\!}+1}^\infty g_\gamma(\,\cdot\,, \ell h) \biggr\|_{L^r(s,t)}
			.
		\end{split}
	\end{equation}
	By Proposition~{\upshape\ref{prop:error_term_full_sum}},
	we can estimate the first term as follows,
	\begin{align}
		\biggl\| k_\gamma -  h \sum_{\ell  \in \bbZ} g_\gamma(\,\cdot\,,\ell h) \biggr\|_{L^r(s,t)}
		& \leq
		\Biggl( \int_s^t
		\biggl| \, 2 \,
		\bigl[ \sin\bigl(\tfrac{\delta}{2\pi}\bigr)\bigr]^{-\gamma }
		\, \frac{ e^{-(\pi^2 - \delta)/h} }{ 1-e^{-(\pi^2 - \delta)/h} }
		\, u^{-\gamma } \biggr|^r \diff{u}
		\Biggr)^{\nicefrac{1}{r}} \quad
		\notag
		\\
		& \leq
		\frac{ 2 \, \bigl[ \sin\bigl( \tfrac{\delta}{2\pi} \bigr)\bigr]^{-\gamma } }{
			1-e^{-(\pi^2 - \delta)/h_0} }
		\, e^{-(\pi^2 - \delta)/h}
		\biggl( \int_s^t u^{-r\gamma } \diff{u} \biggr)^{\nicefrac{1}{r}}\!.\hspace{-50mm}
		\label{eq:Lp_norm_series_term_general}
	\end{align}
	
	\textbf{Proof of~{\upshape\ref{item:thm:error_bound_Lr_general_gamma-i}}:}
	Let
	$r \in [1,\infty)$,
	$\widetilde{\beta} \in \bigl[ 0, \tfrac{1}{r} \bigr)$,
	$\eps \in \bigl( 0 \vee \bigl(
	\gamma+ \widetilde{\beta} - \frac{1}{r} \bigr), \gamma\bigr]$, 
	$\delta \in \bigl( 0,\pi^2 \bigr)$. 
	Then, we obtain for the
	integral term in \eqref{eq:Lp_norm_series_term_general},
	for all $t\in(0,T]$ and $s\in(0,t]$,
	\begin{equation*}
		\biggl(
		\int_s^t u^{-r\gamma } \diff{u}
		\biggr)^{\nicefrac{1}{r}}
		\leq
		s^{-\eps}
		\biggl( \int_s^t u^{-r(\gamma-\eps) } \diff{u} \biggr)^{\nicefrac{1}{r}} 
		\leq
		\frac{s^{-\eps} \, T^{\frac{1}{r} - \gamma + \eps - \widetilde{\beta}}}{(1-r(\gamma-\eps))^{\nicefrac{1}{r}}} 
		\, (t-s)^{\widetilde{\beta}},
	\end{equation*}
	since 
	$0\leq r (\gamma - \eps) < 
	r \bigl( \tfrac{1}{r} - \widetilde{\beta} \bigr) 
	\leq 1$  
	and $(0,\infty) \ni t\mapsto t^{1-r(\gamma-\eps)}$ 
	is subadditive. 
	
	The latter two terms of \eqref{eq:Lp_norm_three_terms_general}
	can be approximated using
	the tail estimates \eqref{eq:sinc_tail_neg} 
	and \eqref{eq:sinc_tail_pos_s>0} 
	of Lemma~{\upshape\ref{lma:sinc_quad_series_conv+tail}}.
	Interchanging the norm and the series 
	yields
	\begin{align*} 
		\biggl\|
		h \sum_{\ell  = -\infty}^{-(M^{-\!}+1)}  g_\gamma(\,\cdot\,,\ell h)
		\biggr\|_{L^r(s,t)}
		&\leq
		\frac{(t-s)^{\nicefrac{1}{r}}}{\Gamma(\gamma + 1)} \,
		e^{-\gamma  M^{-\!} h}
		\leq
		\frac{T^{\frac{1}{r} - \widetilde{\beta} } }{ \Gamma(\gamma + 1)} \, e^{-\gamma  M^{-\!} h} \, (t-s)^{\widetilde{\beta}}, 
		\\
		\biggl\|
		h \sum_{\ell  = M^{+\!}+1}^\infty g_\gamma(\,\cdot\,,\ell h)
		\biggr\|_{L^r(s,t)}
		&\leq
		\frac{\eps^\eps e^{-\eps}}{
			\bigl(\frac{1}{r} - \gamma-\widetilde{\beta}+\eps\bigr) \, \Gamma(\gamma)}
		\, e^{-\left(\frac{1}{r} - \gamma - \widetilde{\beta} + \eps\right) M^{+\!} h} \, s^{-\eps} (t-s)^{\widetilde{\beta}}. 
	\end{align*}
	Using these three bounds in
	\eqref{eq:Lp_norm_three_terms_general} 
	and that $1 \leq T^\eps s^{-\eps}$
	completes the proof 
	of~{\upshape\ref{item:thm:error_bound_Lr_general_gamma-i}}.
	
	\textbf{Proof of~{\upshape\ref{item:thm:error_bound_Lr_general_gamma-ii}}:}
	Suppose now that
	$\gamma\in(0,1)$,
	and let
	${r\in\bigl[ 1, \gamma^{-1} \bigr)}$, 
	$\beta \in \bigl[0, \frac{1}{r} - \gamma \bigr)$, 
	and $\delta \in \bigl( 0,\pi^2 \bigr)$. 
	Since $r\gamma \in (0,1)$, it holds that,
	for all  $t\in(0,T]$ and
	$s \in [0,t]$,
	\begin{equation*}
		\left(\int_s^t u^{-r\gamma} \diff{u}\right)^{\nicefrac{1}{r}}
		=
		\frac{ \bigl( t^{1-r\gamma} - s^{1-r\gamma} \bigr)^{\nicefrac{1}{r}} }{
			(1-r\gamma)^{\nicefrac{1}{r}} }
		\leq
		(1-r\gamma)^{-\frac{1}{r}} \,
		T^{\frac{1}{r} - \gamma - \beta} \, (t-s)^{\beta},
	\end{equation*}
	where we also used
	that
	$t^{1-r\gamma} = (t-s+s)^{1-r\gamma}
	\leq (t-s)^{1-r\gamma} + s^{1-r\gamma}$
	by subadditivity.
	
	Combined
	with \eqref{eq:Lp_norm_series_term_general},
	this yields an estimate for the first term
	in \eqref{eq:Lp_norm_three_terms_general},
	\begin{equation*}
		\biggl\|
		k_\gamma -  h \sum_{\ell  \in \bbZ} g_\gamma(\,\cdot\,,\ell h)
		\biggr\|_{L^r(s,t)}
		\leq
		\frac{\bigl[\sin\bigl( \tfrac{\delta}{2\pi} \bigr)\bigr]^{-\gamma}} { (1-r\gamma )^{\nicefrac{1}{r}} }
		\frac{ 2 \, T^{\frac{1}{r} - \gamma - \beta} }{ 1 - e^{-(\pi^2 - \delta)/h_0} }
		\, e^{-(\pi^2 - \delta)/h}
		\, (t-s)^{\beta}     .
	\end{equation*}
	For the positive tail,
	we apply \eqref{eq:sinc_tail_pos}
	from Lemma~{\upshape\ref{lma:sinc_quad_series_conv+tail}} 
	and use that
	$e^{-s \, e^{M^{+\!}h}\!} \leq 1$,
	\begin{equation*}
		h \sum_{\ell  = M^{+\!}+1}^\infty \left\| g_\gamma(\,\cdot\,,\ell h) \right\|_{L^r(s,t)}
		\leq
		\frac{1 }{
			\bigl( \frac{1}{r}-\gamma-\beta\bigr) \, \Gamma(\gamma) }
		\, e^{-\left( \frac{1}{r}-\gamma-\beta \right)M^{+\!} h}
		\, (t-s)^{\beta}.
	\end{equation*}
	Using
	these bounds and the negative tail bound \eqref{eq:sinc_tail_neg} in
	\eqref{eq:Lp_norm_three_terms_general}
	proves~{\upshape\ref{item:thm:error_bound_Lr_general_gamma-ii}}.  
	
	\textbf{Proof of~{\upshape\ref{item:thm:error_bound_Lr_general_gamma-iii}}:}
	The definitions \eqref{eq:fracker}, 
	\eqref{eq:def_g_fracker} 
	of the functions~$k_\gamma$~and~$g_\gamma$ 
	imply that 
	\begin{equation*}
		k_\gamma'(t)
		= -\gamma \, t^{-(\gamma + 1)}
		= - \gamma k_{\gamma+1}(t),
		\quad
		\tfrac{\partial}{\partial t} \, 
		g_\gamma(t,x)
		= 
		- \gamma g_{\gamma+1}(t,x), 
		\quad 
		t\in(0,\infty), \; 
		x\in\bbR. 
	\end{equation*}
	By Lemma~{\upshape\ref{lma:sinc_quad_series_conv+tail}},
	the sequence $(g_{\gamma+1}(\,\cdot\,,\ell h))_{\ell \in \bbZ}$
	is absolutely summable in $L^r(s,T)$
	for $s \in (0,T]$.
	Thus, taking derivatives commutes 
	with the infinite series, and 
	\begin{equation*}
		\bigl\| 
		\bigl(k_\gamma - \widehat{k}_\gamma^{(h,M^{-\!},M^{+\!})}\bigr)' 
		\bigr\|_{L^r(s,t)}
		= 
		\gamma \, 
		\bigl\| k_{\gamma+1} - \widehat{k}_{\gamma+1}^{(h,M^{-\!},M^{+\!})} 
		\bigr\|_{L^r(s,t)}.
	\end{equation*}
	Bounding the 
	right-hand side using 
	part~{\upshape\ref{item:thm:error_bound_Lr_general_gamma-i}} 
	with ${\widetilde{\beta} := 0}$ and ${\eps := 1-\beta}$ 
	completes the proof, 
	where we note that 
	$\eps \leq 1 < \gamma+1$ and 
	$\eps > \gamma + 1 - \frac{1}{r} 
	\geq \gamma > 0$, 
	since 
	$\beta \in \bigl[ 0, \frac{1}{r} - \gamma\bigr)$ 
	is assumed. 
\end{proof}

\begin{remark}\label{rmk:type_of_parameters} 
	The error estimates 
	of  Theorem~{\upshape\ref{thm:error_bound_Lr_general_gamma}\ref{item:thm:error_bound_Lr_general_gamma-ii},\ref{item:thm:error_bound_Lr_general_gamma-iii}} 
	for the sinc quadrature approximation 
	of the fractional kernel 
	depend on the following parameters: 
	\begin{enumerate}[leftmargin=8mm, label={$\bullet$}]  
		\item 
		\emph{fixed parameters}: 
		$\gamma \in (0,1)$, 
		$T \in (0,\infty)$, $h_0 \in (0,\infty)$, 
		\item 
		\emph{discretization parameters}: 
		$h \in (0,h_0]$, $M^{-\!} \in \bbN$, $M^{+\!} \in \bbN$, 
		\item 
		\emph{free parameters}: 
		$\delta \in \bigl( 0,\pi^2 \bigr)$, 
		$r \in \bigl[1,\gamma^{-1}\bigr)$, 
		$\beta \in \bigl[ 0, \tfrac{1}{r} - \gamma \bigr)$. 
	\end{enumerate}
	In Theorem~{\upshape\ref{thm:error_bound_Lr_general_gamma}} 
	the asymptotic dependence on fixed parameters is omitted.  
	In the notation of Section~{\upshape\ref{sec:general_SVE}}, 
	we have that $\xi = (h,M^{-\!}, M^{+\!})$ 
	and $\Xi = (0,h_0] \times \bbN \times \bbN$. 
\end{remark}

\begin{remark}\label{rmk:quadrature_error_fixed_total_nodes} 
	To calibrate the truncation errors 
	with the sinc quadrature error 
	in the $L^r$-bound for $\gamma\in(0,1)$, 
	see  
	Theorem~{\upshape\ref{thm:error_bound_Lr_general_gamma}\ref{item:thm:error_bound_Lr_general_gamma-ii}}, 
	we choose, for $h\in (0,h_0]$, 
	\begin{equation}
	\label{eq:sinc_quad_kernel-cali} 
		M^{-\!}  
		:= 
		\biggl\lceil \frac{\pi^2 - \delta}{\gamma h^2} \biggr\rceil, 
		\qquad
		M^{+\!}  
		:= 
		\biggl\lceil \frac{\pi^2 - \delta}{\bigl(\frac{1}{r} - \gamma - \beta\bigr) h^2} \biggr\rceil. 
	\end{equation}
	This choice 
	reduces the bound of 
	Theorem~{\upshape\ref{thm:error_bound_Lr_general_gamma}\ref{item:thm:error_bound_Lr_general_gamma-ii}}, 
	for all $t\in(0,T]$ and $s\in[0,t]$, 
	to 
	\begin{equation*}  
		\bigl\|
		k_\gamma -  \widehat{k}_\gamma^{(h,M^{-\!},M^{+\!})}
		\bigr\|_{L^r(s,t)}
		\leq
		3 \, c'' \, 
		e^{- (\pi^2 - \delta)/h} \, (t-s)^\beta. 
	\end{equation*}
	The total number of quadrature nodes 
	$N = M^{-\!} + M^{+\!} + 1$ 
	satisfies, under this calibration, 
	$N < 
	\tfrac{\pi^2 - \delta}{h^2} 
	\bigl( \gamma^{-1} + \bigl( \tfrac{1}{r} - \gamma -\beta\bigr)^{-1} \bigr) 
	+ 3$. 
	Therefore, 
	for all $t\in(0,T]$ and $s\in[0,t]$, 
	the error bound 
	may be expressed in terms of $N$ as 
	\begin{equation*} 
		\bigl\|
		k_\gamma -  \widehat{k}_\gamma^{(h,M^{-\!},M^{+\!})}
		\bigr\|_{L^r(s,t)}
		\leq
		3 \, c''  
		\exp\Biggl( - \sqrt{\frac{(\pi^2 - \delta) \, \gamma \, \bigl( \frac{1}{r} - \gamma - \beta \bigr)}{\frac{1}{r} - \beta}} \, \cdot \sqrt{N-3}  \Biggr)
		\, (t-s)^{\beta}.  
	\end{equation*} 
\end{remark}


\section{Markovian approximations for stochastic Volterra equations}
\label{sec:markov_approx} 

In this section we 
reparametrize and rescale 
the fractional kernel $k_\gamma$ 
in \eqref{eq:fracker}, for $\gamma\in(0,1)$, 
using the \emph{Hurst parameter} 	
$H\in\bigl(-\tfrac{1}{2},\tfrac{1}{2}\bigr)$, 
\begin{equation}\label{eq:fracker_hurst} 
	k^{H\!}(t)
	:= 
	\frac{1}{\Gamma(H + \nicefrac{1}{2})} \,  
	k_{\frac{1}{2}-H}(t) 
	= 
	\frac{1}{\Gamma(H + \nicefrac{1}{2})} \,  t^{H-\frac{1}{2}}, 
	\qquad 
	t\in(0,\infty),  
\end{equation} 
and we construct Markovian approximations
for stochastic Volterra 
processes \eqref{eq:SVE_def_eq_conv} 
with fractional convolution kernels  
of Hurst parameters 
$H_1\in\bigl(-\tfrac{1}{2},\tfrac{1}{2}\bigr)$, 
$H_2\in\bigl(0,\tfrac{1}{2}\bigr)$. 
As noted in \eqref{eq:intr_RLfBM}, 
an important member 
of this class of processes  
is the Riemann--Liouville fractional 
Brownian motion $B^H\!$, 
for which in \eqref{eq:SVE_def_eq_conv} 
we may set 
$X_0 = 0$,
$b \equiv 0$,  
$\sigma \equiv 1$, and 
${k_1 = k_2 = k^H\!}$. 
We recall that 
$B^H$ is pathwise H\"older continuous 
of order $\nu$
on the interval $[0,T]$, 
for all $\nu\in(0,H)$ 
and every $T\in(0,\infty)$.   

We will approximate the fractional 
kernel \eqref{eq:fracker_hurst} 
as discussed 
in Section~{\upshape\ref{sec:sinc-fractional}}. 
In Subsection~{\upshape\ref{subsec:markov_approx_intro}}, 
we first recall the concept of 
Markovian approximations 
from~\cite{abijabereleuch2018, 
	alfonsikebaier2022, 
	bayerbreneis2023} 
which allows us 
to approximate a stochastic Volterra equation
using a system of ordinary stochastic differential equations, 
and we relate this concept  
to the sinc quadrature 
in \eqref{eq:sinc_quad_kernel}. 
In Subsections~{\upshape\ref{subsec:markov_approx_diff}} 
and~{\upshape\ref{subsec:markov_approx_holder}},
we apply the sinc quadrature 
error estimates derived in 
Subsection~{\upshape\ref{subsec:sinc_quad_trunc}} 
to the results of
Subsections~{\upshape\ref{subsec:general_SVE_diff}} 
and~{\upshape\ref{subsec:general_SVE_holder_diff}} 
and obtain strong convergence    
in various norms, including 
pathwise H\"older convergence.  


\subsection{Sinc-Markovian approximations for the fractional kernel}
\label{subsec:markov_approx_intro} 

Approximations of 
stochastic Volterra processes 
with convolution kernels 
of completely monotone 
type 
can be obtained by 
replacing the integral representation 
\eqref{eq:compl_mon_type_ker} 
of the kernel 
by a finite sum. 
This motivates to consider, 
for $N \in \bbN$,  
${w = (w_1,\ldots,w_N)\in [0, \infty)^N\!}$,
and
$\rho=(\rho_1,\ldots,\rho_N) \in [0, \infty)^N\!$, 
the kernel approximation 
\begin{equation}
\label{eq:markovian_approximation_kernel}
	\widehat{k}^{(N, w,\rho)}(t) 
	= 
	\sum_{n =1}^N w_n  e^{-\rho_n  t}. 
\end{equation} 
Correspondingly, 
a stochastic Volterra process 
with kernels of completely 
monotone type may be 
approximated 
by a Volterra process with 
kernels of the form 
\eqref{eq:markovian_approximation_kernel}. 
Since such an approximation 
can be represented 
as a linear functional applied
to a finite-dimensional Markov process, 
this methodology is referred to 
as \emph{Markovian approximations} 
in the literature \cite{abijabereleuch2018, 
	alfonsikebaier2022, 
	bayerbreneis2023}. 
The following proposition 
summarizes this result and 
is a straightforward extension 
of \cite[Proposition 2.1(2)]{alfonsikebaier2022} 
to the case of different (numbers of) nodes  
for scalar kernels 
(in \cite{alfonsikebaier2022}, 
$N_1=N_2$ and $\rho_1 = \rho_2$).

\begin{proposition}\label{prop:MA_SDE}
	Let $N_1, N_2 \in \bbN$,  
	$w_1, \rho_1 \in [0,\infty)^{N_1\!}$, 
	and 
	$w_2, \rho_2 \in [0,\infty)^{N_2\!}$. 
	Define the functions
	$\widehat{k}_1^{(N_1, w_1, \rho_1)\!}, 
	\, 
	\widehat{k}_2^{(N_2, w_2, \rho_2)\!} 
	\colon (0,\infty) \to [0,\infty)$
	according to 
	\eqref{eq:markovian_approximation_kernel}.
	
	Let $p \in [2,\infty)$ and 
	$T\in (0,\infty)$. 
	Suppose that 
	$\widehat{X}\colon[0,\infty) \times \Omega\to\bbR^m$ 
	is a continuous stochastic Volterra process
	with initial value $X_0 \in L^p(\Omega,\cF_0,\bbP)$,
	coefficients $b,\sigma$ satisfying 
	Assumption~{\upshape\ref{ass:lipschitz_coef}} 
	with the constant $L \in (0,\infty)$,
	and convolution kernels $\widehat{k}_1^{(N_1, w_1, \rho_1)\!}$,  
	$\widehat{k}_2^{(N_2, w_2, \rho_2)\!}$.
	Then, up to indistinguishability,
	$\widehat{X}$ has the representation,
	\begin{equation*}
		\forall t \in [0,T]: 
		\quad 
		\widehat{X}_t = X_0 + Y_t w_1 + Z_t w_2, 
	\end{equation*}
	where the matrix-valued process
	$(Y,Z) \colon [0,T] \times \Omega \to \bbR^{m \times N_1\!} \times \bbR^{m \times N_2\!}$
	solves the system of coupled 
	stochastic differential equations (SDEs) 
	given by 
	\begin{equation}\label{eq:OU_MA_SDE} 
		\begin{cases}
			\diff{Y}_t^i = - \rho_{1,i} Y^i_t \diff{t} + b(X_0 + Y_t w_1 + Z_t w_2) \diff{t},
			&\quad i \in \{1,\ldots,N_1\}, 
			\\
			\diff{Z}_t^n = - \rho_{2,n} Z^n_t \diff{t} + \sigma(X_0 + Y_t w_1 + Z_t w_2) \diff{W}_t,
			&\quad n \in \{1,\ldots,N_2\},
		\end{cases} 
	\end{equation}
	for $t\in[0,T]$, 
	with initial condition $(Y_0, Z_0) = (0,0)$. 
	In \eqref{eq:OU_MA_SDE}, 
	$Y_t^i$ and $Z_t^n$ denote
	the $i$-th and $n$-th 
	$\mathbb{R}^m$-valued columns
	of the stochastic matrices $Y_t$ and $Z_t$.
\end{proposition}  

\begin{proof} 
	We first define the mappings 
	$\widetilde{b}, \, \widetilde{\sigma} 
	\colon \Omega \times \! 
	\bigl(\bbR^{m\times N_1\!} 
	\times \bbR^{m\times N_2} \bigr) 
	\to \bbR^m$ 
	by 
	\begin{equation*}   
		\widetilde{b}(\omega, (A, B) )
		:= 
		b( X_0(\omega)  + A w_1 + B w_2 ), 
		\quad 
		\widetilde{\sigma}(\omega, (A, B) )
		:= 
		\sigma( X_0(\omega)  + A w_1 + B w_2 ), 
	\end{equation*} 
	and note that 
	$\widetilde{b}$, $\widetilde{\sigma}$ 
	are 
	$\cF_0\otimes \cB(\bbR^{m\times N_1\!} \times  \bbR^{m\times N_2})$-measurable. 
	For every $\omega\in\Omega$ 
	and all 
	$(A,B), (A'\!, B') \in 
	\bbR^{m\times N_1\!} \times  \bbR^{m\times N_2\!}$, 
	we obtain for $\widetilde{b}$ 
	(and similarly for $\widetilde{\sigma}$) that  
	\begin{align*}   
		&\bigl\| 
		\widetilde{b}(\omega, (A, B) )
		- 
		\widetilde{b}(\omega, (A'\!, B') )
		\bigr\|_{\bbR^m} 
		\leq 
		L \, 
		\| (A-A')w_1 + (B - B')w_2 \|_{\bbR^m} 
		\\
		&\hspace*{25mm}\leq 
		L \, 
		\bigl( \|w_1\|_{\bbR^{N_1\!}}^2 		
		+ \|w_2\|_{\bbR^{N_2\!}}^2 \bigr)^{\nicefrac{1}{2}} 
		\bigl( 
		\| A - A' \|_{\bbR^{m\times N_1\!}}^2  
		+ 
		\| B - B'\|_{\bbR^{m\times N_2\!}}^2 
		\bigr)^{\nicefrac{1}{2}} \! . 
	\end{align*} 
	By \cite[Theorem~3.1.1]{LiuRoeckner2015} 
	there exists 
	an up to indistinguishability unique 
	continuous solution of the 
	SDE system \eqref{eq:OU_MA_SDE} 
	and applying It\^o's formula 
	to the processes 
	$e^{\rho_{1,i} t} Y_t^i$ 
	and 
	$e^{\rho_{2,n} t} Z_t^n$ 
	shows that they satisfy 
	the variation of constants formula,
	i.e.,  
	\begin{align*}
		Y^i_t 
		& = \int_0^t e^{-\rho_{1,i} (t-s)} b(X_0 + Y_s w_1 + Z_s w_2) \diff{s},
		&&\bbP\text{-a.s.}, 
		\quad\;\;  
		i \in \{1,\ldots,N_1\},  
		\\
		Z^n_t
		& = \int_0^t e^{-\rho_{2,n} (t-s)} \sigma(X_0 + Y_s w_1 + Z_s w_2) \diff{W}_s,
		&&\bbP\text{-a.s.}, 
		\quad\;\; 
		n \in \{1,\ldots,N_2\},  
	\end{align*}
	for all $t\in[0,T]$. 
	Now define the continuous process
	$\widetilde{X}\colon [0,T] \times \Omega \to \bbR^m$
	by
	\begin{equation*}
		\widetilde{X}_t = X_0 + Y_t w_1 + Z_t w_2.
	\end{equation*}
	Substituting the above integral representations  
	of $Y$ and $Z$ into $\widetilde{X}$, 
	and using linearity of the integrals,
	we find that $\widetilde{X}$ 
	is a continuous stochastic Volterra process 
	on the interval $[0,T]$
	with initial value $X_0$,
	coefficients $b,\sigma$,
	and convolution kernels 
	$\widehat{k}_1^{(N_1, w_1, \rho_1)\!}$ 
	and 
	$\widehat{k}_2^{(N_2, w_2, \rho_2)\!}$. 
	Since the same holds for~$\widehat{X}$, 
	and since both~$\widehat{X}$ and~$\widetilde{X}$ 
	are pathwise continuous, 
	they are indistinguishable.  
\end{proof}

\begin{remark} 
	The stochastic process 
	$\Psi \colon [0,T] \times \Omega 
	\to \bbR^{m\!} \times \bbR^{m \times N_1\!} \times \bbR^{m \times N_2\!}$, 
	given by 
	$\Psi_t := (X_0, Y_t, Z_t)$, 
	where $(Y,Z)$ is the solution 
	of \eqref{eq:OU_MA_SDE}, 
	solves an SDE system 
	with coefficients 
	independent of $\omega$ 
	and, hence, is 
	a Markov process 
	\cite[Proposition~4.3.3]{LiuRoeckner2015}.  
\end{remark} 

In the next assumption 
we specify the kernels  
of the stochastic Volterra processes 
and of the Markovian 
approximations that we consider 
in this section.

\begin{assumption}\label{ass:fractional_sinc}
	Let $H_1 \in \bigl(-\tfrac{1}{2},\tfrac{1}{2}\bigr)$ 
	and $H_2 \in \bigl(0,\tfrac{1}{2}\bigr)$, 
	let the fractional kernels $k^{H_1}$ 
	and $k^{H_2}$ 
	be defined 
	according to \eqref{eq:fracker_hurst}, and 
	let $k_1,k_2 \colon (0,\infty) \to \bbR$ 
	be given by  
	\begin{equation*}
		k_1(t)  := k^{H_1}(t),
		\qquad
		k_2(t)  := k^{H_2}(t),
		\qquad\quad 
		t \in (0,\infty). 
	\end{equation*}
	Let $h_0\in(0,\infty)$ 
	and $\delta\in\bigl(0,\pi^2\bigr)$ be given, 
	define $\Xi:=(0,h_0] \times \bigl[0, \bigl(H_1+\frac{1}{2}\bigr)\wedge H_2\bigr)$, 
	and let 
	$\bigl( \widehat{k}^{(h,\beta)\!}_1,  
	\widehat{k}^{(h,\beta)}_2 \bigr)_{(h,\beta)\in\Xi}$ 
	denote the corresponding family  
	of sinc approximations, 
	\begin{equation}\label{eq:sinc_quad_fractional_kernel}
		\widehat{k}^{(h,\beta)}_j(t) 
		:= \frac{h}{\Gamma(H_j+\nicefrac{1}{2})\Gamma(\nicefrac{1}{2}-H_j)} 
		\sum_{\ell  = -M_j^{-\!}(h)}^{M_j^{+\!}(h,\beta)} 
		e^{-t \, e^{\ell h}}
		e^{\left( \frac{1}{2} - H_j \right) \ell h} ,  
		\quad\;  
		j\in\{1,2\}, 
	\end{equation} 
	where, for $j\in\{1,2\}$, we set 
	\begin{equation}\label{eq:sinc_quad_fractional_kernel-cali}
		M_j^{-\!}(h) 
		:= 
		\biggl\lceil \frac{\pi^2 - \delta}{\bigl(\frac{1}{2}-H_j\bigr) h^2} \biggr\rceil, 
		\qquad
		M_j^{+\!}(h,\beta) 
		:= 
		\biggl\lceil \frac{\pi^2 - \delta}{\bigl(\frac{1}{j} - \frac{1}{2} + H_j - \beta\bigr) h^2} \biggr\rceil. 
	\end{equation} 
\end{assumption} 

\begin{remark} 
	Motivated by the error 
	estimates of Section~{\upshape\ref{sec:general_SVE}}, 
	we have chosen the truncation 
	parameters in \eqref{eq:sinc_quad_fractional_kernel-cali} 
	to calibrate the 
	kernel error of $\widehat{k}^{(h,\beta)\!}_1$ 
	in the $L^1$-norm 
	and that 
	of $\widehat{k}^{(h,\beta)\!}_2$ 
	in the $L^2$-norm, 
	see 
	Remark~{\upshape\ref{rmk:quadrature_error_fixed_total_nodes}}. 
	We emphasize that, 
	since we impose only $L^1$-type 
	assumptions on $k_1$, 
	the Hurst parameter $H_1$ may be negative. 
	We furthermore note that 
	a higher value of $\beta\in\bigl[0, \bigl(H_1+\frac{1}{2}\bigr) \wedge H_2\bigr)$ 
	implies higher values of the truncation parameters 
	$M_1^{+\!}(h,\beta)$ and $M_2^{+\!}(h,\beta)$ 
	for the positive tails.
\end{remark} 

\begin{remark} 
	For $j\in\{1,2\}$,  
	the sinc quadrature approximation   
	\eqref{eq:sinc_quad_fractional_kernel}
	of the fractional kernel $k^{H_j}$ 
	can be expressed 
	in the form \eqref{eq:markovian_approximation_kernel}, 
	$\widehat{k}^{(h,\beta)}_j 
	= 
	\widehat{k}^{(N_j, w_j,\rho_j)\!}$, 
	by setting  
	\begin{equation*}
		N_j 
		:= 
		M_j^{-\!}+M_j^{+\!}+1,
		\quad\;\; 
		w_{j,n} 
		:= 
		\frac{h \, e^{\left(\frac{1}{2} - H_j \right) (n-M_j^{-\!}-1)h} }{ 
			\Gamma(H_j+\nicefrac{1}{2})\Gamma(\nicefrac{1}{2}-H_j) },
		\quad\;\;   
		\rho_{j,n} := e^{(n-M_j^{-\!}-1)h}, 
	\end{equation*} 
	and, consequently, 
	induces corresponding 
	\emph{sinc-Markovian approximations} 
	of stochastic Volterra processes 
	with fractional convolution kernels. 
\end{remark} 


\subsection{Strong convergence in pointwise and pathwise Lebesgue sense}
\label{subsec:markov_approx_diff}

In this subsection we apply   
the $L^r$-error bound of the 
sinc quadrature approximation in 
Theorem~{\upshape\ref{thm:error_bound_Lr_general_gamma}\ref{item:thm:error_bound_Lr_general_gamma-ii}} 
to the error estimates  
for the corresponding Volterra processes 
in $C^0([0,T];L^p(\Omega))$ 
and $L^p(\Omega;L^q(0,T))$ 
derived in Subsection~{\upshape\ref{subsec:general_SVE_diff}}.
For this purpose, 
we first need to verify that the kernels 
$k_1, k_2$ 
and
$\bigl( \widehat{k}^{(h,\beta)\!}_1, \widehat{k}^{(h,\beta)}_2 \bigr)_{(h,\beta) \in \Xi}$ 
introduced in Subsection~{\upshape\ref{subsec:markov_approx_intro}} 
satisfy Assumptions~{\upshape\ref{ass:assump_conv_kernels}} 
and~{\upshape\ref{ass:assu_co_unif_xi}}, respectively.

\begin{proposition}
	\label{prop:5_verif_assumps_2_3}
	Let $H_1 \in \bigl(-\frac{1}{2}, \frac{1}{2}\bigr)$,
	$H_2 \in \bigl(0, \frac{1}{2}\bigr)$,
	$h_0\in(0,\infty)$, and 
	$\delta \in \bigl( 0, \pi^2 \bigr)$.
	Assume that the kernels $k_1$, $k_2$ 
	and, for $(h,\beta)\in\Xi := 
	(0,h_0] \times \bigl[0, \bigl(H_1+\frac{1}{2}\bigr)  \wedge H_2\bigr)$, 
	the approximate kernels 
	$\widehat{k}^{(h,\beta)\!}_1$,  
	$\widehat{k}^{(h,\beta)\!}_2$  
	are defined as in 
	Assumption~{\upshape\ref{ass:fractional_sinc}}. 
	Then: 
	\begin{enumerate}[leftmargin=8mm, label={\upshape(\roman*)}]
		\item \label{item:5_verif_assu_i}
		The functions $k_1$ and $k_2$ satisfy
		Assumption~{\upshape\ref{ass:assump_conv_kernels}} 
		with 
		\begin{equation}\label{eq:alpha-eta-fractional} 
			\qquad 
			\eta_T := 
			\left( 1+2 h_0 \right) \widetilde{\eta}_T, 
			\quad\;\; 
			\alpha_T \equiv \alpha :=\bigl(H_1 + \tfrac{1}{2}\bigr) \wedge H_2,  
			\quad\;\; 
			T \in(0,\infty), 
		\end{equation} 
		where 
		\begin{equation}\label{eq:eta-tilde-fractional}  
			\qquad 
			\widetilde{\eta}_T  
			:= 
			\biggl(
			\frac{T^{H_1 + \frac{1}{2}-\alpha}}{\Gamma(H_1 + \nicefrac{3}{2})}
			\biggr) \vee 
			\biggl(
			\frac{T^{H_2 - \alpha}}{\sqrt{2H_2} \, \Gamma(H_2 + \nicefrac{1}{2})}
			\biggr), 
			\qquad 
			T\in(0,\infty). 
		\end{equation} 
		\item \label{item:5_verif_assu_ii}
		The family of functions
		$\bigl( \widehat{k}^{(h,\beta)\!}_1, \widehat{k}^{(h,\beta)}_2 \bigr)_{(h,\beta) \in \Xi}$
		satisfies Assumption~{\upshape\ref{ass:assu_co_unif_xi}} 
		with the same families of constants 
		$(\eta_T)_{T>0}$ and $(\alpha_T)_{T>0}$ 
		as in \eqref{eq:alpha-eta-fractional}. 
	\end{enumerate}
\end{proposition} 

\begin{proof}
	\textbf{Proof of~{\upshape\ref{item:5_verif_assu_i}}:}
	Recall that, for $j \in \{1,2\}$ and $t \in (0,\infty)$,
	$k_j(t) 
	= 
	\frac{ t^{H_j - \frac{1}{2}}}{ 
		\Gamma(H_j + \nicefrac{1}{2})}$.
	These functions are non-negative and non-increasing 
	on $(0,\infty)$.
	Fix an arbitrary ${T \in (0,\infty)}$ and 
	let $\eta_T, \widetilde{\eta}_T$, 
	and $\alpha_T \equiv \alpha$ 
	be as in \eqref{eq:alpha-eta-fractional}--\eqref{eq:eta-tilde-fractional}.  
	Then,  
	\begin{equation*}
		\| k_j \|_{L^j(0,t)} 
		= 
		\bigl(   
		j \bigl( H_j - \tfrac{1}{2} \bigr) + 1 
		\bigr)^{-\frac{1}{j}} \, 
		\frac{t^{H_j - \frac{1}{2} + \frac{1}{j}}}{\Gamma(H_j + \nicefrac{1}{2})} 
		\leq 
		\widetilde{\eta}_T \, t^\alpha , 
		\qquad 
		t \in [0,T], 
	\end{equation*}
	holds for $j\in\{1,2\}$.  
	Since $\widetilde{\eta}_T < \eta_T$, 
	this 
	shows that $k_1$, $k_2$ satisfy Assumption~{\upshape\ref{ass:assump_conv_kernels}} 
	with the families of constants $(\eta_T)_{T > 0}$
	and $(\alpha_T \equiv \alpha)_{T>0}$  
	as in \eqref{eq:alpha-eta-fractional}. 
	
	\textbf{Proof of~{\upshape\ref{item:5_verif_assu_ii}}:}
	Let $(h,\beta)\in\Xi = 
	(0,h_0] \times \bigl[0, \bigl(H_1+\frac{1}{2}\bigr)  \wedge H_2\bigr)$. 
	Then, by \eqref{eq:sinc_quad_fractional_kernel},  
	\begin{equation*}
		\widehat{k}_j^{(h,\beta)}(t) 
		= 
		\frac{h}{\Gamma(H_j + \nicefrac{1}{2})} 
		\sum_{\ell  = -M_j^{-\!}(h)}^{M_j^{+\!}(h,\beta)} 
		g_{\frac{1}{2}-H_j}(t,\ell h), 
		\qquad 
		t \in (0,\infty), 
		\quad 
		j\in\{1,2\}, 
	\end{equation*}
	where $g_{\frac{1}{2}-H_j}$ is defined as in \eqref{eq:def_g_fracker}.
	The function 
	$t\mapsto g_{\frac{1}{2} - H_j}(t,x)$ 
	is strictly positive, for every $x\in\bbR$. 
	Moreover, the partial derivative 
	of $g_{\frac{1}{2} - H_j}$ with respect to $t$ 
	is strictly negative for all $(t,x) \in (0,\infty)\times\bbR$.
	By linearity, we conclude that 
	$\widehat{k}^{(h,\beta)}_j$ is 
	non-negative and non-increasing
	on $(0,\infty)$, 
	for all $j\in\{1,2\}$  
	and every $(h,\beta) \in \Xi$. 
	
	Moreover, by 
	Lemma~{\upshape\ref{lma:exponential_maximum_bound}}, 
	for $j\in\{1,2\}$, and for every $t\in(0,\infty)$, 
	the function $x \mapsto g_{\frac{1}{2} - H_j}(t,x)$
	has a unique global maximum at 
	$x^*_j = \log\Bigl(\tfrac{ 1/2 - H_j }{t} \Bigr)$. 
	Consequently, it is monotonically increasing on 
	$(-\infty,\ell^*_j h] \subseteq (-\infty, x^*_j]$
	and monotonically decreasing on 
	$[(\ell^*_j +1)h, \infty) \subset (x^*_j, \infty)$, 
	where we set   
	$\ell^*_j := 
	\bigl\lfloor \tfrac{x^*_j}{h} \bigr\rfloor$. 
	We therefore can bound 
	$\widehat{k}^{(h,\beta)}_j(t)$, 
	for $j\in\{1,2\}$, 
	for every $(h,\beta)\in\Xi$, 
	and for all $t\in(0,\infty)$,	
	as follows 
	\begin{align*}
		\Gamma(H_j + \nicefrac{1}{2}) \, 
		&\widehat{k}^{(h,\beta)}_j(t)
		\leq 
		h \sum_{\ell \in \bbZ} 
		g_{\frac{1}{2} - H_j}(t, \ell h) 
		\\
		&\leq 
		h \, g_{\frac{1}{2} - H_j}(t, \ell^*_j h)
		+
		h \, g_{\frac{1}{2} - H_j}(t, (\ell^*_j+1) h)
		+ 
		\int_{-\infty}^{\infty} g_{\frac{1}{2} - H_j}(t, x) \diff{x}. 
	\end{align*} 
	By construction,
	$\int_{-\infty}^{\infty} g_{\frac{1}{2} - H_j}(t, x) \diff{x} 
	= t^{H_j - \frac{1}{2}}$ 
	and by 
	Lemma~{\upshape\ref{lma:exponential_maximum_bound}}, 
	for all $x\in\bbR$, 
	\begin{equation*}
		h \, g_{\frac{1}{2} - H_j}(t, x) 
		\leq 
		h \, g_{\frac{1}{2} - H_j}(t, x_j^*) 
		=  
		\frac{h}{\Gamma(\nicefrac{1}{2} - H_j)} \, 
		\biggl( \frac{ \nicefrac{1}{2} - H_j }{ e } 
		\biggr)^{\frac{1}{2} - H_j} 
		t^{H_j - \frac{1}{2}} 
		<  
		h_0 \, t^{H_j - \frac{1}{2}} , 
	\end{equation*}
	where we also used that 
	$\Gamma(\nicefrac{1}{2} - H_j) > 1$  
	and $e^{-1} (\nicefrac{1}{2} - H_j) < 1$. 
	Since the preceding 
	estimates hold 
	for every 
	$(h,\beta) \in \Xi 
	= (0,h_0] \times \bigl[0, \bigl(H_1+\frac{1}{2}\bigr) \wedge H_2\bigr)$, 
	we conclude that 
	\begin{equation*}
		\forall (h,\beta)\in\Xi, 
		\quad\;\; 
		\forall t\in(0,\infty): 
		\qquad  
		\widehat{k}^{(h,\beta)}_j (t)
		\leq 
		(1 + 2 h_0 ) \, k_j(t). 
	\end{equation*}
	We therefore set $\overline{k}_j(t) 
	:= (1 + 2 h_0  ) \, k_j(t)$, 
	and it then follows from the proof of 
	part~{\upshape\ref{item:5_verif_assu_i}} 
	that, for $j\in\{1,2\}$, 
	for all $T\in(0,\infty)$, 
	and every $t\in[0, T]$, 
	we have that 
	$\| \overline{k}_j \|_{L^j(0,t)} 
	\leq (1+2 h_0)\, \widetilde{\eta}_T \, t^\alpha 
	= \eta_T \, t^\alpha$, 
	which completes the proof 
	of~{\upshape\ref{item:5_verif_assu_ii}}.
\end{proof}

Having verified 
Assumptions~{\upshape\ref{ass:assump_conv_kernels}} 
and~{\upshape\ref{ass:assu_co_unif_xi}}  
for the fractional kernels 
and the corresponding 
family of sinc approximations, 
by means of Corollaries~{\upshape\ref{cor:strong_error_Lp_conv}} 
and~{\upshape\ref{cor:path_dependent_Lp}}, 
we can now establish  
strong convergence 
of the associated sinc-Markovian 
approximations to    
stochastic Volterra processes
with fractional kernels, 
pointwise in the norm of $C^0([0,T];L^p(\Omega))$, 
and pathwise 
in the norm of $L^p(\Omega;L^q(0,T))$, 
respectively.

\begin{theorem}\label{thm:C0_error_fracker} 
	Let $H_1 \in \bigl(-\frac{1}{2}, \frac{1}{2}\bigr)$,
	$H_2 \in \bigl(0, \frac{1}{2}\bigr)$, 
	$h_0\in (0,\infty)$,  
	and $\delta\in\bigl(0, \pi^2\bigr)$ 
	be given, and  
	let the kernels $k_1,k_2\colon(0,\infty)\to\bbR$ 
	and the family  
	$\bigl( \widehat{k}^{(h,\beta)\!}_1,  
	\widehat{k}^{(h,\beta)}_2 \bigr)_{(h,\beta)\in\Xi}$ 
	be defined 
	as in Assumption~{\upshape\ref{ass:fractional_sinc}}, 
	where 
	$\Xi:=(0,h_0] \times \bigl[0, \bigl(H_1+\frac{1}{2}\bigr)\wedge H_2\bigr)$.  
	Define  
	$\alpha \in \bigl(0, \frac{1}{2}\bigr)$ 
	as in \eqref{eq:alpha-eta-fractional}, 
	and let $p\in(\nicefrac{1}{\alpha}, \infty)$.  
	
	Suppose that $X \colon [0,\infty) \times \Omega \to \bbR^m$
	is a continuous stochastic Volterra process with
	initial value $X_0 \in L^p(\Omega,\cF_0,\bbP)$,
	coefficients $b, \sigma$ satisfying
	Assumption~{\upshape\ref{ass:lipschitz_coef}},  
	and convolution kernels $k_1, k_2$. 
	For every $h\in(0,h_0]$, 
	let  
	$\widehat{X}^{(h, 0)}
	\colon [0,\infty) \times \Omega \to \bbR^m$ 
	be a continuous stochastic Volterra process 
	with the same initial value 
	and coefficients as~$X$, 
	and with convolution kernels
	$\widehat{k}^{(h,0)\!}_1$ 
	and 
	$\widehat{k}^{(h,0)\!}_2$. 
	
	Then, 
	for every $T\in(0,\infty)$, 
	there exists a constant $C \in (0,\infty)$, 
	independent of $h\in(0,h_0]$ 
	and $q\in[1,p]$, 
	such that, 
	for all $h \in (0,h_0]$ 
	and every~$q\in[1,p]$,
	\begin{equation*}
		\sup_{t\in[0,T]} 
		\bigl\| 
		X_t - \widehat{X}^{(h,0)}_t 
		\bigr\|_{L^p(\Omega)} 
		+ 
		\bigl\| X - \widehat{X}^{(h,0)} 
		\bigr\|_{L^p(\Omega;L^q(0,T))}
		\leq 
		C  e^{- (\pi^2 - \delta)/h} . 
	\end{equation*} 
\end{theorem}

\begin{proof}
	By Proposition~{\upshape\ref{prop:5_verif_assumps_2_3}}, 
	the kernels 
	$k_1$, $k_2$ satisfy 
	Assumption~{\upshape\ref{ass:assump_conv_kernels}},
	and the family
	$\bigl( \widehat{k}^{(h,0)\!}_1,
	\widehat{k}^{(h,0)}_2 \bigr)_{h \in (0,h_0]}$
	fulfills Assumption~{\upshape\ref{ass:assu_co_unif_xi}},  
	with ${\alpha_T \equiv  
		\bigl(H_1+\tfrac{1}{2}\bigr) \wedge H_2 
		\in \bigl(0, \tfrac{1}{2}\bigr)}$.
	Hence,  Corollaries~{\upshape\ref{cor:strong_error_Lp_conv}} 
	and~{\upshape\ref{cor:path_dependent_Lp}} 
	are applicable, yielding, for every $T\in(0,\infty)$, 
	the existence 
	of a constant $\widehat{C} \in (0,\infty)$ such that,
	for all $h \in (0,h_0]$ and all~$q\in[1,p]$, 
	\begin{align*}
		\sup_{t\in[0,T]} 
		\bigl\| X_t - \widehat{X}^{(h,0)}_t 
		\bigr\|_{L^p(\Omega)}
		&\leq 
		\widehat{C} 
		\Bigl( 
		\bigl\| k_1 - \widehat{k}^{(h,0)}_1
		\bigr\|_{L^1(0,T)} 
		+ 
		\bigl\| k_2 - \widehat{k}^{(h,0)}_2 
		\bigr\|_{L^2(0,T)} 
		\Bigr), 
		\\
		\bigl\| X - \widehat{X}^{(h,0)} 
		\bigr\|_{L^p(\Omega;L^q(0,T))}
		&\leq 
		\widehat{C} \, T^{\frac{1}{q}} 
		\Bigl( 
		\bigl\| k_1 - \widehat{k}^{(h,0)}_1 
		\bigr\|_{L^1(0,T)} 
		+ 
		\bigl\| k_2 - \widehat{k}^{(h,0)}_2 
		\bigr\|_{L^2(0,T)} 
		\Bigr) . 
	\end{align*}
	The claim then follows from  
	Theorem~{\upshape\ref{thm:error_bound_Lr_general_gamma}\ref{item:thm:error_bound_Lr_general_gamma-ii}} 
	and Remark~{\upshape\ref{rmk:quadrature_error_fixed_total_nodes}}  
	with $\beta=0$, 
	showing that, for 
	all $T\in(0,\infty)$ 
	and $j\in\{1,2\}$, 
	there exists   
	a constant $c_j''\in(0,\infty)$ such that 
	\begin{equation*} 
		\forall h \in (0,h_0]: 
		\qquad 
		\bigl\| k_j - \widehat{k}_j^{(h,0)} 
		\bigr\|_{L^j(0,T)}
		\leq 
		3\, c_j'' \, 
		e^{-(\pi^2-\delta)/h} . 
	\end{equation*}  
	Here, we have used 
	the fact that, for $j\in\{1,2\}$, 
	the truncation parameters 
	$M_j^{-\!}(h)$ and 
	$M_j^{+\!}(h,0)$  
	of $\widehat{k}_j^{(h,0)\!}$ 
	in 
	\eqref{eq:sinc_quad_fractional_kernel-cali}  
	are chosen such that  
	the  $L^j$-error is calibrated.  
\end{proof}

\begin{remark}
	\label{rmk:calibrate_all_exponents_beta0} 
	Similarly as for the estimate 
	of the kernel difference   
	in Remark~{\upshape\ref{rmk:quadrature_error_fixed_total_nodes}}, 
	we may  
	use the calibration \eqref{eq:sinc_quad_fractional_kernel-cali} 
	to express the error bound  
	of Theorem~{\upshape\ref{thm:C0_error_fracker}} 
	in terms of the numbers of quadrature nodes 
	$N_1 = N_1(h,0) := M_1^{-\!}(h) + M_1^{+\!}(h,0) + 1$ 
	and 
	$N_2 = N_2(h,0) := M_2^{-\!}(h) + M_2^{+\!}(h,0) + 1$
	used to approximate the kernels 
	$k_1$ and~$k_2$, respectively. 
	This shows that, 
	for every $T\in(0,\infty)$, 
	there exists a constant 
	$C\in(0,\infty)$, 
	independent of $h\in(0,h_0]$ and $q\in[1,p]$, 
	such that 
	\begin{align*}
		\sup_{t\in[0,T]} 
		&\bigl\| 
		X_t - \widehat{X}^{(h,0)}_t 
		\bigr\|_{L^p(\Omega)} 
		+ 
		\bigl\| X - \widehat{X}^{(h,0)} 
		\bigr\|_{L^p(\Omega;L^q(0,T))}
		\\
		&\qquad\leq 
		C \Bigl( 
		e^{- \sqrt{(\pi^2 - \delta) \left(\frac{1}{2}-H_1\right) \left( H_1 + \frac{1}{2} \right)} 
			\cdot \sqrt{N_1 - 3}}
		+ e^{- \sqrt{2 (\pi^2 - \delta) \left( \frac{1}{2}-H_2 \right) H_2} \cdot \sqrt{N_2 - 3}} 
		\Bigr). 
	\end{align*}  
\end{remark}

\begin{remark}[Comparison with \cite{alfonsikebaier2022}]
	The pointwise error bound  
	in $C^0([0,T];L^2(\Omega))$ 
	of~\cite[Theorem~3.1]{alfonsikebaier2022} 
	requires that the convolution kernel 
	$k_1$ in the drift 
	is square-integrable, i.e., $k_1 \in L^2(0,T)$,  
	and this excludes the range  
	$H_1\in\bigl(-\tfrac{1}{2}, 0\bigr]$.   
	Moreover, the constant in that result
	may depend on the 
	choice of the 
	approximations
	$\widehat{k}_1^\xi, \widehat{k}_2^\xi$ 
	(via the constant $\overline C$ 
	in~\cite[Eq.~(3.6)]{alfonsikebaier2022}). 
	Assumption~{\upshape\ref{ass:assu_co_unif_xi}}
	ensures uniformity in $\xi\in\Xi$.
\end{remark}


\subsection{Strong convergence in H\"older norms}
\label{subsec:markov_approx_holder}

To extend our strong error analysis
from the pointwise to the pathwise 
H\"older sense,  
we first establish that 
the sinc kernel approximations
introduced in Assumption~{\upshape\ref{ass:fractional_sinc}} 
satisfy the stricter  
conditions of 
Assumption~{\upshape\ref{ass:convker_holder_assu}}  
in Subsection~{\upshape\ref{subsec:general_SVE_holder_diff}}.
This is addressed in the following proposition.

\begin{proposition}\label{prop:verification_conv_holder} 
	Let $H_1 \in \bigl(-\frac{1}{2}, \frac{1}{2}\bigr)$,
	$H_2 \in \bigl(0, \frac{1}{2}\bigr)$,
	$h_0\in(0,\infty)$, and 
	$\delta \in \bigl( 0, \pi^2 \bigr)$.
	Assume that the kernels $k_1$, $k_2$ 
	and, for $(h,\beta)\in\Xi := 
	(0,h_0] \times \bigl[0, \bigl(H_1+\frac{1}{2}\bigr)  \wedge H_2\bigr)$, 
	the approximate kernels 
	$\widehat{k}^{(h,\beta)\!}_1$,  
	$\widehat{k}^{(h,\beta)\!}_2$  
	are defined as in 
	Assumption~{\upshape\ref{ass:fractional_sinc}}. 
	
	Fix $\beta_{\star} \in \bigl(0, \bigl(H_1+\frac{1}{2}\bigr)  \wedge H_2\bigr)$. 
	Then, $k_1, k_2$ 
	and the family $\bigl( \widehat{k}^{(h,\beta_\star)\!}_1,  
	\widehat{k}^{(h,\beta_\star)}_2 \bigr)_{h\in(0,h_0]}$ satisfy 
	Assumption~{\upshape\ref{ass:convker_holder_assu}}. 
	Moreover, for every $T \in (0,\infty)$, 
	the constants and functions 
	in Assumption~{\upshape\ref{ass:convker_holder_assu}} 
	may be chosen as follows: 
	$\eta_T\in(0,\infty)$ and 
	$\alpha_T\in\bigl(0,\frac{1}{2}\bigr)$  
	as in \eqref{eq:alpha-eta-fractional}, 
	$\beta_T := \beta_\star$, 
	$\eta_T' := \beta_\star^{-1}\!$, 
	$\theta_T(h) := 3 \, c_T \, e^{-(\pi^2 - \delta)/h}$, 
	$g_T(u) := u^{\beta_{\star\!} - 1}$. 
	Here, for all $T\in(0,\infty)$, 
	the constant $c_T \in(0,\infty)$ 
	is independent of $h\in(0,h_0]$, 
	and defined as  
	$c_T := c_{1}'' \vee c_{2}'' \vee c_{1}''' \vee c_{2}'''$, 
	where 
	$c_j'', c_j''' \in (0,\infty)$, $j\in\{1,2\}$,  
	are as in 
	Theorem~{\upshape\ref{thm:error_bound_Lr_general_gamma}\ref{item:thm:error_bound_Lr_general_gamma-ii}} and~{\upshape\ref{item:thm:error_bound_Lr_general_gamma-iii}}, 
	respectively, 
	with   
	$\gamma = \frac{1}{2} - H_j$, $r = j$, 
	and 
	$\beta = \beta_\star\in\bigl(0, \bigl( H_1 + \frac{1}{2}\bigr) \wedge H_2 \bigr)$.   
\end{proposition}

\begin{proof} 
	By Proposition~{\upshape\ref{prop:5_verif_assumps_2_3}},  
	$k_1, k_2$ and the family   
	$\bigl( \widehat{k}^{(h,\beta_\star)\!}_1, 
	\widehat{k}^{(h,\beta_\star)}_2 \bigr)_{h\in(0,h_0]}$ 
	satisfy Assumptions~{\upshape\ref{ass:assump_conv_kernels}} 
	and~{\upshape\ref{ass:assu_co_unif_xi}}, 
	respectively, 
	with $(\eta_T)_{T>0}$ and 
	$(\alpha_T)_{T>0}$ 
	as in \eqref{eq:alpha-eta-fractional}.  
	
	Furthermore, 
	for every $T\in(0,\infty)$, 
	and for $j\in\{1,2\}$, 
	by 
	Theorem~{\upshape\ref{thm:error_bound_Lr_general_gamma}\ref{item:thm:error_bound_Lr_general_gamma-ii}} 
	and Remark~{\upshape\ref{rmk:quadrature_error_fixed_total_nodes}}, 
	applied with 
	$\gamma = \tfrac{1}{2} - H_j$, 
	$r=j$, and $\beta=\beta_\star$, 
	there exists a constant $c_{j}''\in(0,\infty)$ 
	such that, for all $h \in (0,h_0]$,  
	\begin{equation*} 
		\forall t \in [0,T] : \quad 
		\bigl\| k_j - \widehat{k}_j^{(h,\beta_\star)} \bigr\|_{L^j(0,t)} 
		\leq 
		3\, c_{j}'' \, e^{-(\pi^2-\delta)/h} \, t^{\beta_\star} . 
	\end{equation*} 
	Thus, 
	Assumption~{\upshape\ref{ass:convker_holder_assu}\ref{item:diff_diff_holder_conv_i}} 
	is fulfilled  
	with 
	$\theta_T(h) 
	= 3 \, c_T 
	\, e^{-(\pi^2 - \delta)/h}$ 
	and 
	$\beta_T=\beta_\star$. 
	
	Finally, by 
	Theorem~{\upshape\ref{thm:error_bound_Lr_general_gamma}\ref{item:thm:error_bound_Lr_general_gamma-iii}}, 
	Remark~{\upshape\ref{rmk:quadrature_error_fixed_total_nodes}}, 
	and by noting that 
	$e^{-(\gamma+1)x} \leq e^{-\gamma x}$ 
	holds for all $\gamma, x \in(0,\infty)$, 
	we conclude that, 
	for every $T\in(0,\infty)$, 
	and for $j\in\{1,2\}$, 
	there exists a constant $c_j'''\in(0,\infty)$ 
	such that 	 
	\begin{equation*}
		\forall u \in (0,T): 
		\quad 
		\bigl\| \bigl(k_j - \widehat{k}^{(h,\beta_\star)}_j \bigr)' 
		\bigr\|_{L^j(u,T)} 
		\leq 
		3\, c_{j}''' \,
		e^{- (\pi^2 - \delta)/h} \, 
		u^{\beta_{\star\!} - 1}.
	\end{equation*}
	Since 
	$\| u\mapsto u^{\beta_{\star\!} - 1}\|_{L^1(0,t)} 
	= 
	\frac{1}{\beta_\star} \, t^{\beta_\star}\!$, 
	we find that also 
	Assumption~{\upshape\ref{ass:convker_holder_assu}\ref{item:diff_diff_holder_conv_ii}} 
	is satisfied  
	with 
	$\theta_T(h) = 3 \, c_T \, e^{-(\pi^2 - \delta)/h}$, 
	$g_T(u) = u^{\beta_{\star\!} - 1}$,   
	and 
	$\beta_T = \beta_\star$, 
	$\eta_T' = \beta_\star^{-1}$.  
\end{proof}

\begin{theorem}\label{thm:holder_bound_fracker}
	Let $H_1 \in \bigl(-\frac{1}{2}, \frac{1}{2}\bigr)$,
	$H_2 \in \bigl(0, \frac{1}{2}\bigr)$, 
	$h_0\in (0,\infty)$,  
	and $\delta\in\bigl(0, \pi^2\bigr)$ 
	be given, and  
	let the kernels $k_1,k_2\colon(0,\infty)\to\bbR$ 
	and the family  
	$\bigl( \widehat{k}^{(h,\beta)\!}_1,  
	\widehat{k}^{(h,\beta)}_2 \bigr)_{(h,\beta)\in\Xi}$ 
	be defined 
	as in Assumption~{\upshape\ref{ass:fractional_sinc}}, 
	where 
	$\Xi:=(0,h_0] \times \bigl[0, \bigl(H_1+\frac{1}{2}\bigr)\wedge H_2\bigr)$.  
	Define  
	$\alpha \in \bigl(0, \frac{1}{2}\bigr)$ 
	as in \eqref{eq:alpha-eta-fractional}, 
	and let $p\in(\nicefrac{1}{\alpha}, \infty)$.   
	
	Suppose that $X \colon [0,\infty) \times \Omega \to \bbR^m$
	is a continuous stochastic Volterra process with
	initial value $X_0 \in L^p(\Omega,\cF_0,\bbP)$,
	coefficients $b, \sigma$ satisfying
	Assumption~{\upshape\ref{ass:lipschitz_coef}},  
	and convolution kernels $k_1, k_2$. 
	Fix ${\beta_\star\in\bigl(0, \bigl( H_1 + \tfrac{1}{2}\bigr) \wedge H_2\bigr)}$ 
	and, for every $h\in(0,h_0]$, 
	let  
	$\widehat{X}^{(h, \beta_\star)}
	\colon [0,\infty) \times \Omega \to \bbR^m$ 
	be a continuous stochastic Volterra process 
	with the same initial value 
	and coefficients as~$X$, 
	and with convolution kernels
	$\widehat{k}^{(h,\beta_\star)\!}_1,  
	\widehat{k}^{(h,\beta_\star)\!}_2$.  
	
	Then, 
	for every $T\in(0,\infty)$, 
	there exists a constant $C \in (0,\infty)$, 
	independent of $h\in(0,h_0]$, 
	such that, 
	for all $h \in (0,h_0]$,  
	\begin{equation*} 
		\bigl\| X - \widehat{X}^{(h,\beta_\star)} 
		\bigr\|_{C^{\beta_{\star\!}}([0,T];L^p(\Omega))}
		\leq 
		C  e^{- (\pi^2 - \delta)/h} . 
	\end{equation*} 
	
	Moreover, if $\beta_\star > \nicefrac{1}{p}$,
	then, for every 
	$T\in(0,\infty)$ 
	and all 
	$\nu \in (0, \beta_{\star\!} - \nicefrac{1}{p})$,
	there exists a constant $C_\nu \in (0,\infty)$
	such that, for all $h \in (0,h_0]$, 
	\begin{equation*} 
		\bigl\| X - \widehat{X}^{(h,\beta_\star)}  
		\bigr\|_{L^p(\Omega;C^\nu([0,T]))}
		\leq 
		C_\nu 
		e^{- (\pi^2 - \delta)/h} . 
	\end{equation*}
\end{theorem} 

\begin{proof} 
	For a fixed 
	$\beta_{\star} \in \bigl(0, \bigl(H_1+\frac{1}{2}\bigr)  \wedge H_2\bigr)$, 
	by Proposition~{\upshape\ref{prop:verification_conv_holder}} above, 
	the functions $k_1, k_2$ 
	and the family $\bigl( \widehat{k}^{(h,\beta_\star)\!}_1,  
	\widehat{k}^{(h,\beta_\star)}_2 \bigr)_{h\in(0,h_0]}$ satisfy 
	Assumption~{\upshape\ref{ass:convker_holder_assu}}, 
	and, for each fixed $T\in(0,\infty)$, 
	the constants and functions 
	in Assumption~{\upshape\ref{ass:convker_holder_assu}} 
	may be chosen as follows: 
	$\eta_T\in(0,\infty)$  
	as in \eqref{eq:alpha-eta-fractional}, 
	$\alpha_T \equiv\alpha =\bigl(H_1 + \tfrac{1}{2}\bigr) \wedge H_2$, 
	$\beta_T = \beta_\star$, 
	$\eta_T' = \beta_\star^{-1}$, 
	and  
	$\theta_T(h) = 3 \, c_T \, e^{-(\pi^2 - \delta)/h}$, 
	$g_T(u) = u^{\beta_{\star\!} - 1}$, 
	where $c_T\in(0,\infty)$ 
	is independent of $h\in(0,h_0]$. 
	Thus, $\alpha_T\wedge \beta_T =\beta_\star$ 
	and the claim follows from   
	Theorem~{\upshape\ref{thm:SVE_holder_Lp_bound}}.  
\end{proof}

\begin{remark}\label{rmk:calibrate_all_exponents}
	As in Remark~{\upshape\ref{rmk:calibrate_all_exponents_beta0}}, 
	by means of the calibration \eqref{eq:sinc_quad_fractional_kernel-cali} 
	we may express the error estimate 
	of Theorem~{\upshape\ref{thm:holder_bound_fracker}} 
	in terms of the numbers of quadrature nodes 
	$N_1 = N_1(h,\beta_\star) := M_1^{-\!}(h) + M_1^{+\!}(h,\beta_\star) + 1$,  
	$N_2 = N_2(h,\beta_\star) := M_2^{-\!}(h) + M_2^{+\!}(h,\beta_\star) + 1$
	used to approximate the fractional kernels 
	$k_1$ and~$k_2$, respectively. 
	This shows that there exists a constant 
	$C\in(0,\infty)$, 
	independent of $h\in(0,h_0]$, 
	such that 
	\begin{align*} 
		&\bigl\| X - \widehat{X}^{(h,\beta_\star)} 
		\bigr\|_{C^{\beta_{\star\!}}([0,T];L^p(\Omega))}
		\\
		&\qquad\leq 
		C \Bigl( 
		e^{- \sqrt{(\pi^2 - \delta) \left(\frac{1}{2}-H_1\right) 
				\frac{ H_1 + 1/2 -\beta_\star }{1-\beta_\star} } 
			\cdot \sqrt{N_1 - 3}}
		+ e^{- \sqrt{ (\pi^2 - \delta) \left( \frac{1}{2}-H_2 \right) 
				\frac{ H_2 -\beta_\star}{1/2 - \beta_\star} } \cdot \sqrt{N_2 - 3}} 
		\Bigr). 
	\end{align*}   
	Notably, taking the limit 
	$\beta_{\star\!} \downarrow 0$
	in this expression  
	recovers the $C^0([0,T];L^p(\Omega))$-result, 
	see  
	Remark~{\upshape\ref{rmk:calibrate_all_exponents_beta0}}. 
	However, as the integrability criterion of 
	Assumption~{\upshape\ref{ass:convker_holder_assu}\ref{item:diff_diff_holder_conv_ii}} 
	requires that $\beta_\star > 0$, 
	we have discussed 
	the pointwise and H\"older bounds
	separately. 
\end{remark}


\section{Numerical experiments}\label{sec:numerical_experiments}

In this section we empirically verify 
the convergence results 
of the sinc quadrature approximations 
for the fractional kernel 
from Section~{\upshape\ref{sec:sinc-fractional}} 
and of the corresponding Markovian 
approximations discussed  
in Section~{\upshape\ref{sec:markov_approx}}.   
For the latter, we will focus on 
the Riemann--Liouville 
fractional Brownian motion \eqref{eq:intr_RLfBM} 
with varying Hurst parameter $H$, 
whose Gaussian distribution 
is uniquely determined by its covariance  
function. This covariance function 
admits a closed formula 
and, thus, exact simulation via  
Cholesky factorization. 
All numerical experiments were implemented in Python~3.13.3 using
\texttt{SciPy}~1.17.1 and \texttt{NumPy}~2.4.3. 
The code to reproduce these experiments
is openly available on Zenodo at \cite{corneilleetal26_holder_code}.

We assume that 
$H \in \bigl(-\frac{1}{2},\frac{1}{2}\bigr)$
is given, 
and we consider the sinc quadrature 
approximation \eqref{eq:sinc_quad_kernel} 
of the fractional kernel \eqref{eq:fracker_hurst}. 
To this end, we fix ${\delta := 0.01}$ and, 
for given integrability and regularity parameters 
$r \in \bigl[1, (\frac{1}{2} - H)^{-1}\bigr)$ and  
${\beta \in \bigl[0, \frac{1}{r} - \frac{1}{2} + H\bigr)}$,
we calibrate the discretization 
parameters~$(h, M^{-\!}, M^{+\!})$ 
as discussed in 
Remark~{\upshape\ref{rmk:quadrature_error_fixed_total_nodes}}. 
That is, 
for a  sequence of quadrature step sizes, 
\begin{equation}\label{eq:num_def_h_n}
	h_n 
	= 
	h_n(\beta,r)
	:= 
	\sqrt{ 
		\frac{ (\pi^2 - \delta)
			\Bigl( \bigl(\frac{1}{2}-H \bigr)^{-1} 
			+ 
			\bigl (\frac{1}{r} - \frac{1}{2} + H - \beta \bigr)^{-1} 
			\Bigr)}{n} } \, , 
	\quad\;  
	n\in\bbN, 
\end{equation}
we define $M^{-\!}$ and $M^{+\!}$ 
according to the 
calibration~\eqref{eq:sinc_quad_kernel-cali}. 
The sequence $(h_n)_{n\in\bbN}$ 
in \eqref{eq:num_def_h_n} 
is chosen such that 
the corresponding  
numbers of quadrature nodes
$N_n =  N_n(\beta,r) 
:= M^{-\!}(h_n) + M^{+\!}(h_n, \beta, r) + 1$ 
scale like $N_n \sim n$  
(with $n+1 \leq N_n \leq n+3$). 

We then consider 
the fractional kernel $k$ and its sinc approximation
$\widehat{k}^{(h_n, \beta, r)\!}$, 
i.e., 
\begin{equation}\label{eq:num_def_kernels}
	k(t) := k^H(t), 
	\qquad 
	\widehat{k}^{(h_n, \beta, r)} (t) 
	:= 
	\Gamma(H+\nicefrac{1}{2})^{-1} \,
	\widehat{k}^{(h_n, M^{-\!}(h_n), M^{+\!}(h_n,\beta,r))}_{\frac{1}{2}-H}(t), 
\end{equation}
for all $t\in(0,\infty)$, 
where $k^H$ and 
$\widehat{k}^{(h_n, M^{-\!}, M^{+\!})}_{\frac{1}{2}-H}$ 
are defined
as in \eqref{eq:fracker_hurst}
and
\eqref{eq:sinc_quad_kernel}. 


\subsection{Sinc quadrature errors}
\label{subsec:num_exp_det_quad_err} 

We let 
$H \in \bigl(-\frac{1}{2},\frac{1}{2} \bigr)$, 
and consider $T:=1$. 
The results of 
Theorem~{\upshape\ref{thm:error_bound_Lr_general_gamma}\ref{item:thm:error_bound_Lr_general_gamma-ii}} 
and
Remark~{\upshape\ref{rmk:quadrature_error_fixed_total_nodes}},
imply that, 
for all 
${r \in \bigl[1, (\frac{1}{2} - H)^{-1}\bigr)}$,
for every 
$\beta \in \bigl[0, \frac{1}{r} - \frac{1}{2} + H\bigr)$, 
and for all $(t,s) \in \triangle[0,1]$, 
the $L^r(s,t)$-error bound, 
\begin{equation}\label{eq:num_exp_sinc_conv} 
	\bigl\| k - \widehat{k}^{(h_n, \beta, r)} 
	\bigr\|_{L^r(s,t)}
	\lesssim 
	\exp\bigl(-a_{H,\beta,r} \cdot \sqrt{N_n - 3} \bigr)
	\, (t-s)^\beta, 
\end{equation}
holds 
for the sinc quadrature approximation 
of the fractional kernel 
in \eqref{eq:num_def_kernels}, 
where
\begin{equation}\label{eq:sinc_coef}
	a_{H,\beta,r} 
	= 
	\sqrt{\frac{ (\pi^2 - \delta) 
			\bigl(\frac{1}{2} - H \bigr)
			\bigl(\frac{1}{r} - \frac{1}{2} + H - \beta\bigr)}{\frac{1}{r} - \beta}} \, .
\end{equation}

For $\beta=0$ and $r\in\{1,2\}$, 
we can compare 
$a_{H,0,r}$ with the 
factors in   
$L^1$- and $L^2$-error bounds 
of the form 
$e^{-a_{H,0,r} \cdot \sqrt{N}}$ 
for the Gaussian quadrature schemes 
analyzed by Bayer and Breneis 
\cite{bayerbreneis2023, bayer2023weakma_arxiv}. 
This comparison is summarized 
in Table~{\upshape\ref{tbl:theoretical_coef}}. 
Figure~{\upshape\ref{fig:conv_rates_theory_L1L2}} 
plots these factors  
against the Hurst parameter    
${H\in\bigl(\frac{1}{2}-\tfrac{1}{r}, \tfrac{1}{2} \bigr)}$.
For $r=1$,
the curve of the 
sinc approximation intersects
with the geometric Gaussian (GG) and 
non-geometric Gaussian (NGG)
approximations
of \cite{bayer2023weakma_arxiv} at 
$H \approx 0.18$ and $H \approx -0.07$. 
For $r=2$,
the intersections 
with~\cite[Theorem~2.1 and 
\emph{optimal $m, \xi_0, \xi_n$}]{bayerbreneis2023}  
are at ${H \approx 0.46}$ 
and $H \approx 0.37$. 
For all Hurst parameters 
less or equal than these 
thresholds, the factor \eqref{eq:sinc_coef} 
outperforms those of 
the Gaussian quadrature schemes, 
and the sinc approximation
is particularly advantageous 
in (super) rough regimes.

\begin{table}[b!]
	\centering
	{\small
	\begin{tabular}{ccc}
		\toprule
		Method                                             & $a_{H,0,1}$                                                                         & $a_{H,0,2}$                                                                     \\
		\midrule
		Sinc; \eqref{eq:sinc_coef} with $\beta=0$                & $\sqrt{(\pi^2 - \delta) \bigl(\frac{1}{2} + H\bigr) \bigl(\frac{1}{2} - H\bigr)}$ & $\sqrt{2 (\pi^2 - \delta) H \bigl(\frac{1}{2} - H\bigr)}$                      \\
		\cite[Theorem~2.1]{bayerbreneis2023}               & ---                                                                                 & $1.06418 \cdot \bigl(\frac{1}{H} + \frac{1}{3/2 - H} \bigr)^{-\nicefrac{1}{2}}$ \\
		\cite[Optimal $m, \xi_0, \xi_n$]{bayerbreneis2023} & ---                                                                                 & $1.8 \cdot \bigl(\frac{1}{H} + \frac{1}{3/2 - H} \bigr)^{-\nicefrac{1}{2}}$     \\
		\cite[GG, Theorem~3.9]{bayer2023weakma_arxiv}      & $2 \log\bigl(\sqrt{2} + 1\bigr) \cdot \sqrt{\frac{1}{2} + H}$                       & ---                                                                             \\
		\cite[NGG, Theorem~3.16]{bayer2023weakma_arxiv}    & $2.3854 \cdot \sqrt{\frac{1}{2} + H}$                                               & ---                                                                             \\
		\bottomrule
	\end{tabular}}
	\caption{Comparison of the theoretical 
		convergence rate factors.} 
	\label{tbl:theoretical_coef}
\end{table}

\begin{figure}[t!]
	\centering
	\includegraphics[width=0.49\linewidth]{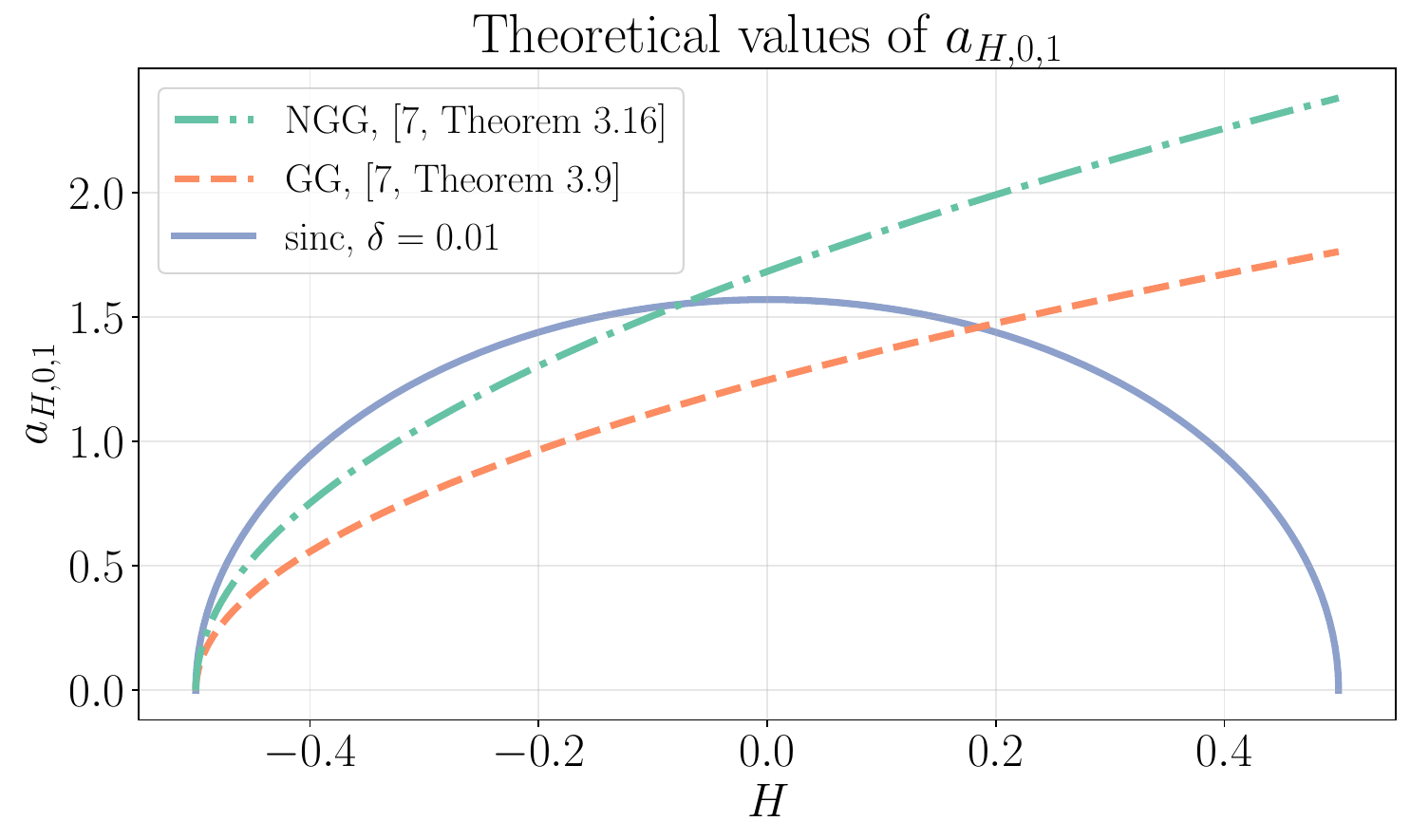}
	\includegraphics[width=0.49\linewidth]{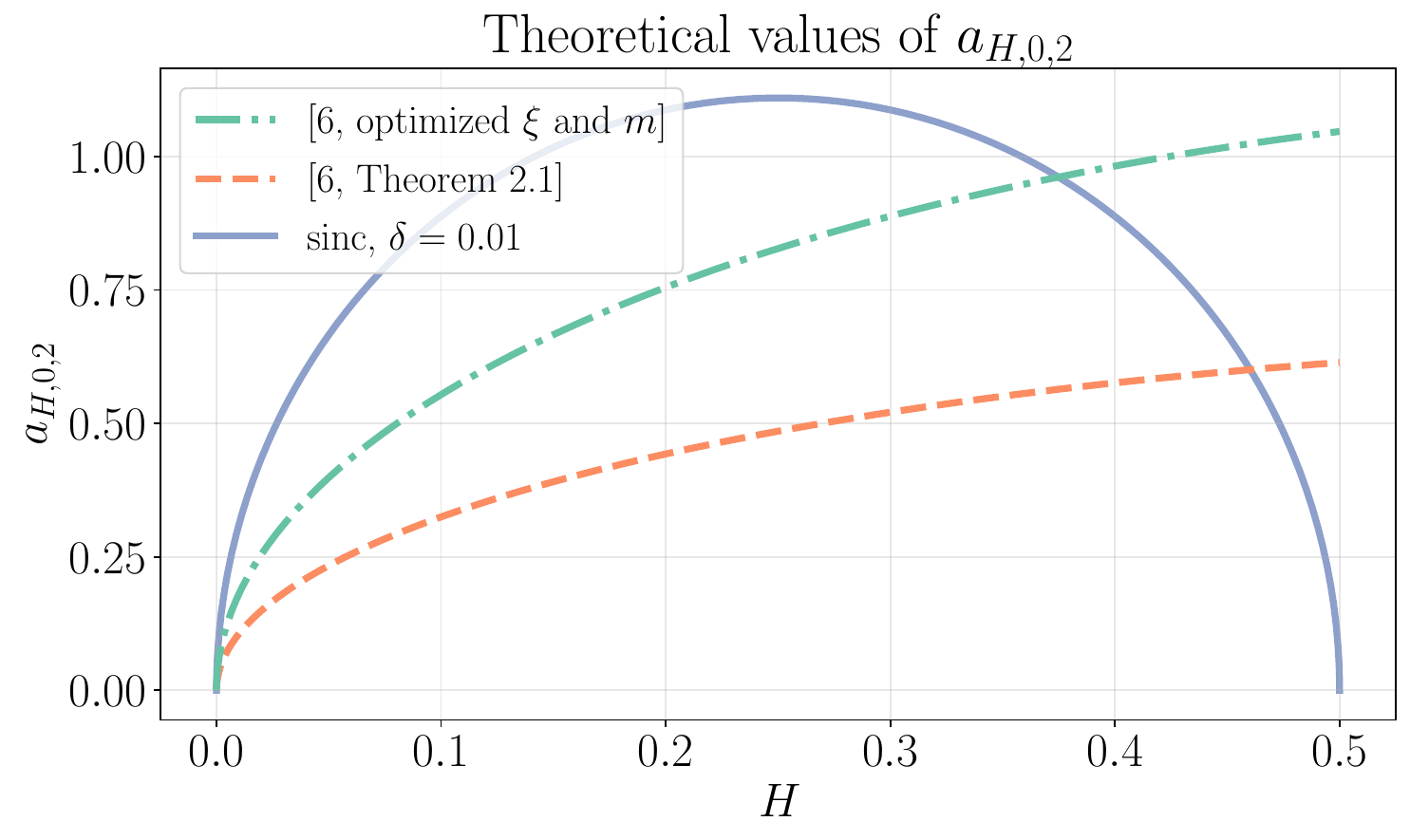}
	\caption{The factors $a_{H,0,r}$
		of Table~{\upshape\ref{tbl:theoretical_coef}} 
		as a function of $H$.}
	\label{fig:conv_rates_theory_L1L2}
\end{figure}

The error bound \eqref{eq:num_exp_sinc_conv} 
implies that, 
for sufficiently large $n$,
the log-squared error 
behaves like  
$\log^2\bigl(\| k-\widehat{k}^{(h_n, \beta, r)} \|_{L^r(0,1)}\bigr) \! 
\sim a_{H,\beta,r}^2 \, N_n$. 
To compute 
the $L^r(0,1)$-norm of the difference   
$\Delta := k-\widehat{k}^{(h_n, \beta, r)\!}$,  
we use a right rectangle rule
on a uniform mesh on $[0,1]$
with step size $\tau$ (i.e., $\tau^{-1\!}\in\bbN$), 
which avoids the singularity
at the origin,  
$\| \Delta \|_{L^r(0,1)} 
\approx I_{\tau}^{r}(\Delta) 
:= 
\bigl( \tau \sum_{j=1}^{\tau^{-1}} |\Delta(j\tau)|^r \bigr)^{\nicefrac{1}{r}}\!$. 
We note that 
${I_\tau^r(\Delta) = \| \Delta_\tau \|_{L^r(\tau,1+\tau)}}$, 
where 
$\Delta_\tau \colon [\tau, 1+\tau) \to \bbR$ 
is a step function 
corresponding to
a left point approximation
of $\Delta$ on $[\tau, 1+\tau)$, i.e.,  
$\Delta_\tau := \sum_{j=1}^{\tau^{-1}} \Delta(j\tau) \ind{[ j \tau, (j+1) \tau)}$.  
Then, by the (reverse) triangle inequality, 
we obtain that 
\begin{equation*}
	\bigl| \|\Delta\|_{L^r(0,1)} - I_\tau^r(\Delta) \bigr| 
	\leq 
	\| \Delta \|_{L^r(0,\tau)} 
	+
	\| \Delta \|_{L^r(1,1+\tau)} 
	+
	\| \Delta - \Delta_\tau \|_{L^r(\tau,1+\tau)}.
\end{equation*}
By the fundamental theorem of calculus
and H\"older's inequality, 
it follows that 
\begin{align*}
	\| \Delta - \Delta_\tau \|_{L^r(\tau,1+\tau)}^r
	&\leq 
	\sum_{j=1}^{\tau^{-1}} \int_{j\tau}^{(j+1)\tau} \left( \int_{j\tau}^u |\Delta'(v)| \diff{v} \right)^r \diff{u} \\
	&\leq 
	\sum_{j=1}^{\tau^{-1}} \int_{j\tau}^{(j+1)\tau} \tau^{r-1} \| \Delta' \|_{L^r(j\tau, u)}^r \diff{u}
	\leq 
	\tau^{r} \| \Delta' \|_{L^r(\tau,1+\tau)}^r.
\end{align*}
Theorem~{\upshape\ref{thm:error_bound_Lr_general_gamma}\ref{item:thm:error_bound_Lr_general_gamma-ii}} 
and~{\upshape\ref{item:thm:error_bound_Lr_general_gamma-iii}} 
then imply that 
$\bigl| \|\Delta\|_{L^r(0,1)} - I_\tau^r(\Delta) \bigr|
\lesssim \tau^\beta e^{-a_{H,\beta,r} \sqrt{N_n - 3}}$, 
and we expect
$\log^2(I_\tau^r(\Delta))$, 
plotted against~$N_n$, 
to approach a line 
with slope $a_{H,\beta,r}^2$. 

Figure~{\upshape\ref{fig:C0Lr_error_deterministic}} 
and Table~{\upshape\ref{tbl:C0Lr_error_deterministic}}  
confirm this for $\beta = 0$, 
where we used regression on  
the last quarter of data points
to compute the slope
in the asymptotic regime. 
In addition, as predicted by \eqref{eq:sinc_coef}, 
we observe 
a symmetric behavior of $a_{H,0,r}$
around $H = \tfrac{1}{2}-\tfrac{1}{2r}$.

\begin{figure}[b!]
	\centering
	\includegraphics[width=0.48\linewidth]{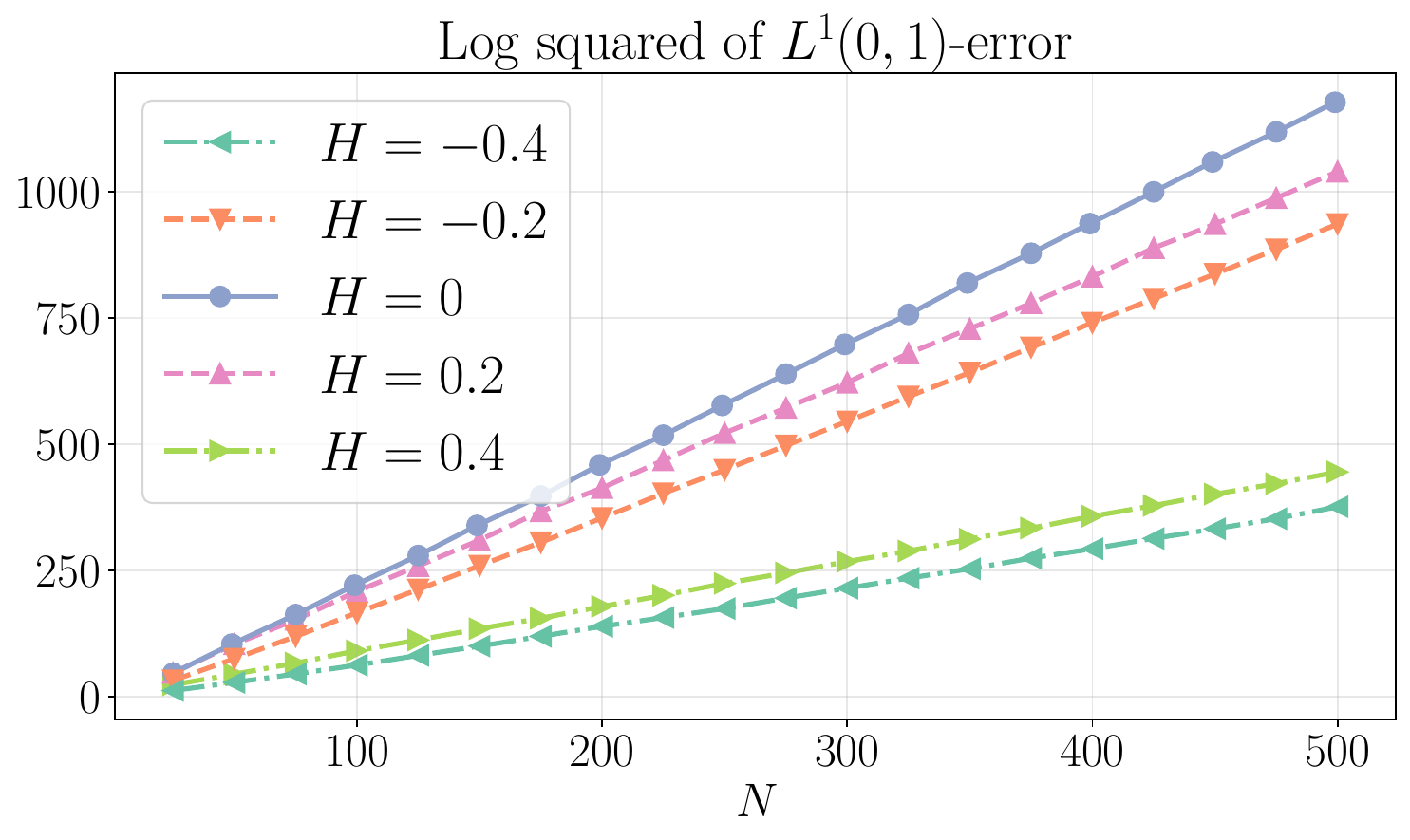}
	\includegraphics[width=0.48\linewidth]{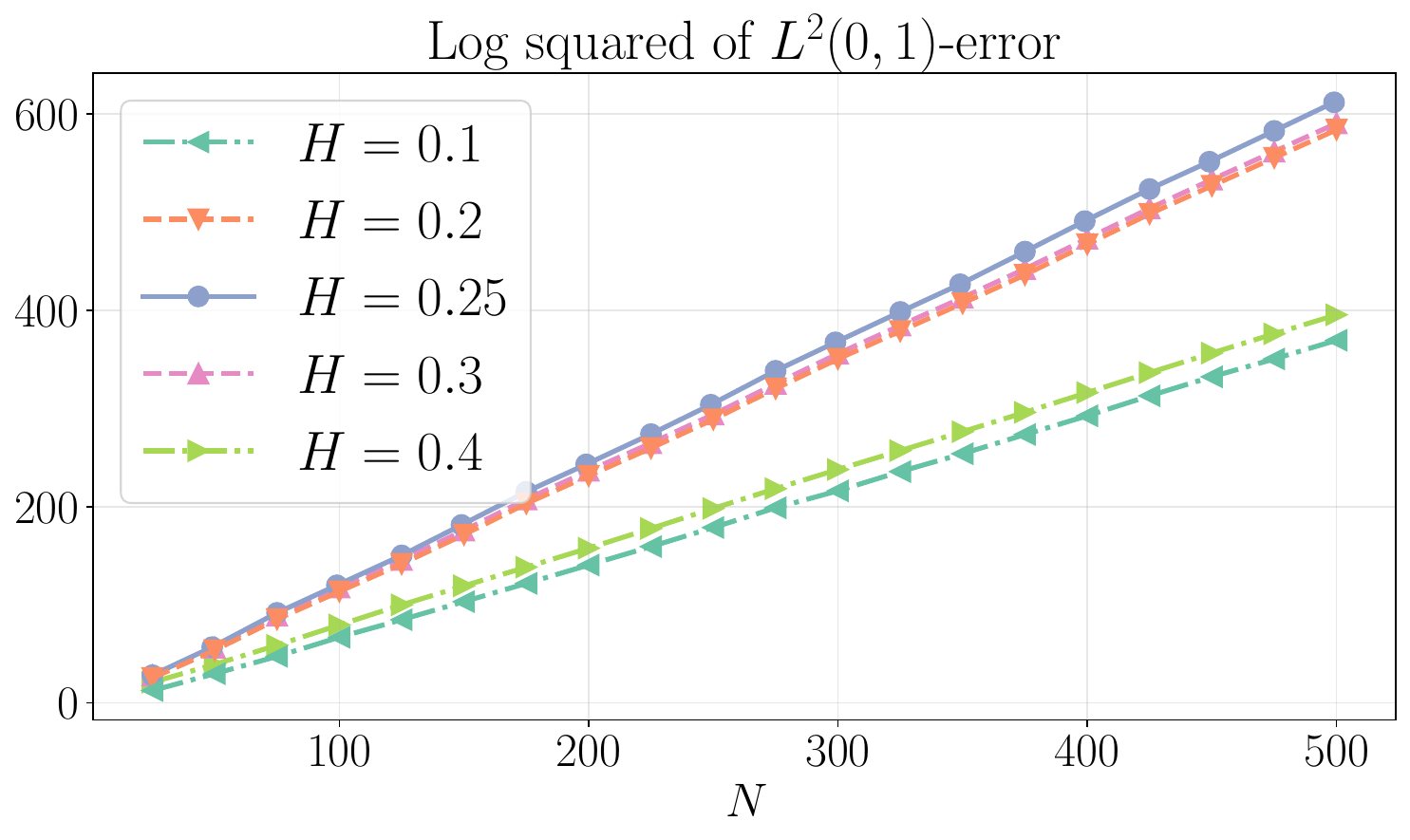}
	\vspace*{-1mm} 
	
	\caption{$\log^2 \bigl( I^r_\tau( k - \widehat{k}^{(h_n, \beta, r)} ) \bigr)$ 
		for $r=1$ (left) and $r=2$ (right) 
		for various $H$, 
		as a function of $N_n$, 
		computed with  
		${\tau = 2^{-14}}$.}
	\label{fig:C0Lr_error_deterministic} 
\end{figure}

\begin{table}[t!]
	\centering
	\small
	\begin{tabular}{cccccc}
		\toprule
		& Theor.\ slope              & Emp.\ slope &  & Theor.\ slope & Emp.\ slope \\
		\cmidrule(lr){2-3} \cmidrule(l){5-6} 
		$H$ & \multicolumn{2}{c}{$r=1$} & $H$ & \multicolumn{2}{c}{$r=2$}  \\ 
		\cmidrule(r){1-3} \cmidrule(l){4-6}
		-0.4 & 0.887 & 0.822  & 0.1 & 0.789 & 0.766  \\
		-0.2 & 2.071 & 1.954  & 0.2 & 1.183 & 1.161  \\
		0    & 2.465 & 2.397  & 0.25 & 1.232 & 1.205 \\
		0.2  & 2.071 & 2.060  & 0.3 & 1.183 & 1.171  \\
		0.4  & 0.887 & 0.875  & 0.4 & 0.789 & 0.792  \\
		\bottomrule
	\end{tabular}
	\caption{Theoretical and empirical slopes
		for the lines in Figure~{\upshape\ref{fig:C0Lr_error_deterministic}}.}
	\label{tbl:C0Lr_error_deterministic}
\end{table} 

Figure~{\upshape\ref{fig:holder_plots_deterministic}} 
and Table~{\upshape\ref{tbl:holder_slopes}} 
verify the H\"older regularity of
the mapping   
${t\mapsto \|\Delta\|_{L^r(0,t)}}$ 
on the interval $[0,1]$,  
for $r\in\{1,2\}$, where 
for several values of the 
Hurst parameter $H\in \bigl(\frac{1}{2} - \frac{1}{r}, \frac{1}{2}\bigr)$, 
we vary 
the H\"older exponent 
$\beta\in \bigl[0, H - \frac{1}{2} + \frac{1}{r} \bigr)$. 
For this purpose, 
we investigate the quantity 
$\sup_{0 \leq s < t \leq 1} (t-s)^{-\beta} \| \Delta \|_{L^r(s,t)}$, 
which is an upper bound 
for the H\"older seminorm 
of the mapping 
${[0,1] \ni t\mapsto \|\Delta\|_{L^r(0,t)}}$.

\begin{figure}[h!]
	\centering
	\includegraphics[width=0.32\linewidth]{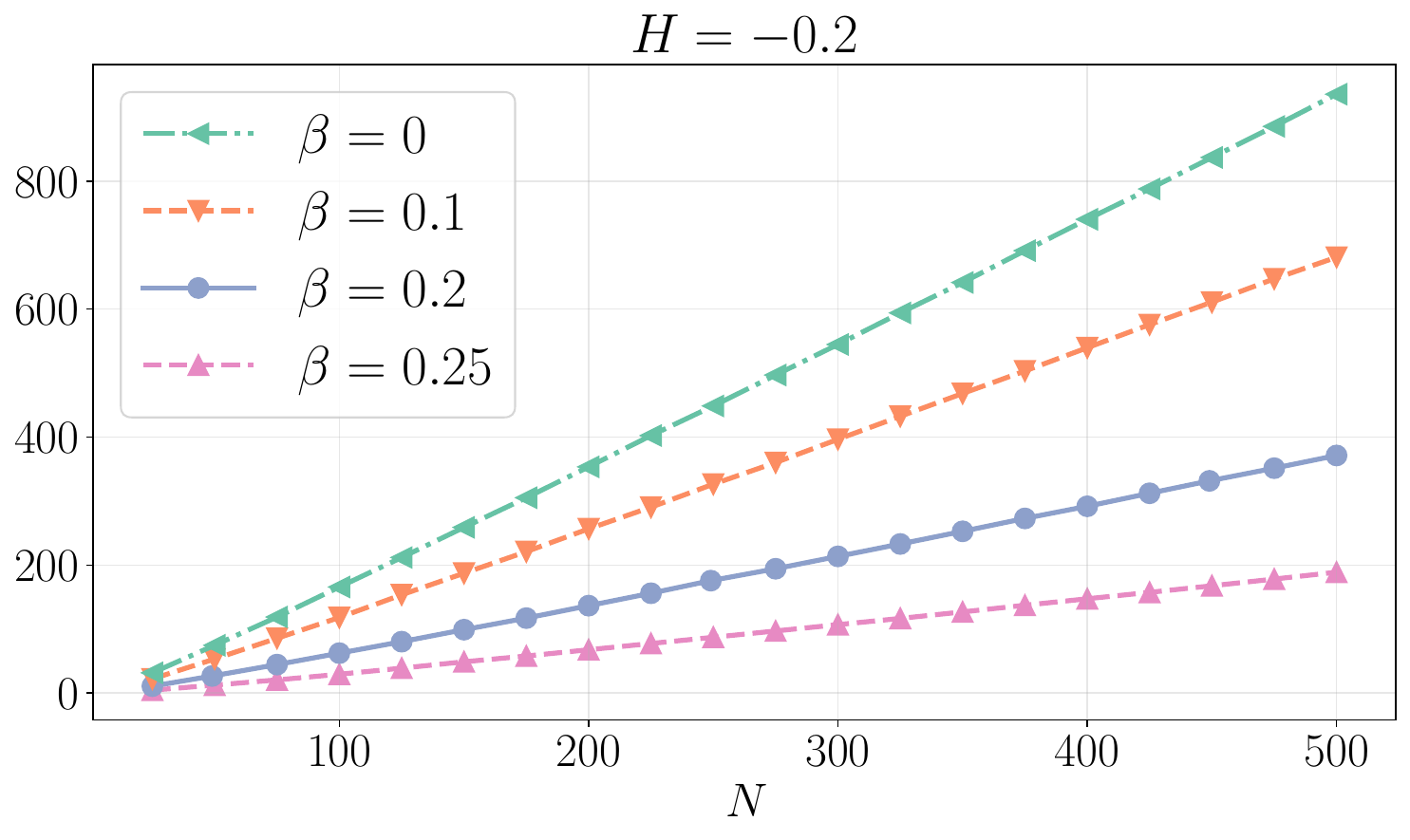}
	\includegraphics[width=0.32\linewidth]{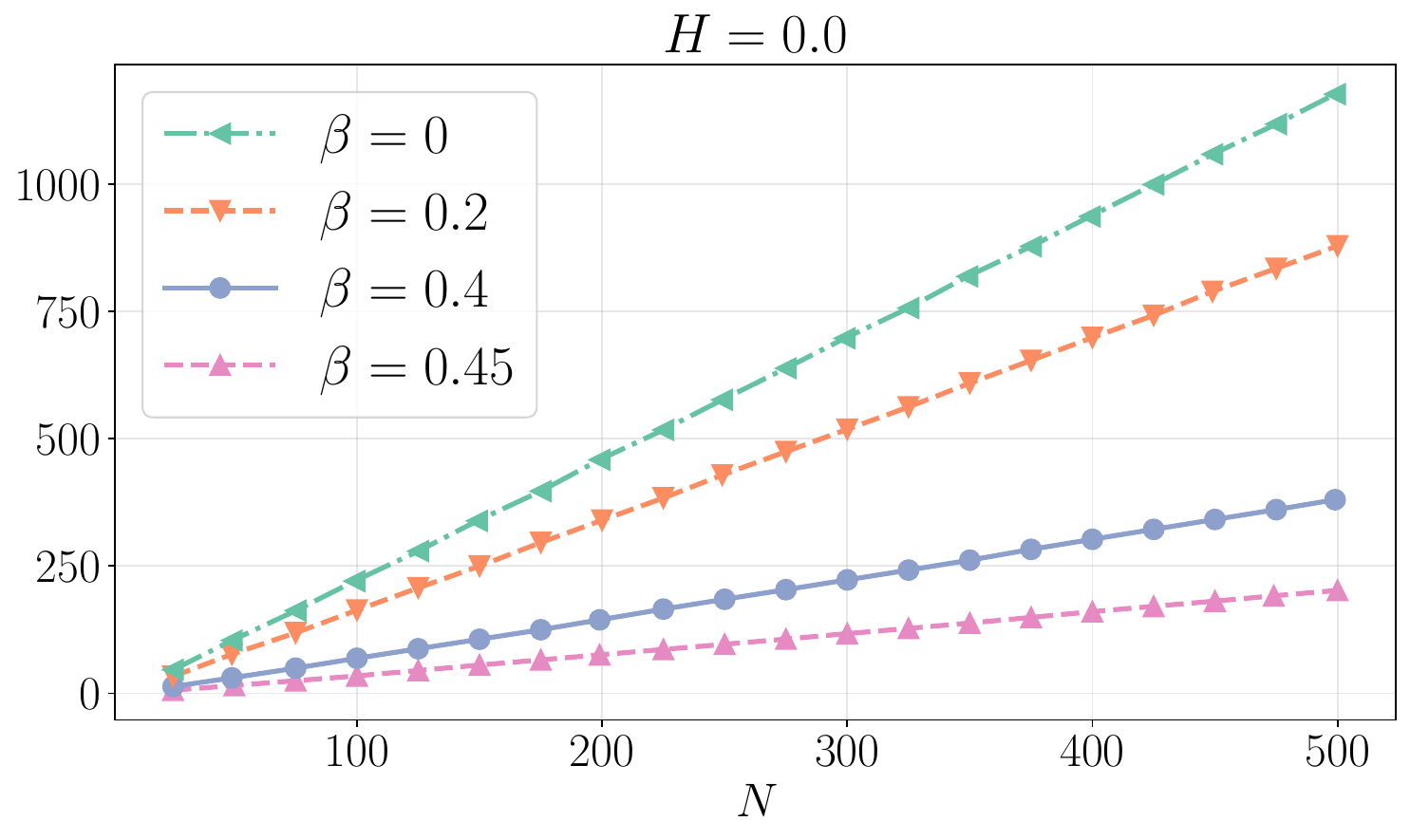}
	\includegraphics[width=0.32\linewidth]{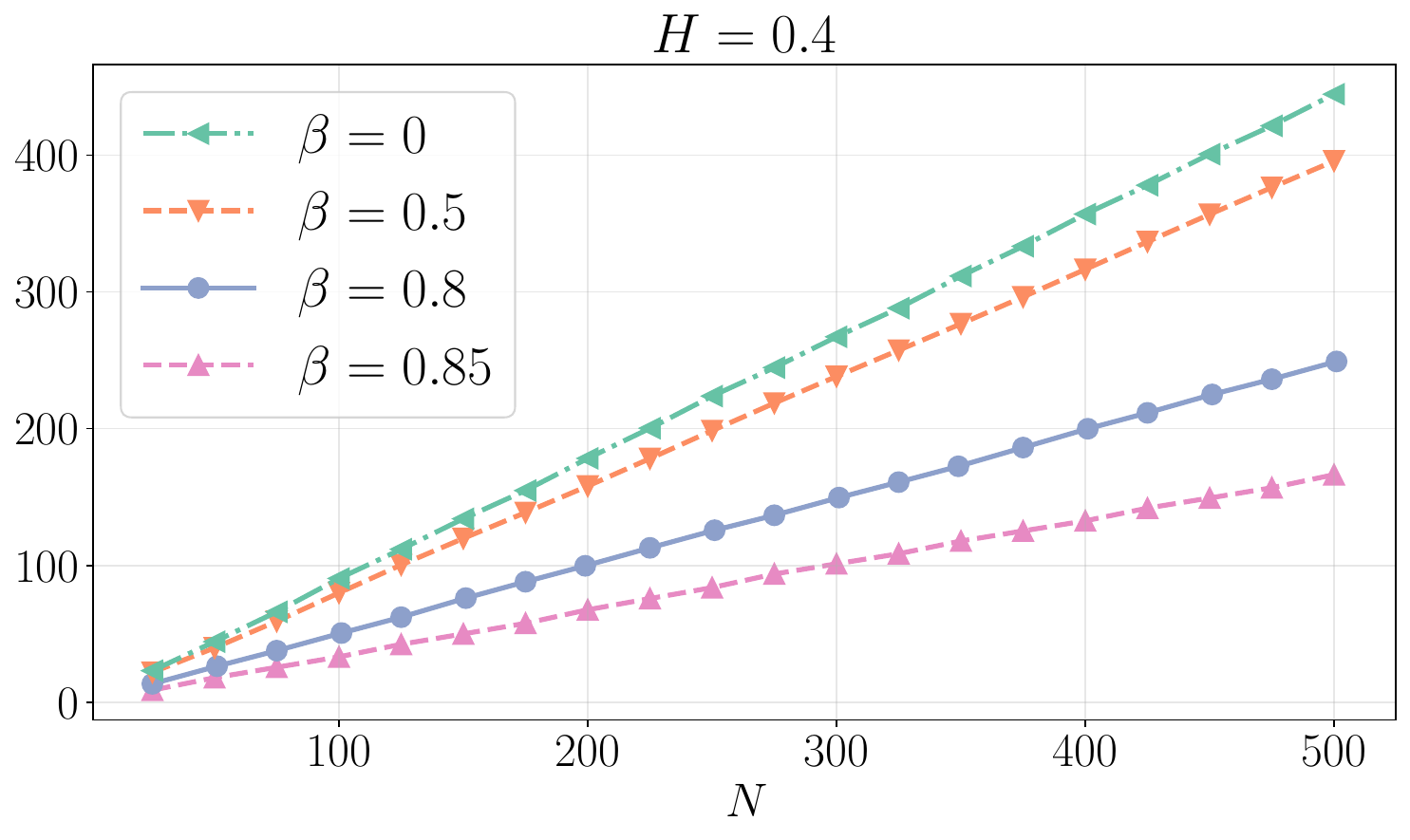}
	\includegraphics[width=0.32\linewidth]{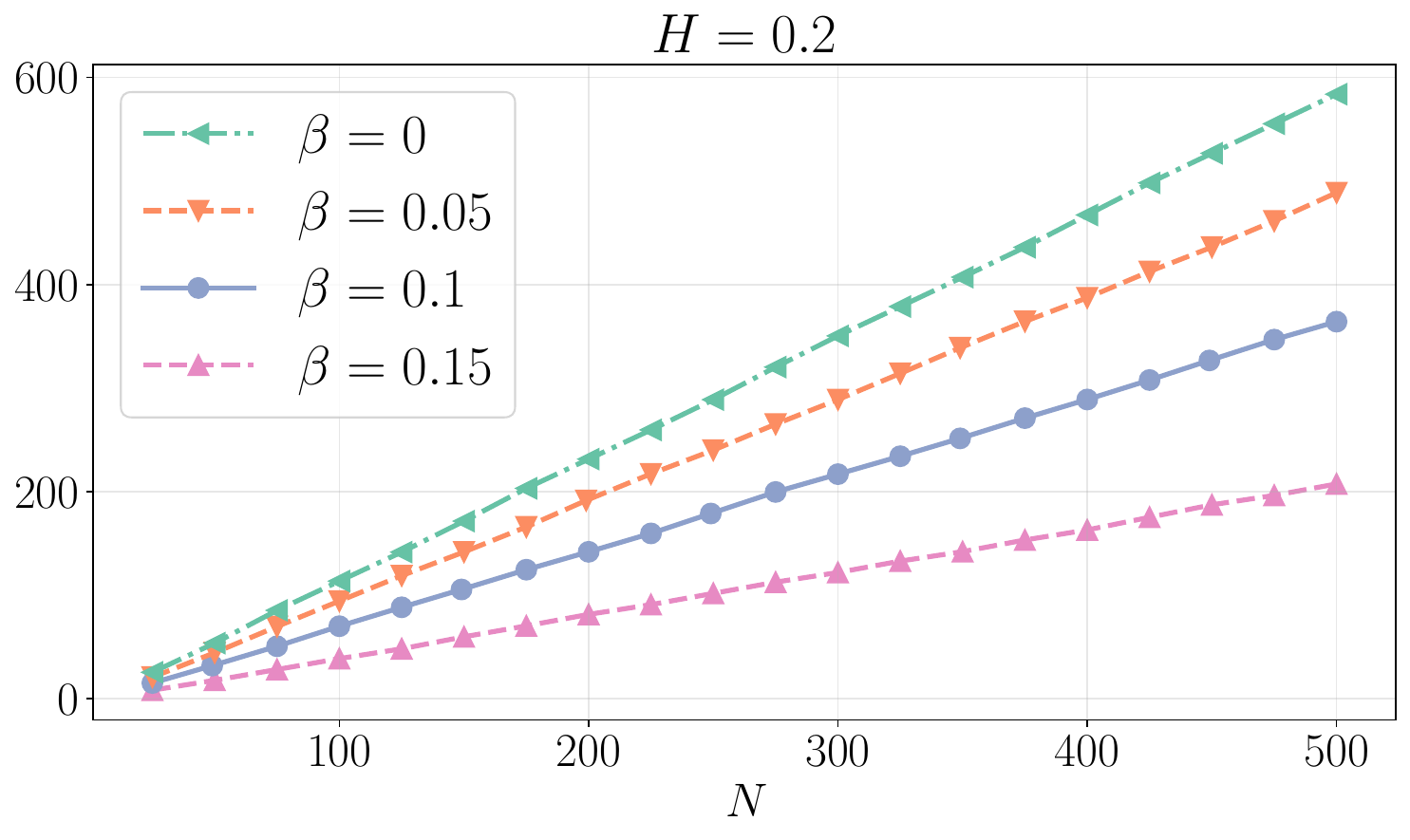}
	\includegraphics[width=0.32\linewidth]{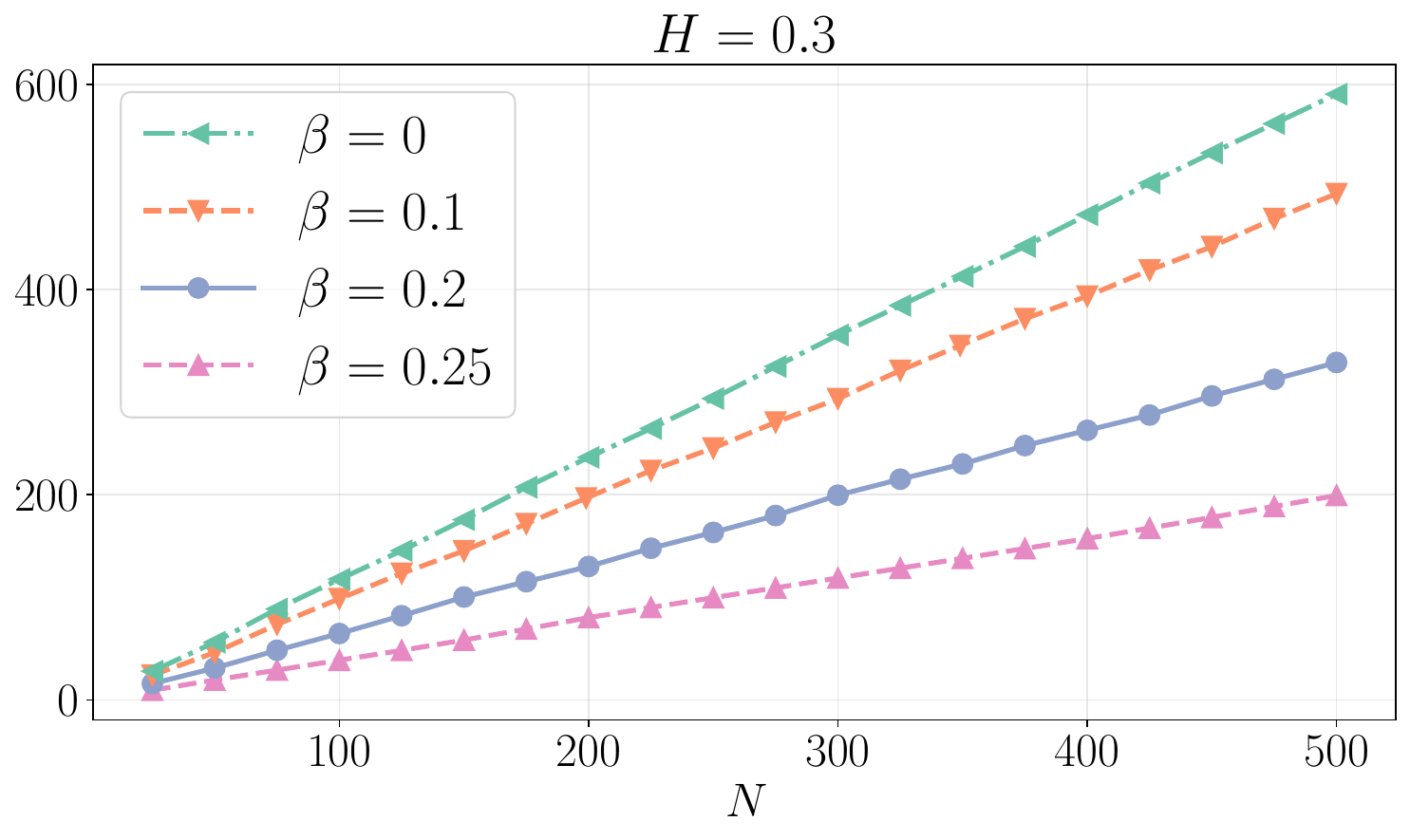}
	\includegraphics[width=0.32\linewidth]{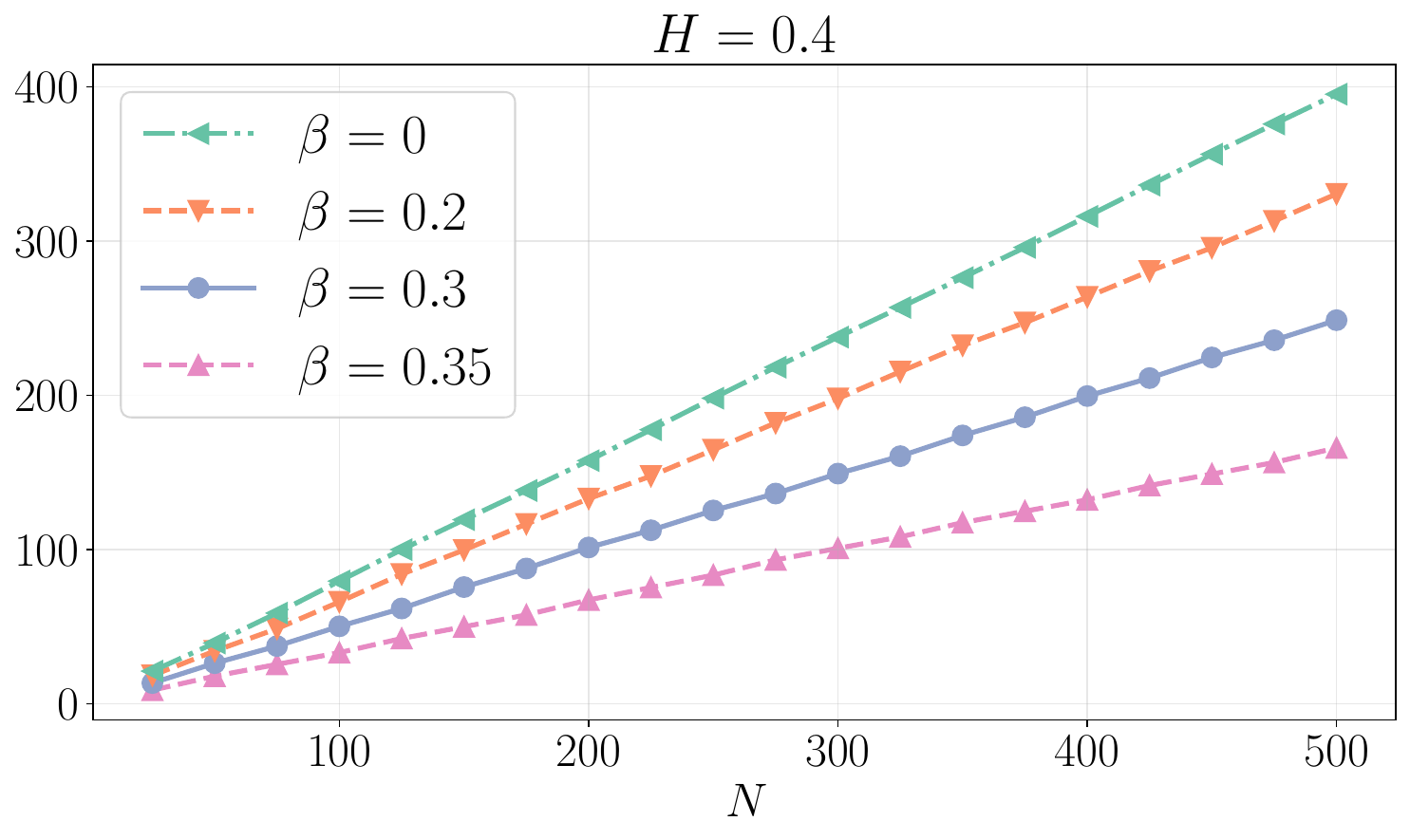}
	\caption{$\log^2 \bigl(E_{\tau}^{(h_n,\beta,r)}\bigr)$ 
		for $r=1$ (top) and $r=2$ (bottom) 
		for various $H$ and $\beta$, 
		as a function of $N_n$, 
		computed with  
		${\tau = 2^{-14}}$.\vspace{-1.5\baselineskip}}  
	\label{fig:holder_plots_deterministic}
\end{figure}

\begin{table}[h!]
	\centering
	{\small
		\begin{tabular}{@{}ccccccccc@{}}
			\toprule
			& Th.\ slope & Emp.\ slope & & Th.\ slope & Emp.\ slope & &  Th.\ slope & Emp.\ slope \\
			\cmidrule(lr){2-3} \cmidrule(lr){5-6} \cmidrule(l){8-9}
			$\beta$ & \multicolumn{2}{c}{$r=1$, $H = -0.2$} & 
			$\beta$ & \multicolumn{2}{c}{$r=1$, $H = 0$} & 
			$\beta$ & \multicolumn{2}{c}{$r=1$, $H = 0.4$} \\ 
			\cmidrule(r){1-3} \cmidrule(lr){4-6} \cmidrule(l){7-9}
			0 & 2.071 & 1.954    & 0 & 2.465 & 2.397     & 0 & 0.887 & 0.875    \\
			0.1 & 1.534 & 1.417  & 0.2 & 1.849 & 1.810   & 0.5 & 0.789 & 0.793  \\
			0.2 & 0.863 & 0.792  & 0.4 & 0.822 & 0.783   & 0.8 & 0.493 & 0.492  \\
			0.25 & 0.460 & 0.415 & 0.45 & 0.448 & 0.418  & 0.85 & 0.329 & 0.330 \\
			
			\specialrule{.08em}{.4ex}{.65ex} 
			
			$\beta$ & \multicolumn{2}{c}{$r=2$, $H = 0.2$} & 
			$\beta$ & \multicolumn{2}{c}{$r=2$, $H = 0.3$} & 
			$\beta$ & \multicolumn{2}{c}{$r=2$, $H = 0.4$} \\ 
			\cmidrule(r){1-3} \cmidrule(lr){4-6} \cmidrule(l){7-9}
			0 & 1.183 & 1.161    & 0 & 1.183 & 1.171    & 0 & 0.789 & 0.792    \\
			0.05 & 0.986 & 1.007 & 0.1 & 0.986 & 0.998  & 0.2 & 0.657 & 0.665  \\
			0.1 & 0.739 & 0.757  & 0.2 & 0.657 & 0.668  & 0.3 & 0.493 & 0.492  \\
			0.15 & 0.423 & 0.440 & 0.25 & 0.394 & 0.420 & 0.35 & 0.329 & 0.329 \\
			\bottomrule
	\end{tabular}}
	\caption{Theoretical and empirical slopes
		for the lines in Figure~{\upshape\ref{fig:holder_plots_deterministic}}.}
	\label{tbl:holder_slopes}
\end{table}
   
We approximate 
$\sup_{0 \leq s < t \leq 1} (t-s)^{-\beta} \| \Delta \|_{L^r(s,t)}$ on 
a uniform grid of step size~$\tau$, 
where we again use the right rectangle rule  
for computing the $L^r(s,t)$-norms, i.e., 
\begin{equation*} 
	E^{(h_n,\beta,r)}_{\tau} 
	= 
	I^{r,\beta}_{\tau}\bigl(k-\widehat{k}^{(h_n, \beta, r)} \bigr) 
	:= 
	\max_{\substack{ i, j \in \{0,\ldots,\tau^{-1}\} \\ j < i}}
	(\tau(i - j))^{-\beta}
	\Biggl( \tau \sum_{\ell=j+1}^{i} |\Delta(\ell \tau)|^r \Biggr)^{\nicefrac{1}{r}}\!.
\end{equation*}


\subsection{Strong stochastic errors}
\label{subsec:num_exp_stoch_err}

For a given $H \in \bigl(0,\frac{1}{2}\bigr)$, 
and for $\beta \in [0, H)$, 
we consider $h_n := h_n(\beta, 2)$ 
and $\widehat{k}^{(h_n, \beta)} := \widehat{k}^{(h_n, \beta, 2)}$
as in \eqref{eq:num_def_h_n} 
and \eqref{eq:num_def_kernels} 
with $r=2$. 
We define $X$
and $\widehat{X}^{(h_n,\beta)}$
as continuous
stochastic Volterra processes with
initial value $X_0 \equiv 0$,
coefficients $b \equiv 0$ and $\sigma \equiv 1$,
and with convolution kernels
$k^H$ and
$\widehat{k}^{(h_n,\beta)\!}$, 
respectively. 
Then, $X$ is a Riemann--Liouville 
fractional Brownian motion 
with Hurst parameter $H$,  
and both $X$
and $\widehat{X}^{(h_n,\beta)}$
are Gaussian processes 
(due to having zero drift and additive noise). 
Therefore, 
we can simulate their paths exactly
by computing the lower Cholesky factors
of the respective covariance matrices. 

Let $s,t \in [0,\infty)$
and assume w.l.o.g.\ that $s \leq t$
(for $s > t$, we may use symmetry of the covariance).
By \cite[Equation~(11)]{limSithi1995}
the covariance of $X$ is given by 
\begin{equation*} 
	\operatorname{Cov}(X_t, X_s)
	= \frac{s^{H+\frac{1}{2}} t^{H-\frac{1}{2}}}{\Gamma(H+\nicefrac{1}{2})\Gamma(H+\nicefrac{3}{2})}
	\,{}_2F_1\bigl({\tfrac12 - H, 1, H+\tfrac{3}{2}, \tfrac{s}{t}}\bigr) . 
\end{equation*} 
Here, $_2F_1$ is the hypergeometric function 
which we evaluate in Python 
by means of the 
\texttt{scipy.special.hyp2f1} function.  
Furthermore, 
using It\^o's isometry,
we also can derive  
the covariance function of
$\widehat{X}^{(h_n,\beta)}$ 
explicitly,   
\begin{align}
	&\hspace*{-0.95mm}\operatorname{Cov}  
	\bigl( \widehat{X}^{(h_n,\beta)}_t, \widehat{X}^{(h_n,\beta)}_s \bigr)
	= 
	\int_0^s \widehat{k}^{(h_n,\beta)}(t-u) \, \widehat{k}^{(h_n,\beta)}(s-u) \diff{u} \notag
	\\
	&\hspace*{-0.95mm}= 
	\sum_{\ell_1=-M^{-\!}_n}^{M^{+\!}_n(\beta)} 
	\sum_{\ell_2=-M^{-\!}_n}^{M^{+\!}_n(\beta)}
	\frac{h_n^2 \, e^{\left(\frac{1}{2} - H \right) (\ell_1+\ell_2)h_n} }{ 
			\Gamma(H+\nicefrac{1}{2})^2 \, \Gamma(\nicefrac{1}{2}-H)^2}
	\int_0^s e^{u \left(e^{\ell_1 h_n} + e^{\ell_2 h_n} \right) - t e^{\ell_1 h_n} - s e^{\ell_2 h_n} } \diff{u} 
	\notag
	\\
	&\hspace*{-0.95mm}= 
	\sum_{\ell_1=-M^{-\!}_n}^{M^{+\!}_n(\beta)} 
	\sum_{\ell_2=-M^{-\!}_n}^{M^{+\!}_n(\beta)}
	\frac{h_n^2 \, e^{\left(\frac{1}{2} - H \right) (\ell_1+\ell_2)h_n} }{ 
		\Gamma(H+\nicefrac{1}{2})^2 \, \Gamma(\nicefrac{1}{2}-H)^2}
	\, e^{- (t-s) e^{\ell_1 h_n}} \, 
	\frac{1 - e^{-s \left( e^{\ell_1 h_n} + e^{\ell_2 h_n} \right)} }{ e^{\ell_1 h_n} + e^{\ell_2 h_n}},\hspace*{-0.95mm} 
	\label{eq:sinc_cov_optim}
\end{align}
where 
we set $M_n^- := M^-(h_n)$ 
and $M_n^+(\beta) := M^+(h_n, \beta, 2)$. 
Evaluating the fraction 
$R_{\ell_1 \ell_2}^{-1} \bigl(1-e^{-s R_{\ell_1 \ell_2}}\bigr)$,  
with $R_{\ell_1 \ell_2} := e^{\ell_1 h_n} + e^{\ell_2 h_n}$, 
is numerically unstable for $R_{\ell_1 \ell_2} \downarrow 0$.
To increase stability,
we express its numerator as $-\operatorname{expm1}(-sR_{\ell_1 \ell_2})$
using \texttt{NumPy}'s \texttt{expm1} function.
When $sR_{\ell_1 \ell_2} < 2^{-510}$,
the second-order term in
the numerator's Taylor expansion
is smaller than machine precision, 
and we therefore approximate the fraction
by its analytical limit,
$\lim_{R \downarrow 0}R^{-1}\bigl(1 - e^{-sR}\bigr) = s$.

For the simulation of the strong errors, 
we let $\timemesh = \{t_1, t_2, \ldots, t_{\tau^{-1} - 1}, 1\}$, 
where $t_j = j\tau$ 
and $\tau$ is the step size of a uniform mesh on $[0,1]$.  
We use $M \in \bbN$ samples 
$\boldsymbol{\zeta}^{(1)\!}, \ldots, \boldsymbol{\zeta}^{(M)\!} \in \bbR^{|\timemesh|}$ 
of a standard normally distributed vector $\boldsymbol{\zeta}$ 
to generate samples
of the difference process
$X - \widehat{X}^{(h_n,\beta)}$
on $\timemesh$ via
$\mathbf{D}^{(h_n,\beta,i)}
:= (\mathbf{L} - \widehat{\mathbf{L}}^{(h_n,\beta)}) \boldsymbol{\zeta}^{(i)\!}$,
where $\mathbf{L}$ and $\widehat{\mathbf{L}}^{(h_n,\beta)}$
are the lower Cholesky factors
of the respective covariance matrices
evaluated on~$\timemesh$.
To empirically investigate 
the strong errors in $C^\beta([0,1];L^2(\Omega))$ 
and $L^2(\Omega;C^\beta([0,1]))$ for $\beta\in[0,H)$, 
we then consider  
\begin{align*}
	E^{(h_n)}_{C^0(L^2)}
	& :=
	\max_{{j_1}} \biggl( \frac{1}{M} \sum_{i=1}^M \left| D^{(h_n,0,i)}_{j_1} \right|^2 \biggr)^{\nicefrac{1}{2}},
	\
	E^{(h_n)}_{L^2(C^0)}
	:=
	\biggl( \frac{1}{M} \sum_{i=1}^M
	\max_{{j_1}} \left| D^{(h_n,0,i)}_{j_1} \right|^2 \biggr)^{\nicefrac{1}{2}},
\end{align*} 
\begin{align*}  
	E^{(h_n,\beta)}_{C^\beta(L^2)}
	& :=
	\max_{{j_1} > {j_2}} \,
	(t_{j_1} - t_{j_2})^{-\beta}\biggl( \frac{1}{M} \sum_{i=1}^M \left| D^{(h_n,\beta,i)}_{j_1} - D^{(h_n,\beta,i)}_{j_2} \right|^2 \biggr)^{\nicefrac{1}{2}},
	\\ E^{(h_n,\beta)}_{L^2(C^\beta)}
	& := \biggl( \frac{1}{M} \sum_{i=1}^M
	\max_{{j_1} > {j_2}} \,
	(t_{j_1} - t_{j_2})^{-2\beta}\left| D^{(h_n,\beta,i)}_{j_1} - D^{(h_n,\beta,i)}_{j_2} \right|^2 \biggr)^{\nicefrac{1}{2}}.
\end{align*} 

Figures~{\upshape\ref{fig:stochastic_L2_C0_errors}} 
and~{\upshape\ref{fig:stochastic_L2_holder_errors}},
alongside Tables~{\upshape\ref{tbl:stochastic_L2_C0_slopes}} 
and~{\upshape\ref{tbl:stochastic_holder_slopes}},
show that the observed log-squared errors
lie on lines with the slopes $a_{H,\beta,2}^2$
which confirms the strong error bounds 
of Theorems~{\upshape\ref{thm:C0_error_fracker}} 
and~{\upshape\ref{thm:holder_bound_fracker}}. 
Here, we again used linear regression on  
the last quarters of data points
for computing the empirical slopes.

\begin{figure}[h!]
	\centering
	\includegraphics[width=0.49\linewidth]{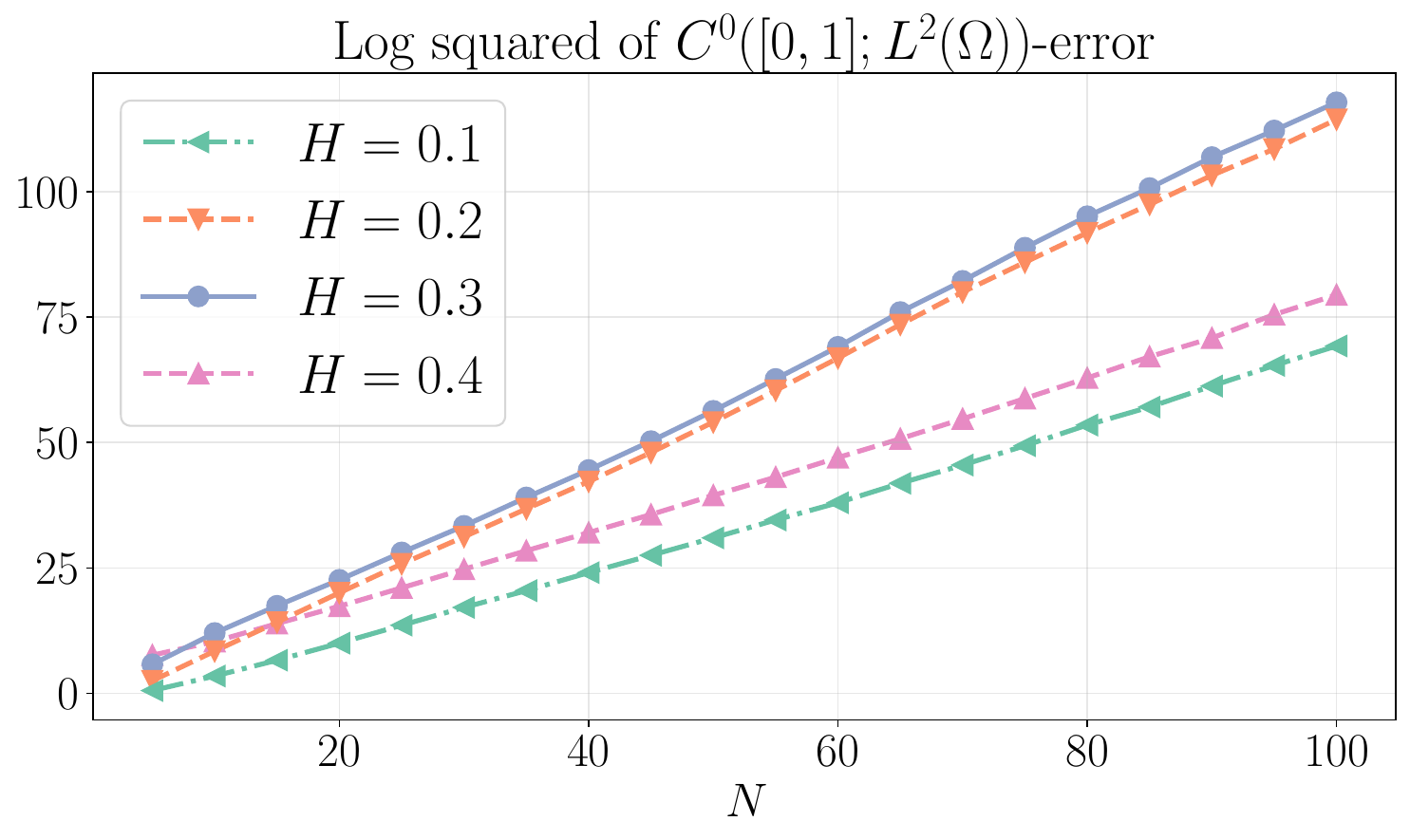}
	\includegraphics[width=0.49\linewidth]{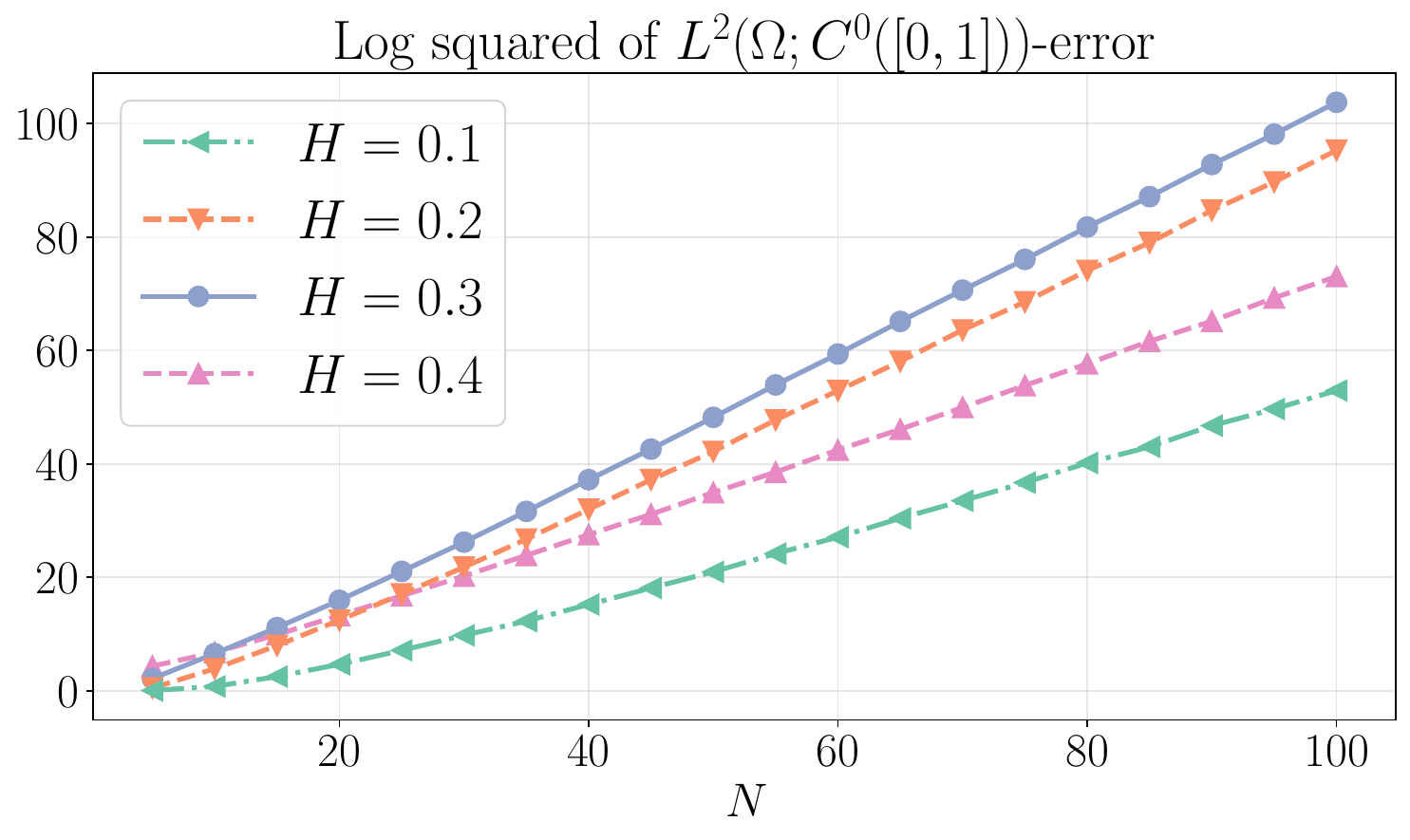}
	\caption{$\log^2\bigl(E^{(h_n)}_{C^0(L^2)}\bigr)$
		(left) and $\log^2\bigl(E^{(h_n)}_{L^2(C^0)}\bigr)$ (right) 
		for various $H$, 
		as a function of $N_n$, 
		computed with  
		$\tau = 2^{-10}$
		and $M = 2^{13}$.\vspace{-.75\baselineskip}} 
	\label{fig:stochastic_L2_C0_errors}
\end{figure}

\begin{table}[h!]
	\centering
	\small
	\begin{tabular}{@{}cccc@{}}
		\toprule
		& Theor.\ slope & Emp.\ slope  & Emp.\ slope  \\
		$H$ & $C^0([0,1];L^2(\Omega))$ & $C^0([0,1];L^2(\Omega))$ & $L^2(\Omega;C^0([0,1]))$ \\
		\midrule 
		 0.1 &  0.789 &  0.796 &  0.644 \\
		 0.2 &  1.183 &  1.125 &  1.062 \\
		 0.3 &  1.183 &  1.138 &  1.100 \\
		 0.4 &  0.789 &  0.832 &  0.769 \\
		\bottomrule
	\end{tabular}
	\caption{Theoretical and empirical slopes
		for the lines in Figure~{\upshape\ref{fig:stochastic_L2_C0_errors}}.\vspace{-.75\baselineskip}}  
	\label{tbl:stochastic_L2_C0_slopes}
\end{table}

\begin{figure}[h!]
	\centering
	\includegraphics[width=0.32\linewidth]{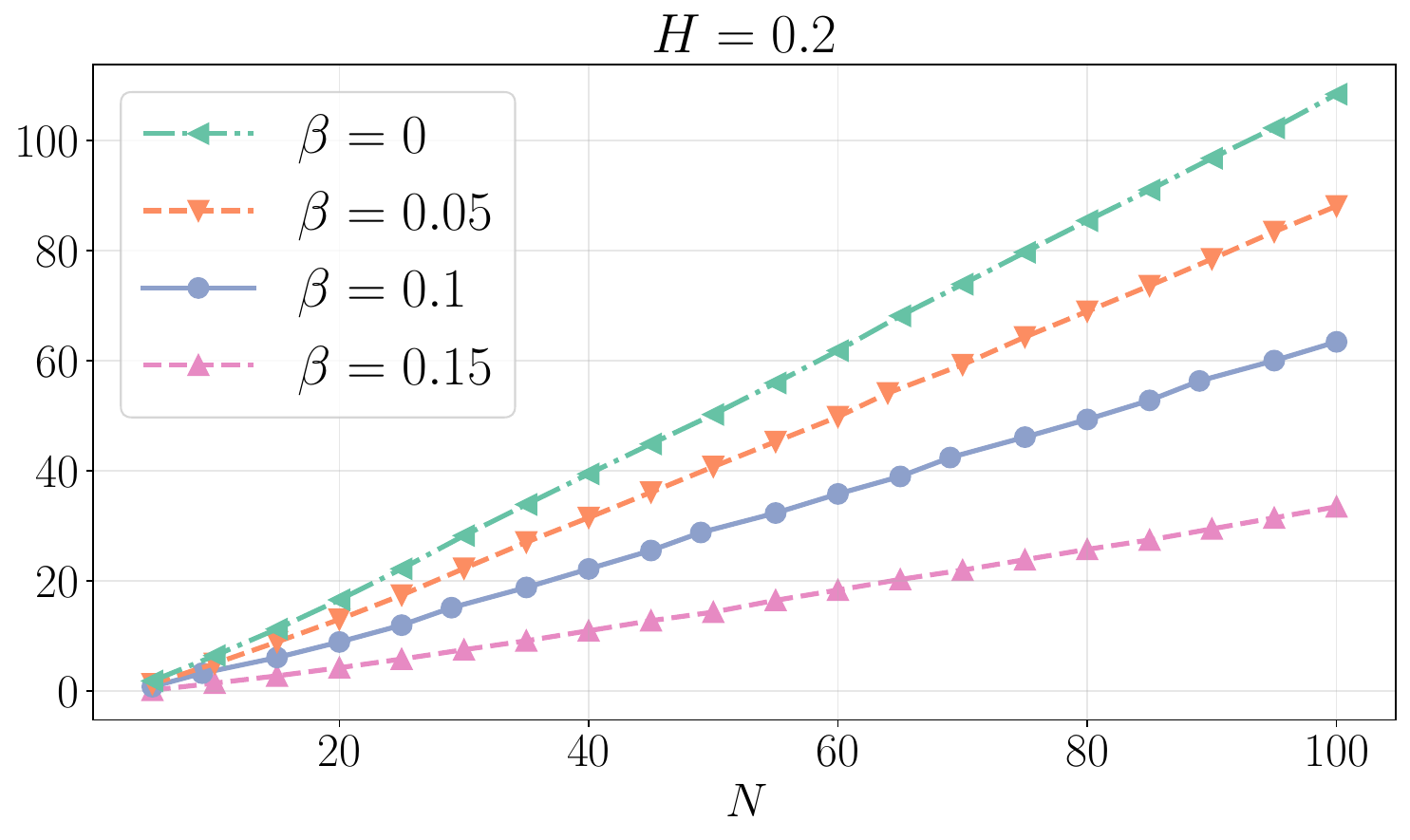}
	\includegraphics[width=0.32\linewidth]{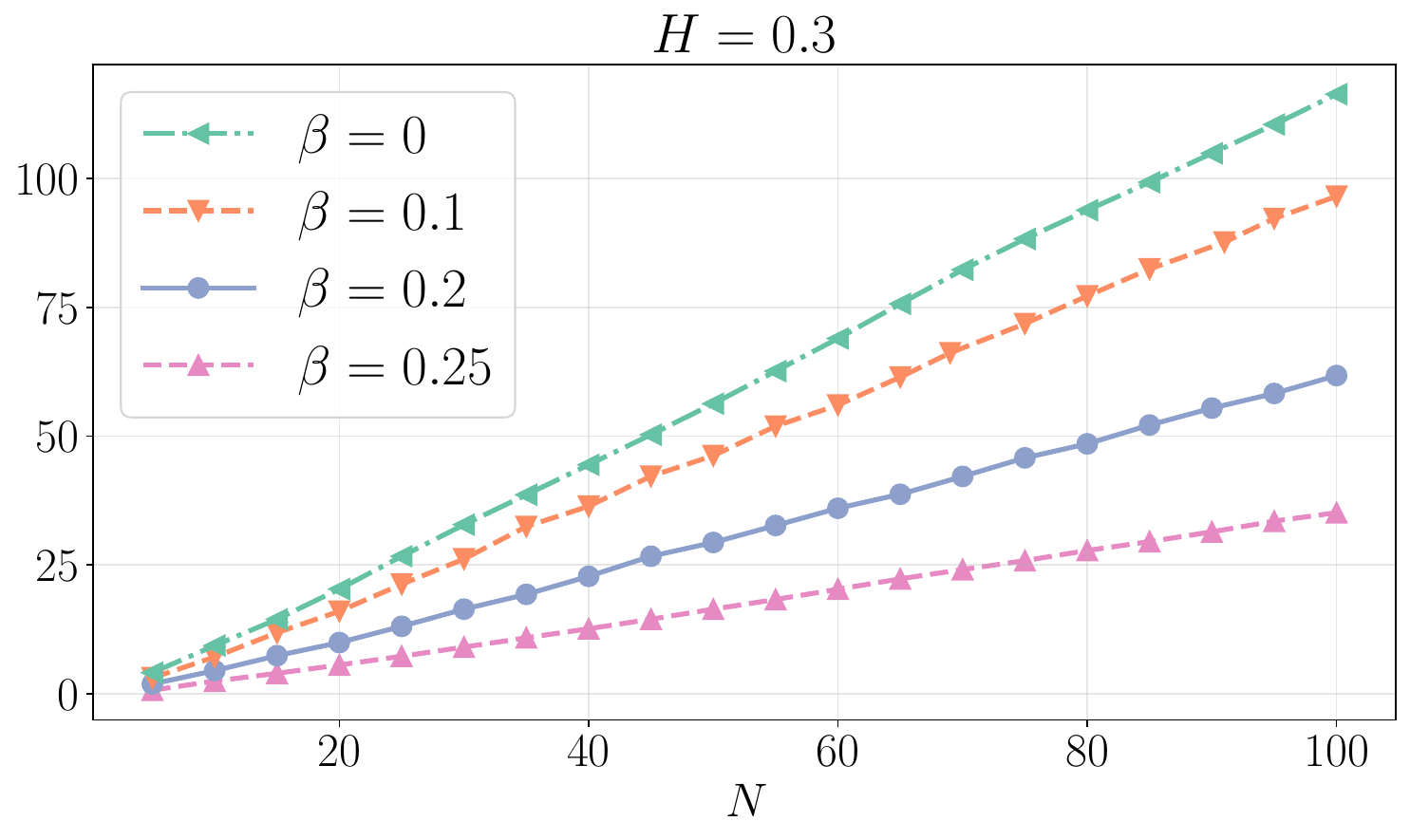}
	\includegraphics[width=0.32\linewidth]{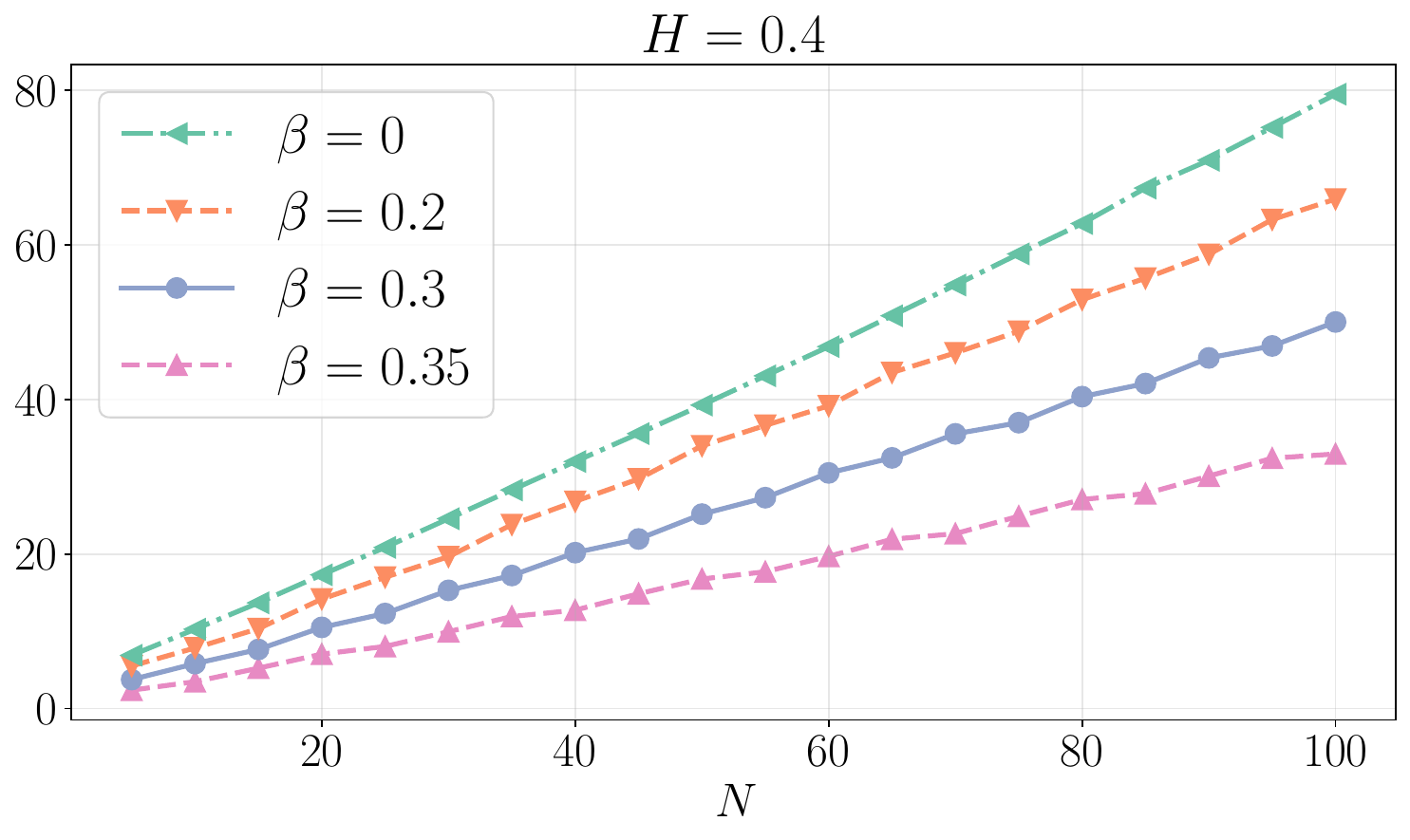}
	\includegraphics[width=0.32\linewidth]{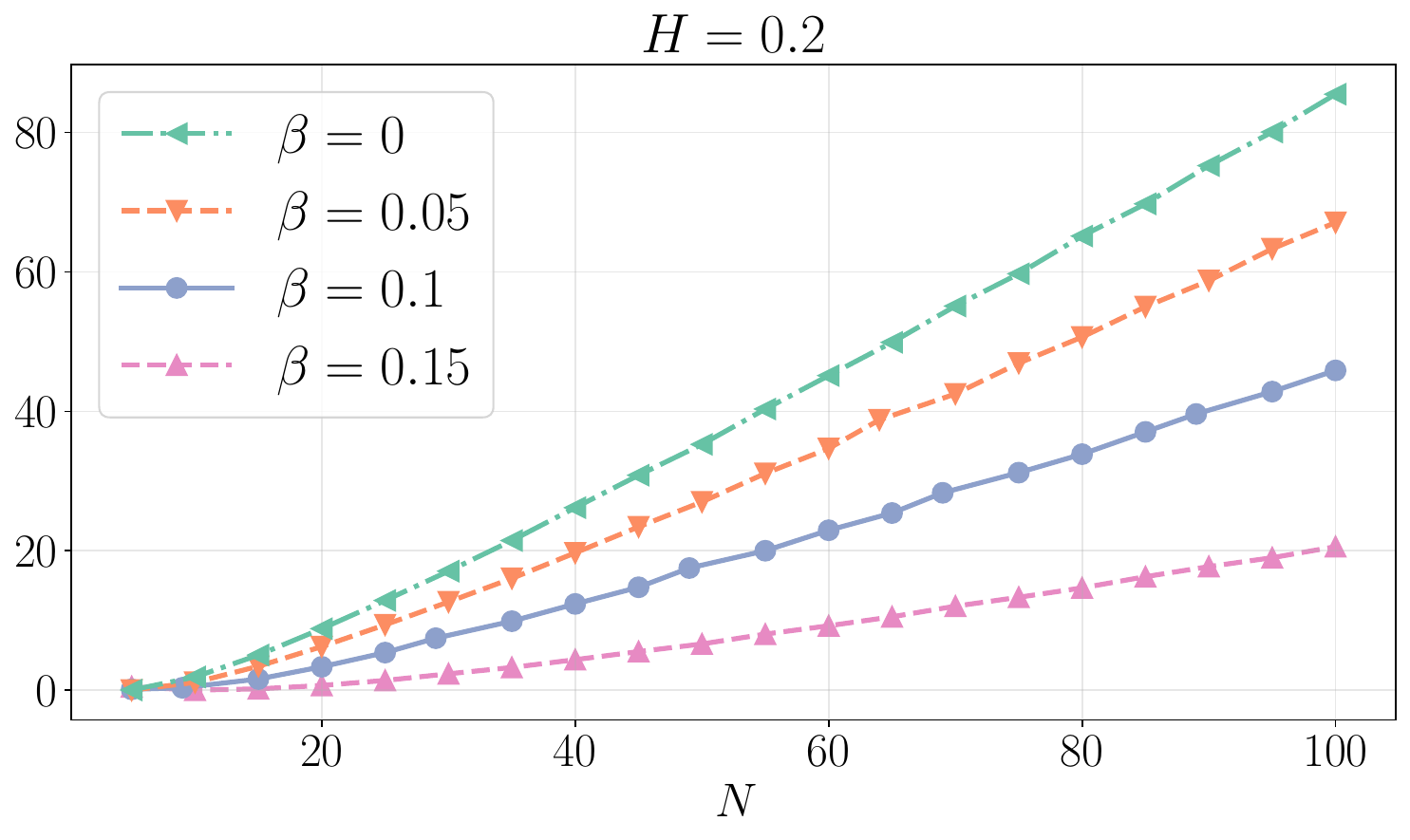}
	\includegraphics[width=0.32\linewidth]{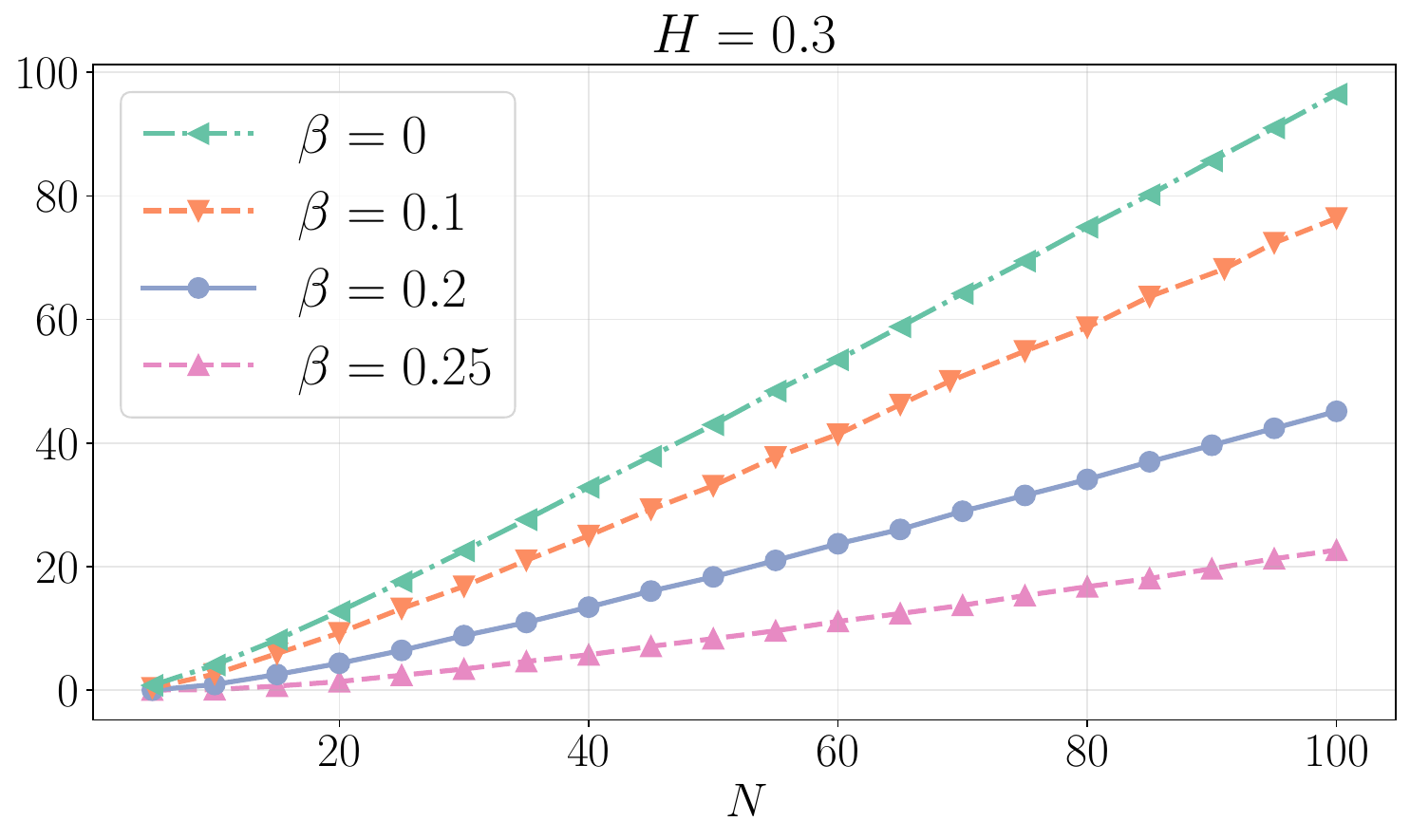}
	\includegraphics[width=0.32\linewidth]{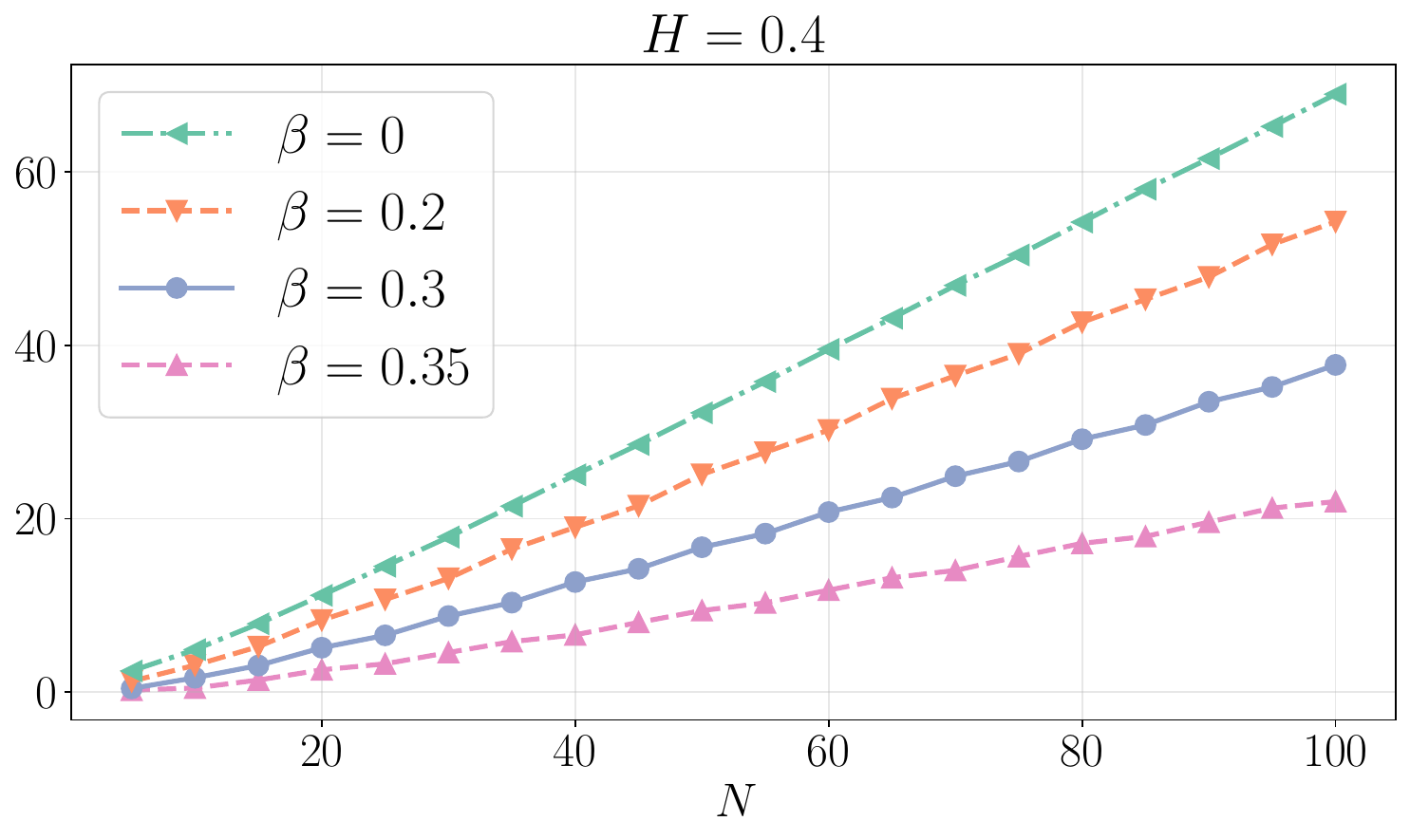}
	\caption{$\log^2\bigl(E^{(h_n,\beta)}_{C^\beta(L^2)}\bigr)$
		(top) and 
		$\log^2\bigl(E^{(h_n,\beta)}_{L^2(C^\beta)}\bigr)$
		(bottom) for various~$H$ and~$\beta$,  
		computed with   
		$\tau = 2^{-10}$
		and $M = 2^{13}\!$.}
	\label{fig:stochastic_L2_holder_errors}
\end{figure}
 
 \begin{table}[t!]
 	\centering 
 	\small
 	\begin{tabular}{cccc}
 		\toprule
 		& Theor.\ slope & Emp.\ slope  & Emp.\ slope  \\
 		& $C^\beta([0,1];L^2(\Omega))$ & $C^\beta([0,1];L^2(\Omega))$ & $L^2(\Omega;C^\beta([0,1]))$ \\
 		\cmidrule(lr){2-4} 
 		$\beta$ & \multicolumn{3}{c}{$H=0.2$} \\
 		\midrule   
 		0    & 1.183 & 1.145 & 1.020 \\
 		0.05 & 0.986 & 0.961 & 0.823 \\
 		0.1  & 0.739 & 0.707 & 0.595 \\
 		0.15 & 0.423 & 0.390 & 0.291 \\
 		\specialrule{.08em}{.4ex}{.65ex}
 		$\beta$ & \multicolumn{3}{c}{$H=0.3$} \\
 		\midrule 
 		0    & 1.183 & 1.125 & 1.077 \\
 		0.1  & 0.986 & 0.969 & 0.879 \\
 		0.2  & 0.657 & 0.653 & 0.549 \\
 		0.25 & 0.394 & 0.373 & 0.301 \\
 		\specialrule{.08em}{.4ex}{.65ex}
 		$\beta$ & \multicolumn{3}{c}{$H=0.4$} \\
 		\midrule 
 		0    & 0.789 & 0.826  & 0.735 \\
 		0.2  & 0.657 & 0.672  & 0.591 \\
 		0.3  & 0.493 & 0.483  & 0.431 \\
 		0.35 & 0.329 & 0.327  & 0.258 \\
 		\bottomrule
 	\end{tabular}
 	\caption{Theoretical and empirical slopes
 		for the lines in 
 		Figure~{\upshape\ref{fig:stochastic_L2_holder_errors}}.}
 	\label{tbl:stochastic_holder_slopes}
 \end{table}


\appendix


\section{Volterra kernels and a Gr\"onwall type inequality}
\label{appendix:Gronwall}

This appendix serves as a means, firstly, 
to introduce Volterra kernels and   
some of their properties which are 
relevant for  
the results in Section~{\upshape\ref{sec:general_SVE}} 
(see Appendix~{\upshape\ref{subappendix:Volterra}})
and, secondly, 
by means of resolvent equations 
to establish a generalized 
Gr\"onwall type inequality 
for non-convolution kernels, 
see Appendix~{\upshape\ref{subappendix:gronwall}}.  


\subsection{Auxiliary results on Volterra kernels}\label{subappendix:Volterra}

In this section we follow 
\cite[Chapter~9]{gripenberg1990} 
to introduce the concept of Volterra kernels, 
and we discuss some of their important features. 
Recall from Subsection~{\upshape\ref{subsec:notation}} 
that $\triangle J := \{(t,s) \in J\times J : s \leq t\}$.

\begin{definition}\label{def:volterra_kernel}
	Assume that $J\subset\bbR$ is an interval.
	A Borel measurable function
	${\kappa \colon J\times J \to \bbR}$
	is called a Volterra kernel
	if $\kappa(t,s) = 0$ a.e.\ outside of 
	$\triangle J$.
	The space of Volterra kernels on $J$ 
	is denoted by $\cV(J)$.
	
	A Volterra kernel is of type $L^\infty$ on $J$ if 
	\begin{equation*}
		\| \kappa \|_{\cV(L^\infty;J)} 
		:= 
		\esssup_{t \in J} \int_J | \kappa(t,s) | \diff{s} < \infty.
	\end{equation*}
	The space of all such kernels
	(modulo equality a.e.)
	is denoted by $\cV(L^\infty;J)$.
\end{definition}

Equipped with the product operation 
${\star\colon\cV(L^\infty;J) \times \cV(L^\infty;J) \to \cV(J)}$,  
\begin{equation}\label{eq:def_star} 
	(a\star b)(t,s) = \int_J a(t,u) b(u,s) \diff{u},
	\qquad 
	\text{a.e.\ in } J\times J, 
\end{equation}
the space $\left(\cV(L^\infty;J), \, \star\right)$ 
forms a Banach algebra, 
see~\cite[Theorem~9.2.4]{gripenberg1990}. 
As Banach algebra norms
are submultiplicative,
it follows that $\star$
maps into 
$\cV(L^\infty;J) \subseteq \cV(J)$.

Definition~{\upshape\ref{def:volterra_kernel}} 
does not guarantee that 
a Volterra kernel of type $L^\infty$  
is pointwise defined 
in its first variable, 
rather than merely almost everywhere. 
The next two lemmas show that 
if the conditions of 
Assumption~{\upshape\ref{ass:assump_nonconv_kernels}} 
hold in an almost everywhere sense, 
then there exists a unique representative 
satisfying these bounds everywhere.

\begin{lemma}\label{lma:continuous_representative}
	Let $T \in (0,\infty)$ and 
	$K_1, K_2$ be Volterra kernels
	satisfying Assumption~{\upshape\ref{ass:assump_nonconv_kernels}}, 
	where the 
	conditions~{\upshape\ref{item:assu_noco_i},\ref{item:assu_noco_ii}}  
	instead hold
	for \emph{almost every} $(t,s) \in \triangle[0,T]$ 
	with $\eta := \eta_T \in (0,\infty)$ and 
	$\alpha := \alpha_T \in (0,1]$. 
	Then, for $j \in \{1,2\}$, the mapping  
	$[0,T] \ni t \mapsto K_j(t,\,\cdot\,) \in L^j(0,T)$
	has a unique representative in $C^\alpha([0,T];L^j(0,T))$.
\end{lemma}

\begin{proof}
	Let $A \subset \triangle[0,T]$ be the null set
	where the bounds of 
	Assumptions~{\upshape\ref{ass:assump_nonconv_kernels}\ref{item:assu_noco_i}} 
	or of~{\upshape\ref{ass:assump_nonconv_kernels}\ref{item:assu_noco_ii}} 
	fail.
	By Fubini's theorem, the sections
	$A_t := \{s \in [0,t] : (t,s) \in A\}$ 
	are Lebesgue null sets
	for almost every $t \in [0,T]$.
	Consequently, the set
	$D$ defined by 
	$D := \{t \in [0,T] : A_t \text{ is a null set}\}$
	has full  Lebesgue 
	measure and is dense in $[0,T]$.
	
	Fix $s, t \in D$ with $s < t$.
	The sets $D_t := [0,t] \setminus A_t$
	and $D_s := [0,s] \setminus A_s$
	have full Lebesgue measure 
	in $[0,t]$ and $[0,s]$, respectively.
	Their intersection $D_{t,s} := D_t \cap D_s$
	therefore has full Lebesgue 
	measure in $[0,s]$, 
	making it dense in $[0,s]$.
	
	For any $u \in D_{t,s}$, we have 
	$u \notin A_t \cup A_s$.
	Using the 
	conditions~{\upshape\ref{item:assu_noco_i},\ref{item:assu_noco_ii}}   
	of Assumption~{\upshape\ref{ass:assump_nonconv_kernels}} 
	and the properties of Volterra kernels,
	we obtain that, for $j\in\{1,2\}$, 
	\begin{align*}
		\| K_j(t,\,\cdot\,) - K_j(u,\,\cdot\,) \|_{L^j(0,T)}
		&\leq 
		\| K_j(t,\,\cdot\,) - K_j(u,\,\cdot\,) \|_{L^j(0,u)} + \| K_j(t,\,\cdot\,) \|_{L^j(u,t)} 
		\\
		&\leq 
		\eta (t - u)^\alpha + \eta (t - u)^\alpha
		= 
		2 \eta (t - u)^\alpha, 
	\end{align*}
	and similarly for 
	$\| K_j(s,\,\cdot\,) - K_j(u,\,\cdot\,) \|_{L^j(0,T)}$. 
	Thus, by the triangle inequality, 
	\begin{equation*}
		\| K_j(t,\,\cdot\,) - K_j(s,\,\cdot\,) \|_{L^j(0,T)} 
		\leq 
		2 \eta (t - u)^\alpha + 2 \eta (s - u)^\alpha.
	\end{equation*}
	Since $D_{t,s}$ is dense in $[0,s]$,
	we can find a sequence 
	$(u_n)_{n \in \bbN} \subset D_{t,s}$
	such that $u_n \uparrow s$.
	Evaluating the right-hand side 
	at $u=u_n$ and taking the limit 
	as $n \to \infty$ yields
	\begin{equation*}
		\forall s,t \in D :
		\quad 
		\| K_j(t,\,\cdot\,) - K_j(s,\,\cdot\,) \|_{L^j(0,T)}
		\leq 
		2 \eta (t - s)^\alpha, 
	\end{equation*}
	i.e., 
	the map  
	$t \mapsto K_j(t,\,\cdot\,)$
	is uniformly $\alpha$-H\"older continuous
	on the dense subset~$D$.
	Since $L^j(0,T)$ is a Banach space,
	it uniquely extends to a continuous function  
	on the entire interval, yielding the desired
	representative in $C^\alpha([0,T]; L^j(0,T))$.
\end{proof}

\begin{lemma}\label{lma:u_continuity_lemma}
	Let $T \in (0,\infty)$,
	let $q \in [1,\infty)$,
	and let $U \in C^0([0,T];L^q(0,T))$.
	Then, the following statements hold:
	\begin{enumerate}[leftmargin=8mm, label={\upshape(\roman*)}]
		\item\label{item:u_continuity_lemma_i}
		The mapping
		$(t,s) \mapsto f(t,s) := \| U(t) \|_{L^q(s,t)}$
		is continuous on $\triangle [0,T]$.
		\item\label{item:u_continuity_lemma_ii}
		The mapping
		$(t,s) \mapsto  g(t,s) :=\| U(t) - U(s) \|_{L^q(0,s)}$
		is continuous on $\triangle [0,T]$.
	\end{enumerate}
\end{lemma}

\begin{proof}
	Let $( (t_n, s_n) )_{n \in \bbN}$ be
	a sequence in $\triangle [0,T]$
	converging to $(t,s) \in \triangle[0,T]$.
	
	For~{\upshape\ref{item:u_continuity_lemma_i}},
	we find that
	\begin{align*}
		\bigl|   f (t_n,s_n) -  f&(t,s) \bigr|
		\leq
		\bigl|  \| U(t_n) \|_{L^q(s_n,t_n)} - \| U(t) \|_{L^q(s_n,t_n)} \bigr| 
		\\
		& \hspace*{35mm}
		+
		\bigl| \| U(t) \|_{L^q(s_n,t_n)} - \| U(t) \|_{L^q(s,t)}  \bigr| 
		\\
		& \leq
		\|  U(t_n) - U(t) \|_{L^q(0,T)}
		+
		\|  U(t) \bigl(\ind{(s_n, t_n)} - \ind{(s,t)} \bigr) \|_{L^q(0,T)}, 
	\end{align*}
	where the last step holds by the reverse triangle inequality.
	In the limit $n\to\infty$,
	the first term  
	of this bound is zero 
	by continuity of the map
	${[0,T] \ni t \mapsto U(t) \in L^q(0,T)}$,
	and the second term vanishes
	by the dominated convergence theorem. 
	
	For~{\upshape\ref{item:u_continuity_lemma_ii}},
	we proceed similarly,
	\begin{equation}\label{eq:u_ct_proof_eq2}
		\begin{split}
			\bigl|   g(t_n,s_n) -  g(t,s) \bigr|
			&\leq 
			\bigl|  \| U(t_n) - U(s_n) \|_{L^q(0,s_n)} 
			- \| U(t) - U(s) \|_{L^q(0,s_n)}  \bigr|
			\\
			&\quad
			+ \bigl| \| U(t) - U(s) \|_{L^q(0,s_n)} - \| U(t) - U(s) \|_{L^q(0,s)} \bigr| .
		\end{split}
	\end{equation}
	Using the (reverse) triangle inequality, 
	the first term of \eqref{eq:u_ct_proof_eq2} 
	can be bounded by  
	\begin{equation*} 
		\| U(t_n) - U(t) \|_{L^q(0,T)}
		+
		\| U(s) - U(s_n) \|_{L^q(0,T)},
	\end{equation*}
	which converges to zero as $n\to\infty$, 
	since $U \in C^0([0,T];L^q(0,T))$.
	For the second\pagebreak 
	  
	\noindent term 
	of \eqref{eq:u_ct_proof_eq2}, 
	we obtain, 
	again by the (reverse) triangle inequality, 
	the upper bound  
	$\| U(t)  \|_{L^q(s_n \wedge s, s_n \vee s)} 
	+ 
	\| U(s) \|_{L^q(s_n \wedge s, s_n \vee s)}$ 
	which, 
	in the limit $n\to\infty$,
	converges to zero
	by absolute continuity of the
	Lebesgue integral, 
	completing the proof 
	of~{\upshape\ref{item:u_continuity_lemma_ii}}. 
\end{proof}

The next lemma 
facilitates to bound increments 
of non-negative, non-increasing convolutions 
kernels in $L^q$-norms. 
Subsequently, 
in Lemma~{\upshape\ref{lma:kernel_assumption_implications}},  
this is  
used to show that  
the uniform properties of convolution kernels
in Assumption~{\upshape\ref{ass:assu_co_unif_xi}} 
imply that the associated non-convolution kernels 
satisfy the conditions of  
Assumption~{\upshape\ref{ass:assu_noco_unif_xi}}.

\begin{lemma}\label{lma:Lq_norm_of_increment_conv}
	Let $T\in(0,\infty)$, 
	$q \in [1,\infty)$, and let 
	$f\in L^q(0,T)$ 
	be non-negative 
	and non-increasing almost everywhere.
	Then, for all $(t,s) \in \triangle [0,T]$,
	\begin{equation*}
		\| f(t-\,\cdot\,) - f(s-\,\cdot\,) \|_{L^q(0,s)} 
		\leq 
		\| f \|_{L^q(0,t-s)}.
	\end{equation*}
\end{lemma} 

\begin{proof}
	Since $f$ is non-negative and non-increasing almost everywhere,
	for any fixed 
	$(t,s)\in\triangle[0,T]$, 
	the inequality $f(s-u) \geq f(t-u) \geq 0$
	holds for almost all ${u \in [0,s]}$. 
	Superadditivity of the function 
	$x\mapsto x^q$ for $q\in[1,\infty)$ 
	implies that the inequality 
	$(b-a)^q \leq b^q - a^q$ 
	holds for all $0 \leq a \leq b$. 
	Thus, 
	we obtain, for almost all $u \in [0,s]$,
	\begin{equation*}
		|f(t-u) - f(s-u)|^q 
		= 
		( f(s-u) - f(t-u) )^q 
		\leq  
		f(s-u)^q - f(t-u)^q.
	\end{equation*}
	Integrating over $[0,s]$
	and applying changes of variables yields
	\begin{equation*}
		\| f(t-\,\cdot\,) - f(s-\,\cdot\,) \|_{L^q(0,s)}^q   
		\leq 
		\int_0^s f(u)^q \diff{u} 
		- 
		\int_{t-s}^t f(u)^q \diff{u}. 
	\end{equation*}
	By adding and subtracting
	the term $\int_0^t f(u)^q \diff{u}$,
	we find that the right-hand side equals
	$\int_0^{t-s} f(u)^q \diff{u} - \int_s^t f(u)^q \diff{u}$.
	Omitting the non-positive second term
	and taking the $q$-th root
	completes the proof of the claimed inequality. 
\end{proof}

\begin{lemma}\label{lma:kernel_assumption_implications}
	Let $\Xi$ be a non-empty set.
	Let $(k^\xi_1)_{\xi\in\Xi}$ and $(k^\xi_2)_{\xi\in\Xi}$
	be families of Borel measurable real-valued functions 
	on $(0,\infty)$. 
	Let $(K^\xi_1)_{\xi\in\Xi}$ and $(K^\xi_2)_{\xi\in\Xi}$
	be the associated families of 
	Volterra kernels on $[0,\infty)$
	defined, for $j \in \{1,2\}$, by
	\begin{equation*}
		\forall \xi\in\Xi:
		\quad\; 
		K^\xi_j(t,s) = 
		\begin{cases}
			k^\xi_j(t-s) & t > s,   \\
			0            & t \leq s,
		\end{cases}
		\quad\;  
		(t,s)\in 
		[0,\infty) \times [0,\infty).  
	\end{equation*}  
	
	Whenever $(k^\xi_1,k^\xi_2)_{\xi \in \Xi}$
	satisfies Assumption~{\upshape\ref{ass:assu_co_unif_xi}} 
	with the families of constants  
	${(\eta_T)_{T > 0} \subseteq (0,\infty)}$
	and $(\alpha_T)_{T>0} \subseteq (0,1]$,
	then the family
	$(K^\xi_1, K^\xi_2)_{\xi\in\Xi}$
	satisfies Assumption~{\upshape\ref{ass:assu_noco_unif_xi}} 
	with the same families of constants $(\eta_T)_{T > 0}$
	and $(\alpha_T)_{T>0}$.
\end{lemma}

\begin{proof} 
	By Assumption~{\upshape\ref{ass:assu_co_unif_xi}}, 
	there exist non-increasing 
	Borel measurable functions 
	$\overline{k}_1, \overline{k}_2 \colon (0,\infty) \to[0,\infty)$ 
	such that, 
	for all $\xi\in\Xi$, 
	the functions $k^\xi_1$, $k^\xi_2$ 
	are a.e.\ bounded 
	from above by   
	$\overline{k}_1$ and~$\overline{k}_2$, 
	respectively.   
	We define the corresponding Volterra kernels as 
	$\overline{K}_j(t,s) :=  
	\overline{k}_j(t-s) \ind{\{t > s\}}(t,s)$, 
	for ${j \in \{1,2\}}$ and all $(t,s)\in[0,\infty)\times[0,\infty)$. 
	Let $T\in(0,\infty)$ 
	and $\xi\in\Xi$ be arbitrary. 
	Then, 
	since the functions 
	$k^\xi_1$ and $k^\xi_2$  
	are non-negative a.e., 
	we find that,  
	for $j\in\{1,2\}$,  
	and for almost all $(t,s)\in\triangle[0,T]$, 
	\begin{equation*}
		|K^\xi_j(t,s)|
		= 
		k^\xi_j(t-s)
		\leq 
		\overline{k}_j(t-s) 
		= 
		\overline{K}_j(t,s). 
	\end{equation*}
	Furthermore, for $j\in\{1,2\}$ 
	and all $(t,s) \in \triangle[0,T]$,
	we find  
	by Assumption~{\upshape\ref{ass:assu_co_unif_xi}} 
	that 
	\begin{align*}
		\| \overline{K}_j(t,\,\cdot\,) \|_{L^j(s,t)}
		&= 
		\| \overline{k}_j \|_{L^j(0,t-s)}
		\leq 
		\eta_T (t-s)^{\alpha_T}, 
		\\
		\| \overline{K}_j(t,\,\cdot\,) -  \overline{K}_j(s,\,\cdot\,) \|_{L^j(0,s)}
		&\leq 
		\| \overline{k}_j \|_{L^j(0,t-s)} 
		\leq
		\eta_T (t-s)^{\alpha_T},
		\\
		\| K^\xi_j(t,\,\cdot\,) - K^\xi_j(s,\,\cdot\,) \|_{L^j(0,s)}
		&\leq 
		\| k^\xi_j \|_{L^j(0,t-s)}
		\leq 
		\| \overline{k}_j \|_{L^j(0,t-s)}
		\leq
		\eta_T (t-s)^{\alpha_T},
	\end{align*}
	where we also applied    
	Lemma~{\upshape\ref{lma:Lq_norm_of_increment_conv}} 
	for the functions 
	$\overline{k}_j$ and $k^\xi_j$.  
	The preceding estimates show that 
	${\overline{K}_j 
	\in C^0([0,T]; L^j(0,T))}$ 
	holds for $j\in\{1,2\}$, 
	and that 
	the family $(K_1^\xi, K_2^\xi)_{\xi\in\Xi}$ 
	satisfies Assumption~{\upshape\ref{ass:assu_noco_unif_xi}} 
	with the constants $(\eta_T)_{T > 0}$
	and $(\alpha_T)_{T>0}$.
\end{proof}


\subsection{Resolvent equations and Gr\"onwall's inequality} 
\label{subappendix:gronwall} 

In this subsection, 
we use the product operation~\eqref{eq:def_star}
to introduce the resolvent $r$ 
of a Volterra kernel $\kappa$ 
and a series representation of~$r$, 
needed for an inequality of 
Gr\"onwall type. 

\begin{definition}\label{def:resolvent_eqs} 
	Let $\kappa \in \cV(L^\infty;J)$. 
	A Volterra kernel $r$ of type $L^\infty$ on $J$ 
	is called a (Volterra) \emph{resolvent}  
	(of type $L^\infty$) of $\kappa$ on $J$ if  
	(cf.\ \cite[Definition~9.3.1]{gripenberg1990})
	\begin{equation}\label{eq:resolvent_definition}
		r-\kappa 
		= 
		r \star \kappa 
		= \kappa \star r,
		\quad \text{a.e.\ in }\triangle J.
	\end{equation}
\end{definition}

\begin{remark}\label{rmk:rescale_kernel}
	Note that Definition~{\upshape\ref{def:resolvent_eqs}} 
	differs from
	the convention in \cite{gripenberg1990} by a sign. 
	Applying the results of 
	\cite[Chapter~9]{gripenberg1990} requires
	replacing $(\kappa, r)$ with $(-\kappa, -r)$; 
	thus, a 
	non-positive resolvent of $-\kappa$  
	in \cite{gripenberg1990} 
	corresponds to $r$ in \eqref{eq:resolvent_definition} 
	being non-negative. 
\end{remark} 

Every Volterra kernel $\kappa\in\cV(L^\infty;J)$ 
can have at most one resolvent 
$r\in\cV(L^\infty;J)$, see~\cite[Lemma~9.3.3]{gripenberg1990}. 
In contrast to uniqueness, 
existence of a resolvent  
is not guaranteed.   
For the Volterra kernels 
considered in this work, 
we address this below.

\begin{proposition}\label{prop:k_integrable_resolvent}
	Let $T\in(0,\infty)$, 
	and let $\kappa \in \cV(L^\infty;[0,T])$.
	Assume that 
	there exist constants $c \in (0,\infty)$
	and $\alpha \in (0,1]$ such that, 
	for almost all ${t\in[0,T]}$,   
	\begin{equation}\label{eq:kappa_assumption}  
		\forall s\in[0,t]: 
		\quad 
		\int_s^t | \kappa(t,u) |\diff{u} 
		\leq 
		c \, (t-s)^\alpha.
	\end{equation}
	Then, $\kappa$ has a unique resolvent 
	$r \in \cV(L^\infty;[0,T])$
	which admits the following representation as 
	an absolutely convergent Neumann series 
	in $\cV(L^\infty;[0,T])$, 
	\begin{equation}\label{eq:neumann_series}
		r 
		= 
		\sum_{j=1}^\infty \kappa^{\star j} 
		= 
		\kappa + \kappa\star \kappa + \kappa\star \kappa \star \kappa + \ldots.
	\end{equation}
\end{proposition}

\begin{proof}
	For any subinterval $[a,b] \subseteq [0,T]$, 
	we obtain by~\eqref{eq:kappa_assumption} the estimate  
	\begin{equation*}
		\| \kappa \|_{\cV(L^\infty;[a,b])}
		= 
		\esssup_{t \in [a,b]} \int_a^t | \kappa(t,u) | \diff{u}
		\leq 
		c \, (b-a)^\alpha.
	\end{equation*}
	Partitioning $[0,T]$ into
	$\bigl\lceil T (2c)^{\nicefrac{1}{\alpha}\!} \bigr\rceil$
	intervals of length at most 
	$\delta := (2c)^{-\nicefrac{1}{\alpha}\!}$, 
	we can bound the $\cV(L^\infty;J)$-norm of $\kappa$
	on each subinterval $J$ by
	$c \, \delta^\alpha = \tfrac{1}{2}< 1$.
	Thus, $\kappa$ has a unique resolvent
	$r \in \cV(L^\infty;[0,T])$ by
	\cite[Corollary~9.3.14 and Lemma~9.3.3]{gripenberg1990}. 
	
	By \cite[Corollary~9.3.18(i)]{gripenberg1990},
	the spectral radius of $\kappa$ satisfies $R(\kappa) < 1$, 
	and 
	by \cite[Theorem~1.2.7]{murphy1990} 
	we have that 
	$R(\kappa)  
	= 
	\lim_{n \to \infty} \| \kappa^{\star n} \|_{\cV(L^\infty;[0,T])}^{\nicefrac{1}{n}}$. 
	Choosing an arbitrary $\rho \in (R(\kappa), 1)$, 
	there exists an integer $N \in \bbN$
	such that 
	${\| \kappa^{\star n} \|_{\cV(L^\infty;[0,T])} \le \rho^n}$ 
	holds for all $n \geq N$. 
	Thus, the Neumann series 
	on the right-hand side of \eqref{eq:neumann_series} 
	is absolutely convergent in $\cV(L^\infty;[0,T])$. 
	To prove that it coincides with the resolvent~$r$ of~$\kappa$, 
	we note that 
	$\sum_{j=1}^\infty \kappa^{\star j}  
	- 
	\kappa 
	= 
	\sum_{j=1}^\infty \kappa^{\star (j+1)} 
	= 
	\bigl(\sum_{j=1}^\infty \kappa^{\star j} \bigr) 
	\star \kappa 
	= 
	\kappa \star \bigl( \sum_{j=1}^\infty \kappa^{\star j} \bigr)$, 
	where we have used distributivity and joint continuity 
	of the map $\star$ in \eqref{eq:def_star}, 
	see~\cite[p.~2]{murphy1990}, 
	to move it inside and outside 
	the series. 
	In summary, 
	the series $\sum_{j=1}^\infty \kappa^{\star j}$ 
	is absolutely convergent in $\cV(L^\infty;[0,T])$,  
	and it satisfies \eqref{eq:resolvent_definition}. 
	By uniqueness of limits and of the resolvent, 
	the series converges to~$r$, completing 
	the proof of~\eqref{eq:neumann_series}. 
\end{proof}

\begin{remark}\label{rmk:nonnegative_k_r}
	If $\kappa$ is non-negative 
	a.e.~in $[0,T]\times[0,T]$,
	the Neumann series \eqref{eq:neumann_series}
	consists of only non-negative terms 
	and the resolvent $r$ of $\kappa$ 
	is also non-negative~a.e. 
\end{remark}

We next state an inequality 
of Gr\"onwall type tailored 
for our results in Section~{\upshape\ref{sec:general_SVE}}.

\begin{lemma}[Gr\"onwall's inequality]
	\label{lem:gronwall_ineq_1var} 
	Let $T \in (0, \infty)$, and let 
	$\kappa \in \cV(L^\infty; [0,T])$ 
	be non-negative a.e.\ in $[0,T]\times [0,T]$. 
	Suppose that there exist constants
	$c \in (0,\infty)$ and $\alpha \in (0,1]$ such that 
	\eqref{eq:kappa_assumption} holds 
	for almost all $t\in[0,T]$. 
	Let $f, x \colon [0,T] \to \bbR$ be  
	continuous functions such that, 
	for almost all $t\in[0,T]$ 
	\begin{equation*}  
		x(t) 
		\leq 
		f(t) + \int_0^t \kappa(t,u) x(u) \diff{u}.
	\end{equation*}
	Then, there exists a constant $C_{\kappa,T} \in (0,\infty)$ such that 
	\begin{equation*}
		\forall t \in [0,T]: 
		\quad 
		x(t) 
		\leq 
		C_{\kappa,T} \sup_{v \in [0,t]} |f(v)|.
	\end{equation*}
\end{lemma}

\begin{proof}
	By Proposition~{\upshape\ref{prop:k_integrable_resolvent}} 
	the kernel $\kappa$
	has a resolvent
	$r \in \cV(L^\infty; [0,T])$, 
	which is non-negative a.e., 
	see Remark~{\upshape\ref{rmk:nonnegative_k_r}}. 
	By \cite[Theorem~9.3.6 \& Lemma~9.8.2]{gripenberg1990}, 
	recalling the difference in the sign convention 
	from Remark~{\upshape\ref{rmk:rescale_kernel}}, 
	and by H\"older's inequality 
	we find that, 
	for almost all $t\in[0,T)$, 
	\begin{align*} 
		x(t) 
		&\leq 
		f(t) + \int_0^t r(t,u) f(u) \diff{u} 
		\\
		&\leq 
		|f(t)| + 
		\| r(t,\,\cdot\,) \|_{L^1(0,t)} \| f \|_{L^\infty(0,t)} 
		\leq 
		\bigl(1  + \| r \|_{\cV(L^\infty; [0,T])} \bigr) 
		\| f \|_{L^\infty(0,t)} . 
	\end{align*}  
	By continuity of $x$ and $f$, 
	it follows that
	$x(t) 
	\leq 
	\bigl(1  + \| r \|_{\cV(L^\infty; [0,T])} \bigr) \sup_{v \in [0,t]} |f(v)|$ 
	holds   
	for all $t\in[0,T]$.  
	As $r$ is uniquely determined by $\kappa$,
	the value of $\| r \|_{\cV(L^\infty; [0,T])}$ 
	depends only on $\kappa$ and~$T$,
	and the claim follows 
	with 
	$C_{\kappa,T} := 1 + \| r \|_{\cV(L^\infty; [0,T])}$. 
\end{proof}


\section{Auxiliary inequalities}

In the next lemma, 
we formulate an inequality of 
Burkholder--Davis--Gundy type tailored 
for stochastic Volterra processes. 
Finally, 
in Lemma~{\upshape\ref{lma:exponential_maximum_bound}}, 
we state an auxiliary result, 
frequently used 
in Sections~{\upshape\ref{sec:sinc-fractional}} 
and~{\upshape\ref{sec:markov_approx}} to bound 
the integrand in the integral 
representation \eqref{eq:fracker_as_laplace} 
of the fractional kernel. 
We recall that  
$(\Omega, \cF, (\cF_t)_{t \geq 0}, \bbP)$ is
a complete filtered probability space
with a right-continuous filtration.

\begin{lemma}[Burkholder--Davis--Gundy 
	inequality]\label{lma:BDG_SVE} 
	Let $T \in (0,\infty)$,
	and assume that  
	$K \colon [0,T] \times [0,T] \to \bbR$
	is a Borel measurable function
	for which the mapping
	$t \mapsto K(t,\,\cdot\,)$ is in $C^0([0,T];L^2(0,T))$.
	Let $W\colon [0,\infty) \times \Omega \to \bbR^d$
	be a $d$-dimensional standard Brownian motion,
	and let $Z\colon [0,\infty) \times \Omega \to \bbR^{m\times d}$ be
	a progressively measurable stochastic process, 
	i.e., for all $v \in [0,\infty)$, 
	the restriction of $Z$ to $[0,v] \times \Omega$ 
	is $\cB([0,v]) \otimes \cF_v$-measurable,
	where $\cB([0,v])$ denotes 
	the Borel $\sigma$-algebra on $[0,v]$.
	
	Let $p \in [2,\infty)$ 
	and assume that
	$\esssup_{t \in [0,T]} \| Z_t \|_{L^p(\Omega; \bbR^{m\times d})} < \infty$. 
	Then, 
	there exists $\bdg{p}\in(0,\infty)$, 
	depending only on $p$,  
	such that, 
	for all ${(t,s) \in \triangle[0,T]}$, 
	\begin{equation*}  
		\biggl\| \int_s^t K(t,u) Z_u \diff{W}_u  
		\biggr\|_{L^p(\Omega)}
		\leq 
		\sqrt{\bdg{p}} \, 
		\biggl( \int_s^t
		|K(t,u)|^2 \| Z_u \|^2_{L^p(\Omega)} \diff{u} 
		\biggr)^{\nicefrac{1}{2}}.
	\end{equation*}
\end{lemma}

\begin{proof}
	For $(t,s) \in \triangle[0,T]$,  define
	the stochastic process
	$\Psi^{(t,s)} \colon [0,\infty)\times\Omega \to \bbR^{m\times d}$ 
	by 
	$\Psi^{(t,s)}_u := \ind{[s,t]}(u) K(t,u) Z_u$. 
	The process 
	$\Psi^{(t,s)}$
	is then a product of
	a deterministic Borel measurable function
	and a progressively measurable process,
	and is therefore itself 
	progressively measurable.
	Moreover, 
	by the assumptions on~$K$ and on~$Z$, 
	for every $(t,s) \in \triangle[0,T]$, 
	the $L^2(0,\infty; L^p(\Omega))$-norm 
	of $\Psi^{(t,s)}$ is finite, 
	\begin{equation*}
		\bigl\| \Psi^{(t,s)} 
		\bigr\|_{L^2(0,\infty;L^p(\Omega))}^2 
		\leq 
		\sup_{v \in [0,T]} \| K(v,\,\cdot\,) \|_{L^2(0,T)}^2 \,
		\esssup_{v \in [0,T]}\| Z_v \|_{L^p(\Omega)}^2  
		< \infty.
	\end{equation*}
	Therefore, we can apply 
	\cite[Theorem~1.1]{seidler2010}
	and Minkowski's integral inequality
	\cite[Proposition~1.2.22]{analysis_banach_I}, 
	and obtain that, 
	for all $(t,s)\in\triangle[0,T]$, 
	\begin{equation*}
		\bbE \sup_{v \in [0,T]}
		\biggl\| \int_0^v \Psi^{(t,s)}_u \diff{W}_u 
		\biggr\|_{\bbR^m}^p
		\leq 
		\bdg{p}^{\nicefrac{p}{2}} 
		\bigl\| \Psi^{(t,s)} \bigr\|_{L^p(\Omega;L^2(0,T))}^p
		\leq 
		\bdg{p}^{\nicefrac{p}{2}} 
		\bigl\| \Psi^{(t,s)} 
		\bigr\|_{L^2(0,T;L^p(\Omega))}^p, 
	\end{equation*} 
	where $\bdg{p}$ is a positive constant, 
	which depends only on $p$ 
	and satisfies $\bdg{p} = \cO(p)$. 
	
	By construction, 
	$\int_s^t K(t,u) Z_u \diff{W}_u 
	= \int_0^T \Psi^{(t,s)}_u \diff{W}_u$
	holds $\bbP$-almost surely. 
	Thus,
	\begin{align*}
		\biggl\| \int_s^t K(t,u) &Z_u \diff{W}_u \biggr\|_{L^p(\Omega)} 
		\leq 
		\biggl( \bbE \sup_{v \in [0,T]} 
		\biggl\| \int_0^v \Psi^{(t,s)}_u \diff{W}_u \biggr\|_{\bbR^m}^p \biggr)^{\nicefrac{1}{p}}
		\\ 
		&\leq 
		\sqrt{\bdg{p}} \, 
		\bigl\| \Psi^{(t,s)} \bigr\|_{L^2(0,T;L^p(\Omega))}  
		=  
		\sqrt{\bdg{p}} \, 
		\biggl( 
		\int_s^t |K(t,u)|^2 \| Z_u \|_{L^p(\Omega)}^2 \diff{u} 
		\biggr)^{\nicefrac{1}{2}} 
	\end{align*} 
	holds for all $(t,s)\in\triangle[0,T]$, 
	as claimed.
\end{proof}

\begin{lemma}\label{lma:exponential_maximum_bound} 
	Let $s, \eps \in (0,\infty)$, 
	and define the function $f \colon [0,\infty) \to [0,\infty)$
	by $f(x) = x^\eps e^{-sx}$.
	Then, $f$ has a unique maximum at $x=\tfrac{\eps}{s}$, and 
	\begin{equation*}
		\max_{x \in [0,\infty)} f(x) = \Bigl(\frac{\eps}{e}\Bigr)^\eps \, s^{-\eps}.
	\end{equation*}
\end{lemma}

\begin{proof}
	The derivative 
	$f'(x) = x^{\eps-1}e^{-sx}(\eps - sx)$ 
	vanishes only at $x = \frac{\eps}{s}$
	on $(0,\infty)$.
	Since $f(0) = 0$, $f(x) > 0$ for all $x > 0$,
	and $\lim_{x \to \infty} f(x) = 0$,
	this critical point is the unique global maximum, 
	and ${\max_{x\in[0,\infty)} f(x) 
		= f\bigl(\frac{\eps}{s} \bigr) 
		= \bigl(\frac{\eps}{e}\bigr)^\eps s^{-\eps}}$.
\end{proof}

\section*{Acknowledgment} 

The authors thank Antonis Papapantoleon 
for helpful discussions about  
rough stochastic volatility models 
and Markovian approximations. 

\section*{Funding} 

The work of KK was funded by 
the European Union 
(ERC, FunCalc4Stats, 
grant agreement no.~101219883). 
Views and opinions expressed are however 
those of the author(s) only and do not necessarily reflect 
those of the European Union 
or the European Research Council.  
Neither the European Union nor the granting authority 
can be held responsible for them.

\bibliographystyle{siam}
\bibliography{references.bib}  

\end{document}